%% file: polycox.tex
\documentclass[a4paper,twoside,11pt,english,intlimits]{article}

\usepackage{amsmath,amsthm}
\usepackage{babel}
\usepackage[T1]{fontenc}
\usepackage{times,euler}
\usepackage{fancyhdr,titlesec,url}

\usepackage{amsmath,amsfonts,amssymb}
\usepackage{euscript,stmaryrd}
\usepackage[all]{xy}
\usepackage{tikz}
\usepackage{ifthen}
\usepackage{hyperref}
\usepackage{enumerate}

\CompileMatrices
\usetikzlibrary{calc}
\DeclareMathAlphabet{\mathscr}{T1}{pzc}{m}{it} 

\input{macros}

\begin{document}

\thispagestyle{empty}

\begin{center}
\begin{Large}\begin{uppercase}
{Coherent presentations of Artin monoids}
\end{uppercase}\end{Large}

\bigskip\hrule height 1.5pt \bigskip

\begin{large}\begin{uppercase}
{Stéphane Gaussent \qquad Yves Guiraud \qquad Philippe Malbos}
\end{uppercase}\end{large}

\vspace*{\stretch{1}}

\begin{small}
\begin{minipage}{12cm}
\noindent\textbf{Abstract --}
We compute coherent presentations of Artin monoids, that is presentations by generators, relations, and relations between the relations. For that, we use methods of higher-dimensional rewriting that extend Squier's and Knuth-Bendix's completions into a homotopical completion-reduction, applied to Artin's and Garside's presentations. The main result of the paper states that the so-called Tits-Zamolodchikov $3$-cells extend Artin's presentation into a coherent presentation. As a byproduct, we give a new constructive proof of a theorem of Deligne on the actions of an Artin monoid on a category.

\bigskip\noindent\textbf{M.S.C. 2000 --} 20F36, 18D05, 68Q42. 
\end{minipage}
\end{small}

\vspace*{\stretch{1}}

\begin{small}\begin{minipage}{12cm}
\renewcommand{\contentsname}{}
\setcounter{tocdepth}{2}

\tableofcontents
\end{minipage}
\end{small}

\end{center}

\newpage

\section*{Introduction}

A Coxeter system $(\W, S)$ is a group $\W$ together with a presentation by a finite set of involutions~$S$ satisfying some (generalised) braid relations that we recall in Section~\ref{Section:GarsideCoherentPresentation}. 
Forgetting about the involutive character of the generators and keeping only the braid relations, one gets Artin's presentation of the Artin monoid $\B^+(\W)$. For example, if $\W = \S_4$, the group of permutations of $\{1,2,3, 4\}$, then $S$ consists of the elementary transpositions $r=(1\ 2)$, $s=(2\ 3)$ and $t=(3\ 4)$, and the associated Artin monoid is the monoid $\B^+_4$ of positive braids on four strands, with generators $r,s,t$ satisfying the relations
\[
rsr = srs,\;\; \ rt=tr \;\; \text{and} \;\; sts = tst. 
\] 
The aim of this article is to push further Artin's presentation and study the relations between the braid relations. A coherent presentation of a monoid (or more generally of a category) consists of a set of generators, a set of generating relations and some coherence conditions. These coherence conditions can be thought of as elements of a homotopy basis of a $2$-dimensional CW-complex associated to the presentation. In the case of the braid monoid $\B^+_4$ on $4$ strands, Deligne~\cite{Deligne97} notes that the homotopy basis associated to Artin's presentation contains only one element whose boundary consists of the reduced expressions graph of the element of maximal length in $\S_4$ (this graph can be seen in Subsection~\ref{Subsection:Zamolodchikov}). 

Such a graph can be considered for any element $w$ in $\W$. The vertices are the reduced expressions of $w$ and two such are linked by an edge if one is obtained from the other by a braid relation.
In~\cite{Tits81}, Tits proves that the fundamental group of the reduced expressions graph is generated by two types of loops in the graph, the most interesting ones are associated to finite parabolic subgroups of rank $3$ of~$\W$. Actually, for the purpose of finding generators for the homotopy basis of $\B^+(\W)$ associated to Artin's presentation, the generators of the first type are degenerate and part of the generators of the second type are superfluous. The main result of our paper, Theorem~\ref{Theorem:ArtinCoherentPresentation}, states that there exists exactly one nondegenerate generator of the homotopy basis for every finite parabolic subgroup of rank $3$ of $\W$.

We now give some more details on the techniques we are using.
The notion of coherent presentation is formalised in terms of \emph{polygraphs}, which are presentations of higher-dimensional categories introduced by Burroni in~\cite{Burroni93}, and by Street in~\cite{Street76} under the name of \emph{computad}. 
A \emph{$2$-polygraph} corresponds to a presentation of a monoid by a rewriting system, that is a presentation by generators ($1$-cells) and oriented relations ($2$-cells). For example, Artin's presentation of $\B^+_4$ has three generating $1$-cells $r,s,t$ and three generating $2$-cells
\[
rsr \Rightarrow srs,\;\; \ rt\Rightarrow tr \;\;\text{and} \;\; sts \Rightarrow tst.
\] 
In \cite{GuiraudMalbos12}, the last two authors have introduced the notion of \emph{$(3,1)$-polygraph} as a presentation extended by $3$-cells on the $2$-category defined by the congruence generated by the presentation. A coherent presentation is then a $(3,1)$-polygraph such that the extension is a homotopy basis. We recall all these notions in Section~\ref{Section:CoherentPresentations}.

To obtain coherent presentations for monoids, in Section~\ref{Section:CompletionReduction}, we develop a homotopical completion-reduction method that is based on Squier's and Knuth-Bendix's completions. 
The completion-reduction is given in terms of Tietze transformations, known for presentations of groups~\cite{Tietze08,LyndonSchupp77}, here defined for $(3,1)$-polygraphs.
More precisely, we extend Squier's completion to terminating $2$-polygraphs thanks to Knuth-Bendix's completion~\cite{KnuthBendix70}. This is a classical construction of rewriting theory, similar to Buchberger's algorithm for computing Gröbner bases~\cite{Buchberger65}. The procedure transforms a terminating $2$-polygraph $\Sigma$ into a convergent one by adding to $\Sigma$ a potentially infinite number of $2$-cells so that every critical branching is confluent. Confluence of a $2$-polygraph means that every time two $2$-cells share the same source but two different targets, there exist two $2$-cells having those different $1$-cells as source and the same target. 
So, we end up with a $(3,1)$-polygraph $\Sr(\Sigma)$ where every critical branching has given a $3$-cell in the homotopy basis. Since the $2$-polygraph we started with presents the monoid, $\Sr(\Sigma)$ is a coherent presentation of this monoid. 
Next, we introduce \emph{homotopical reduction} as a general construction to coherently eliminate unnecessary cells in a coherent presentation. The $(3,1)$-polygraph $\Sr(\Sigma)$ has usually more cells than one could expect. For example, one can eliminate the pairs of redundant $2$-cells and collapsible $3$-cells adjoined by homotopical completion for nonconfluent critical branchings. Some of the remaining $3$-cells may also be redundant: one way to detect them is to compute the $3$-spheres associated to the triple critical branchings of the presentation. Let us mention that the two last authors and Mimram have applied those methods to compute coherent presentations of plactic and Chinese monoids in~\cite{GuiraudMalbosMimram13}.

In Section~\ref{Section:GarsideCoherentPresentation}, we use the homotopical completion-reduction method to get a coherent presentation $\Gar_3(\W)$ of the Artin monoid $\B^+(\W)$. The starting presentation is Garside's presentation, denoted by $\Gar_2(\W)$. It has the elements of $\W\setminus\ens{1}$ as generators and the relations are
\[
u|v \:=\: uv \qquad\quad\text{if \;\; $l(uv)=l(u)+l(v)$.}
\]
The notation $\cdot|\cdot$ stands for the product in the free monoid over $\W\setminus\ens{1}$ and~$l(u)$ is the length of~$u$ in~$\W$.
The resulting coherent presentation $\Gar_3(\W)$, that we obtain in Theorem~\ref{Theorem:GarsideCoherentPresentation}, corresponds to the coherence data given by Deligne in~\cite[Theorem~1.5]{Deligne97}.
We generalise our construction to Garside monoids, so that we are able to associate to every Garside monoid $\M$ a coherent presentation $\Gar_3(\M)$ (see Theorem~\ref{Theorem:GarsideGarside}).

In Section~\ref{Section:ArtinCoherentPresentation}, we homotopically reduce Garside's coherent presentation $\Gar_3(\W)$ into the smaller coherent presentation $\Art_3(\W)$ associated with  Artin's presentation of the monoid $\B^+(\W)$.  
The homotopy basis of $\Gar_3(\W)$ boils down to one $3$-cell $Z_{r,s,t}$ for all elements $t>s>r$ of~$S$ such that the subgroup of~$\W$ they span is finite.
To sum up, Theorem~\ref{Theorem:ArtinCoherentPresentation} says that the coherent presentation $\Art_3(\W)$ has exactly one $k$-cell, $0\leq k\leq 3$, for every subset $I$ of $S$ of rank $k$ such that the subgroup $\W_I$ is finite. The precise shape of the $3$-cells is given in~\ref{Subsection:Zamolodchikov}. 

As an application, in Theorem~\ref{Theorem:CoherentAction}, we prove that if $\Sigma$ is a coherent presentation of a monoid~$\M$, then the category~$\Act(\M)$ of actions of~$\M$ on categories is equivalent to the category of $2$-functors from the associated $(2,1)$-category~$\tck{\Sigma}$ to~$\Cat$ that send the elements of the homotopy basis to commutative diagrams. In~\cite[Theorem~1.5]{Deligne97}, Deligne already observes that this equivalence holds for Garside's presentation of spherical Artin monoids. 
The constructions are described in the homotopical setting of the canonical model structure on $2$-categories given by Lack~\cite{Lack02,Lack04}.
In this spirit, as a byproduct of our main theorem, to determine the action of an Artin monoid on a category, it suffices to attach to any generating $1$-cell $s\in S$ an endofunctor $T(s)$ and to any generating $2$-cell a natural isomorphism, such that these satisfy coherence relations given by the Tits-Zamolodchikov $3$-cells.

Finally, let us remark that, in~\cite[Theorem~4.5.3]{GuiraudMalbos12}, Squier's completion is extended in higher dimensions to produce \emph{polygraphic resolutions} of monoids, of which coherent presentations form the first three dimensions. From that point of view, the present work is a first step towards the construction of polygraphic resolutions $\Gar_*(\W)$ and $\Art_*(\W)$ of Artin monoids, extending the coherent presentations $\Gar_3(\W)$ and $\Art_3(\W)$. 
Moreover, the relationship between those resolutions and the higher categorical constructions in~\cite{ManinSchechtman89} should be explored. Further, the abelian resolutions obtained from $\Gar_*(\W)$ and $\Art_*(\W)$ by~\cite[Theorem~5.4.3]{GuiraudMalbos12} should be related to the abelian resolutions introduced in~\cite{DehornoyLafont03}. 

\subsubsection*{Acknowledgments}
The authors wish to thank Pierre-Louis Curien, Kenji Iohara, François Métayer, Samuel Mimram, Timothy Porter and the anonymous referee for fruitful exchanges and meaningful suggestions. This work has been partially supported by the project \emph{Cathre}, ANR-13-BS02-0005-02.

\section{Coherent presentations of categories}
\label{Section:CoherentPresentations}

\subsection{Higher-dimensional categories}

If $\Cr$ is an $n$-category (we always consider strict, globular $n$-categories), we denote by $\Cr_k$ the set (and the $k$-category) of $k$-cells of~$\Cr$. If $f$ is a $k$-cell of $\Cr$, then $s_i(f)$ and $t_i(f)$ respectively denote the $i$-source and $i$-target of $f$; we drop the suffix~$i$ if $i=k-1$. The source and target maps satisfy the \emph{globular relations:} 
\[
s_i\circ s_{i+1} \:=\: s_i\circ t_{i+1} 
\qquad\text{and}\qquad
t_i\circ s_{i+1} \:=\: t_i\circ t_{i+1}.
\]
We respectively denote by $f:u\fl v$, $\;f:u\dfl v$, $\;f:u\tfl v\;$ and $\;f:u\qfl v\;$ a $1$-cell, a $2$-cell, a $3$-cell and a $4$-cell $f$ with source~$u$ and target~$v$. If $f$ and $g$ are $i$-composable $k$-cells, that is if $t_i(f)=s_i(g)$, we denote by $f\star_i g$ their $i$-composite; we simply write $fg$ if $i=0$. The compositions satisfy the \emph{exchange relations} given, for every $i\neq j$ and all possible cells $f$, $g$, $h$ and $k$, by
\[
(f \star_i g) \star_j (h \star_i k) \:=\: (f \star_j h) \star_i (g \star_j k).
\]
If $f$ is a $k$-cell, we denote by $1_f$ its identity $(k+1)$-cell. If $1_f$ is composed with cells of dimension $k+1$ or higher, we simply denote it by~$f$; for example, we write $ufv$ and $ufvgw$ instead of $1_u\star_0 f \star_0 1_v$ and $1_u\star_0 f \star_0 1_v \star_0 g \star_0 1_w$ for $1$-cells $u$, $v$ and $w$ and $2$-cells $f$ and $g$.

\subsubsection{\pdf{(n,p)}-categories}

In an $n$-category~$\Cr$, a $k$-cell $f$, with source $x$ and target $y$, is \emph{invertible} if there exists a $k$-cell $f^-$ in $\Cr$, with source $y$ and target $x$ in $\Cr$, called the \emph{inverse of $f$}, such that
\[
f\star_{k-1} f^- \:=\: 1_x 
\qquad\text{and}\qquad
f^-\star_{k-1} f \:=\: 1_y.
\]
An \emph{$(n,p)$-category} is an $n$-category whose $k$-cells are invertible for every $k>p$. In particular, an $(n,n)$-category is an ordinary $n$-category and an $(n,0)$-category is an $n$-groupoid. 

\subsubsection{Spheres}

Let $\Cr$ be an $n$-category. A \emph{$0$-sphere of $\Cr$} is a pair $\gamma=(f,g)$ of $0$-cells of $\Cr$ and, for $1\leq k\leq n$, a \emph{$k$-sphere of $\Cr$} is a pair $\gamma=(f,g)$ of parallel $k$-cells of $\Cr$, \ie, with $s(f)=s(g)$ and $t(f)=t(g)$. We call $f$ the \emph{source} of $\gamma$ and~$g$ its \emph{target} and we write $s(\gamma)=f$ and $t(\gamma)=g$. If $f$ is a $k$-cell of $\Cr$, for $1\leq k\leq n$, the \emph{boundary of $f$} is the $(k-1)$-sphere $(s(f),t(f))$. 

\subsubsection{Cellular extensions}

Let $\Cr$ be an $n$-category. A \emph{cellular extension of $\Cr$} is a set $\Gamma$ equipped with a map from $\Gamma$ to the set of $n$-spheres of $\Cr$, whose value on $\gamma$ is denoted by $(s(\gamma),t(\gamma))$. 
By considering all the formal compositions of elements of~$\Gamma$, seen as $(n+1)$-cells with source and target in~$\Cr$, one builds the \emph{free $(n+1)$-category generated by~$\Gamma$ over $\Cr$}, denoted by $\Cr[\Gamma]$. 
The \emph{quotient of $\Cr$ by $\Gamma$}, denoted by~$\Cr/\Gamma$, is the $n$-category one gets from $\Cr$ by identification of the $n$-cells $s(\gamma)$ and $t(\gamma)$, for every $n$-sphere~$\gamma$ of~$\Gamma$. 
If $\Cr$ is an $(n,1)$-category and $\Gamma$ is a cellular extension of $\Cr$, then the \emph{free $(n+1,1)$-category generated by~$\Gamma$ over $\Cr$} is denoted by~$\Cr(\Gamma)$ and defined as follows: 
\[
\Cr(\Gamma) \:=\: \Cr [ \Gamma,\, \check{\Gamma} ] \big/ \; \Inv(\Gamma) 
\]
where $\check{\Gamma}$ contains the same $(n+1)$-cells as $\Gamma$, with source and target reversed, and $\Inv(\Gamma)$ is the cellular extension of $\Cr[\Gamma,\check{\Gamma}]$ made of two $(n+2)$-cells 
\[
\check{x}\star_{n}x \:\ofl{\lambda_{x}}\: 1_{t(x)}
\qquad\text{and}\qquad
x\star_{n}\check{x} \:\ofl{\rho_{x}}\: 1_{s(x)}
\]
for each $(n+1)$-cell $x$ of $\Gamma$.

\subsubsection{Homotopy bases}

Let $\Cr$ be an $n$-category. A \emph{homotopy basis of $\Cr$} is a cellular extension~$\Gamma$ of~$\Cr$ such that, for every $n$-sphere $\gamma$ of $\Cr$, there exists an $(n+1)$-cell with boundary~$\gamma$ in~$\Cr(\Gamma)$ or, equivalently, if the quotient $n$-category $\Cr/\Gamma$ has $n$-spheres of shape $(f,f)$ only. For example, the $n$-spheres of $\Cr$ form a homotopy basis of~$\Cr$.

\subsection{Coherent presentations of categories}

\subsubsection{Polygraphs}

A \emph{$1$-polygraph} is a pair $\Sigma=(\Sigma_0,\Sigma_1)$ made of a set $\Sigma_0$ and a cellular extension $\Sigma_1$ of $\Sigma_0$. The free category $\Sigma^*$ over $\Sigma$ is $\Sigma^*=\Sigma_0[\Sigma_1]$. A \emph{$2$-polygraph} is a triple $\Sigma=(\Sigma_0, \Sigma_1,\Sigma_2)$ where $(\Sigma_0,\Sigma_1)$ is a $1$-polygraph and $\Sigma_2$ is a cellular extension of the free category $\Sigma^*_1$. The free $2$-category $\Sigma^*$ over $\Sigma$, the free $(2,1)$-category $\tck{\Sigma}$ over $\Sigma$ and the category~$\cl{\Sigma}$ presented by $\Sigma$ are respectively defined by
\[
\Sigma^* \:=\: \Sigma_1^*[\Sigma_2]\:,
\qquad
\tck{\Sigma} \:=\: \Sigma_1^*(\Sigma_2)
 \qquad\text{and}\qquad
\cl{\Sigma} \:=\: \Sigma_1^*/\Sigma_2.
\]
A \emph{$(3,1)$-polygraph} is a pair $\Sigma=(\Sigma_2,\Sigma_3)$ made of a $2$-polygraph $\Sigma_2$ and a cellular extension $\Sigma_3$ of the free $(2,1)$-category $\tck{\Sigma}_2$. The free $(3,1)$-category $\tck{\Sigma}$ over $\Sigma$ and the $(2,1)$-category presented by $\Sigma$ are defined by
\[
\tck{\Sigma} \:=\: \tck{\Sigma}_2(\Sigma_3)
\qquad\text{and}\qquad
\cl{\Sigma} \:=\: \tck{\Sigma}_2/\Sigma_3.
\]
The category presented by a $(3,1)$-polygraph $\Sigma$ is the one presented by its underlying $2$-polygraph, namely $\cl{\Sigma}_2$. If $\Sigma$ is a polygraph, we identify its underlying $k$-polygraph $\Sigma_k$ and the set of \hbox{$k$-cells} of the corresponding cellular extension. We say that $\Sigma$ is \emph{finite} if it has finitely many cells in every dimension. A $(3,1)$-polygraph $\Sigma$ can be summarised by a diagram representing the cells and the source and target maps of the free $(3,1)$-category $\tck{\Sigma}$ it generates:
\[
\xymatrix@C=3em{
\Sigma_0
& \Sigma_1^*
	\ar@<.5ex> [l] ^-{t_0}
	\ar@<-.5ex> [l] _-{s_0}
& \tck{\Sigma}_2
	\ar@<.5ex> [l] ^-{t_1}
	\ar@<-.5ex> [l] _-{s_1}
& \tck{\Sigma}_3.
	\ar@<.5ex> [l] ^-{t_2}
	\ar@<-.5ex> [l] _-{s_2}
}
\]

\subsubsection{Coherent presentations of categories}

Let $\C$ be a category. A \emph{presentation of $\C$} is a $2$-poly\-graph~$\Sigma$ whose presented category $\cl{\Sigma}$ is isomorphic to $\C$. We usually commit the abuse to identify $\C$ and~$\cl{\Sigma}$ and we denote by $\cl{u}$ the image of a $1$-cell $u$ of $\Sigma^*$ through the canonical projection onto $\C$. An \emph{extended presentation of $\C$} is a $(3,1)$-polygraph $\Sigma$ whose presented category is isomorphic to $\C$. A \emph{coherent presentation of $\C$} is an extended presentation $\Sigma$ of $\C$ such that the cellular extension $\Sigma_3$ of $\tck{\Sigma}_2$ is a homotopy basis. 

\begin{example}[The standard coherent presentation]
\label{Example:StandardCoherentPresentation}

The \emph{standard presentation} $\Std_2(\C)$ of a category $\C$ is the $2$-polygraph whose cells are
\begin{itemize}
\item the $0$-cells of $\C$ and a $1$-cell $\rep{u}:x\fl y$ for every $1$-cell $u:x\fl y$ of $\C$,
\item a $2$-cell $\gamma_{u,v}:\rep{u}\rep{v}\dfl\rep{uv}$ for all composable $1$-cells $u$ and $v$ of $\C$, 
\item a $2$-cell $\iota_x:1_x\dfl\rep{1}_x$ for every $0$-cell $x$ of $\C$.
\end{itemize}
The \emph{standard coherent presentation} $\Std_3(\C)$ of $\C$ is $\Std_2(\C)$ extended with the following $3$-cells
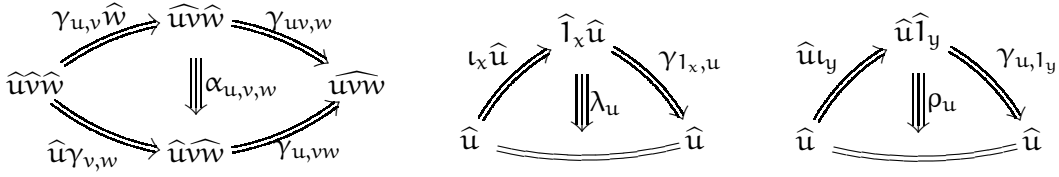
\[
\vcenter{
\xymatrix @C=3em @R=1em {
& {\rep{uv}\rep{w}}
	\ar@2 @/^/ [dr] ^-{\gamma_{uv,w}}
\\
{\rep{u}\rep{v}\rep{w}}
	\ar@2 @/^/ [ur] ^-{\gamma_{u,v}\rep{w}}
	\ar@2 @/_/ [dr] _-{\rep{u}\gamma_{v,w}}
&& {\rep{uvw}}
\\
& {\rep{u}\rep{vw}}
	\ar@2 @/_/ [ur] _-{\gamma_{u,vw}}
\ar@3 "1,2"!<0pt,-15pt>;"3,2"!<0pt,15pt> ^-{\alpha_{u,v,w}}
}
}
\qquad
\vcenter{
\xymatrix{
& {\rep{1}_x\rep{u}}
	\ar@2@/^/ [dr] ^-{\gamma_{1_x,u}}
\\
{\rep{u}}
	\ar@2@/^/ [ur]  ^{\iota_x \rep{u}}
	\ar@2{=}@/_/ [rr] _-{}="tgt"
&& {\rep{u}}
\ar@3 "1,2"!<0pt,-15pt>;"tgt"!<0pt,10pt> ^-{\lambda_u}	
}
}
\qquad
\vcenter{
\xymatrix{
& {\rep{u}\rep{1}_y}
	\ar@2@/^/ [dr] ^-{\gamma_{u,1_y}}
\\
{\rep{u}}
	\ar@2@/^/ [ur] ^-{\rep{u}\iota_y}
	\ar@2{=}@/_/ [rr] _-{}="tgt"
&& {\rep{u}}
\ar@3 "1,2"!<0pt,-15pt>;"tgt"!<0pt,10pt> ^-{\rho_u}	
}
}
\] 
where $u:x\fl y$, $v:y\fl z$ and $w:z\fl t$ range over the $1$-cells of $\C$. It is well known that those $3$-cells form a homotopy basis of $\tck{\Std_2(\C)}$, see~\cite[Chap.~VII, \textsection~2, Corollary]{MacLane98}.
\end{example}

\subsection{Cofibrant approximations of \pdf{2}-categories}

Let us recall the model structure for $2$-categories given by Lack in~\cite{Lack02} and~\cite{Lack04}.
A $2$-category is \emph{cofibrant} if its underlying $1$-category is free.
A $2$-functor $F:\Cr\fl\Dr$ is a \emph{weak equivalence} if it satisfies the following two conditions:
\begin{itemize}
\item every $0$-cell $y$ of $\Dr$ is equivalent to a $0$-cell $F(x)$ for $x$ in $\Cr$, \ie, there exist $1$-cells $u:F(x)\fl y$ and $v:y\fl F(x)$ and invertible $2$-cells $f:u\star_1 v \dfl 1_{F(x)}$ and $g:v\star_1 u\dfl 1_{y}$ in $\Dr$;
\item for all $0$-cells $x$ and $x'$ in $\Cr$, the induced functor $F(x,x'):\Cr(x,x')\fl\Dr(F(x),F(x'))$ is an equivalence of categories.
\end{itemize}
In particular, an equivalence of $2$-categories is a weak equivalence. More generally, a $2$-functor is a weak equivalence $F:\Cr\fl\Dr$ if, and only if, there exists a pseudofunctor $G:\Dr\fl\Cr$, see Section~\ref{Section:Actions}, that is a quasi-inverse for $F$, \emph{i.e.}, such that $GF\simeq 1_{\Cr}$ and $FG\simeq 1_{\Dr}$.

If $\Cr$ is a $2$-category, a \emph{cofibrant approximation of $\Cr$} is a cofibrant $2$-category $\tilde{\Cr}$ that is weakly equivalent to~$\Cr$.

\begin{theorem}
\label{Theorem:CoherentPresentationsCofibrantApproximations}
Let $\C$ be a category and let $\Sigma$ be an extended presentation of $\C$. The following assertions are equivalent:
\begin{enumerate}[\bf i)]
\item the $(3,1)$-polygraph $\Sigma$ is a coherent presentation of $\C$;
\item the $(2,1)$-category $\cl{\Sigma}$ presented by $\Sigma$ is a cofibrant approximation of $\C$. 
\end{enumerate}
\end{theorem}

\begin{proof}
Let us assume that $\Sigma_3$ is a homotopy basis of $\tck{\Sigma}_2$. By definition, the $2$-category $\cl{\Sigma}$ is cofibrant. Let us check that it is weakly equivalent to $\C$. We consider the canonical projection $\pi:\tck{\Sigma}\pfl\C$ that sends every $0$-cell to itself, every $1$-cell to its equivalence class and every $2$-cell and $3$-cell to the corresponding identity. This is well defined since two $1$-cells of $\tck{\Sigma}_2$ have the same equivalence class in $\C$ if, and only if, there exists a $2$-cell between them in $\tck{\Sigma}_2$ and since parallel $2$-cells of~$\tck{\Sigma}$ are sent to the same (identity) $2$-cell of $\C$. 

Since $\pi$ is the identity on $0$-cells, it is sufficient to check that it induces an equivalence of categories between $\cl{\Sigma}(x,y)$ and $\C(x,y)$ for all $0$-cells $x$ and $y$ in $\C$. We define a quasi-inverse~$\iota$ by choosing, for each $1$-cell $u:x\fl y$ of $\C$, an arbitrary representative $1$-cell $\iota(u)$ in $\cl{\Sigma}$. By construction, we have that~$\pi\iota$ is the identity of $\C(x,y)$. Moreover, for every $1$-cell $u:x\fl y$ of $\cl{\Sigma}$, the $1$-cell $\iota\pi(u)$ is a $1$-cell of $\cl{\Sigma}$ from $x$ to $y$ that has the same equivalence class as $u$: we choose an arbitrary $2$-cell $\alpha_u:u\dfl \iota\pi(u)$ in $\cl{\Sigma}$. Since all the parallel $2$-cells of $\cl{\Sigma}$ are equal, we get the following commutative diagram for every $2$-cell $f$ of $\cl{\Sigma}$:
\[
\xymatrix@R=1em @!C {
& {\iota\pi(u)}
	\ar@{=} @/^/ [dr] ^-{\iota\pi(f)} 
\\
u 
	\ar@2 @/^/ [ur] ^-{\alpha_u} 
	\ar@2 @/_/ [dr] _-{f}
		\ar@{} [rr] |{=}
&& {\iota\pi(v)}
\\
& v 
	\ar@2 @/_/ [ur] _-{\alpha_v}
}
\]
This proves that $\alpha$ is a natural isomorphism between $\iota\pi$ and the identity of $\cl{\Sigma}(x,y)$, yielding that~$\pi$ is a weak equivalence and, as a consequence, that $\cl{\Sigma}$ is a cofibrant approximation of $\C$. 

Conversely, let us assume that $\cl{\Sigma}$ is a cofibrant approximation of $\C$. Let $F:\cl{\Sigma}\fl\C$ be a weak equivalence and let $f,g:u\dfl v:x\fl y$ be parallel $2$-cells of $\tck{\Sigma}$. Since $F$ is a $2$-functor and $\C$ has identity $2$-cells only, we must have $F(u)=F(v)$ and $F(f)=F(g)=1_{F(u)}$. By hypothesis, the $2$-functor~$F$ induces an equivalence of categories between $\cl{\Sigma}(x,y)$ and $\C(x,y)$: we choose a quasi-inverse $G$ and a natural isomorphism $\alpha$ between $GF$ and the identity of $\cl{\Sigma}(x,y)$. We write the naturality conditions for $f$ and $g$ and, using $GF(f)=GF(g)=1_{GF(u)}$, we conclude that $f$ and $g$ are equal in $\cl{\Sigma}$:
\[
\xymatrix@R=1em @!C {
& {GF(u)}
	\ar@{=} @/^/ [dr] ^-{GF(f)} 
\\
u 
	\ar@2 @/^/ [ur] ^-{\alpha_u} 
	\ar@2 @/_/ [dr] _-{f}
		\ar@{} [rr] |{=}
&& {GF(v)}
\\
& v 
	\ar@2 @/_/ [ur] _-{\alpha_v}
}
\qquad\qquad
\xymatrix@R=1em @!C {
& {GF(u)}
	\ar@{=} @/^/ [dr] ^-{GF(g)} 
\\
u 
	\ar@2 @/^/ [ur] ^-{\alpha_u} 
	\ar@2 @/_/ [dr] _-{g}
		\ar@{} [rr] |{=}
&& {GF(v)}
\\
& v 
	\ar@2 @/_/ [ur] _-{\alpha_v}
}
\]
Thus $\Sigma$ is a coherent presentation of $\C$.
\end{proof}

\begin{remark}
The cofibrant approximations of a category $\C$ form, in general, a strictly larger class than the $2$-categories presented by coherent presentations of $\C$. Indeed, let $\C$ be the terminal category: it contains one $0$-cell and the corresponding identity $1$-cell only. Then $\C$ is cofibrant and, as a consequence, it is a cofibrant approximation of itself: this corresponds to the coherent presentation of $\C$ given by the $(3,1)$-polygraph with one $0$-cell and no higher-dimensional cells. But $\C$ also admits, as a cofibrant approximation, the ``equivalence'' $2$-category with two $0$-cells $x$ and $y$, two $1$-cells $u:x\fl y$ and $v:y\fl x$ and two invertible $2$-cells $f:uv\dfl 1_x$ and $g:vu\dfl 1_y$, and this $2$-category is not presented by a coherent presentation of $\C$, since it does not have the same $0$-cells as $\C$. 
\end{remark}

\begin{example}[{The standard cofibrant approximation~\cite{Lack02}}]
\label{Example:StandardCofibrantApproximation}

For any $2$-category $\Cr$, we denote by~$\rep{\Cr}$ the cofibrant $2$-category with the same $0$-cells as $\Cr$ and the following higher cells:
\begin{itemize}
\item the $1$-cells of $\rep{\Cr}$ are freely generated by the ones of $\Cr$, with $u$ in $\Cr$ denoted by $\rep{u}$ when seen as a generator of $\rep{\Cr}$;
\item the $2$-cells from $\rep{u}_1\cdots\rep{u}_m$ to $\rep{v}_1\cdots\rep{v}_n$ in $\rep{\Cr}$ are the $2$-cells from $u_1\cdots u_m$ to $v_1\cdots v_n$ in $\Cr$, with the same compositions as in $\Cr$.
\end{itemize}
The canonical projection $\rep{\Cr}\pfl\Cr$ is the identity on $0$-cells and maps each generating $1$-cell~$\rep{u}$ to~$u$ and each $2$-cell to itself: this is a weak equivalence whose quasi-inverse lifts a $2$-cell $f:u\dfl v$ to its distinguished representative $\rep{f}:\rep{u}\dfl\rep{v}$. Hence, the $2$-category $\rep{\Cr}$ is a cofibrant approximation of $\Cr$, called the \emph{standard cofibrant approximation of $\Cr$}. 

When $\Cr=\C$ is a category, the $2$-category $\rep{\C}$ has exactly one $2$-cell from $\rep{u}_1\cdots\rep{u}_m$ to $\rep{v}_1\cdots\rep{v}_n$ if, and only if, the relation $u_1\cdots u_m = v_1\cdots v_n$ holds in $\C$: this is a representative of an identity and, thus, it is invertible. As a consequence, the standard cofibrant approximation $\rep{\C}$ of $\C$ is exactly the $(2,1)$-category presented by the standard coherent presentation $\Std_3(\C)$ of $\C$.
\end{example}

\section{Homotopical completion and homotopical reduction}
\label{Section:CompletionReduction}

\subsection{Tietze transformations of \pdf{(3,1)}-polygraphs}
\label{Subsection:TietzeTransformations}

An equivalence of $2$-categories $F:\Cr\fl\Dr$ is a \emph{Tietze equivalence} if the quotient categories $\Cr_1/\Cr_2$ and $\Dr_1/\Dr_2$ are isomorphic. Two $(3,1)$-polygraphs are \emph{Tietze-equivalent} if the $2$-categories they present are Tietze-equivalent. In that case, they have the same $0$-cells (up to a bijection). In particular, two coherent presentations of the same category are Tietze-equivalent.

\subsubsection{Tietze transformations}

Let $\Sigma$ be a $(3,1)$-polygraph. Following the terminology of~\cite{Brown92}, a $2$-cell (resp. $3$-cell, resp. $3$-sphere) $\gamma$ of~$\Sigma$ is called \emph{collapsible} if it satisfies the following:
\begin{itemize}
\item the target of $\gamma$ is a $1$-cell (resp. $2$-cell, resp. $3$-cell) of the $(3,1)$-polygraph $\Sigma$,
\item the source of $\gamma$ is a $1$-cell (resp. $2$-cell, resp. $3$-cell) of the free $(3,1)$-category over $\Sigma\setminus\ens{t(\gamma)}$.
\end{itemize}
If $\gamma$ is collapsible, then its target is called a \emph{redundant} cell. A collapsible cell and its redundant target can be \emph{coherently} adjoined or removed from a $(3,1)$-polygraph, without changing the presented $2$-category, up to Tietze equivalence. These operations are formalised by Tietze transformations.

An \emph{elementary Tietze transformation} of a $(3,1)$-polygraph~$\Sigma$ is a $3$-functor with domain $\tck{\Sigma}$ that belongs to one of the following six operations:

\begin{enumerate}

\item Coherent adjunction or elimination of a redundant $1$-cell with its collapsible $2$-cell:
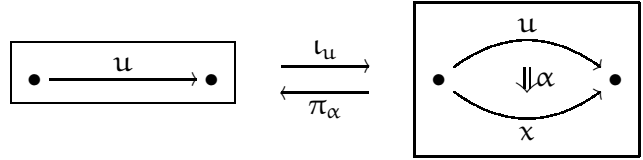
\[
\boxed{\xymatrix@C=5em{
\bullet
	\ar [r] ^-{u}
& \bullet
}}
\quad
\xymatrix@C=3em {
\strut 
	\ar@<1ex> [r] ^-{\iota_u} 
&\strut
	\ar@<1ex> [l] ^-{\pi_{\alpha}}
}
\quad
\boxed{\xymatrix@C=5em{
\bullet
	\ar@/^3ex/ [r] ^-{u} _-{}="src"
	\ar@/_3ex/ [r] _-{x} ^-{}="tgt"
\ar@2 "src"!<0pt,-10pt>;"tgt"!<0pt,10pt> ^-{\alpha}
& \bullet
}}
\]
The coherent adjunction $\iota_{u} : \tck{\Sigma} \ifl \tck{\Sigma}(x)(\alpha)$ is the canonical inclusion. The coherent elimination $\pi_{\alpha} : \tck{\Sigma} \pfl \tck{\Sigma}/\alpha$ maps $x$ to $u$ and $\alpha$ to $1_u$, leaving the other cells unchanged. The $(3,1)$-category $\tck{\Sigma}/\alpha$ is freely generated by the following $(3,1)$-polygraph $\Sigma/\alpha$:
\[
\xymatrix@C=4em{
\Sigma_0
& (\Sigma_1\setminus\ens{x})^*
	\ar@<.5ex> [l] ^-{t_0}
	\ar@<-.5ex> [l] _-{s_0}
& \tck{(\Sigma_2\setminus\ens{\alpha})}
	\ar@<.5ex> [l] ^-{\pi_{\alpha}\circ t_1}
	\ar@<-.5ex> [l] _-{\pi_{\alpha}\circ s_1}
& \tck{\Sigma}_3.
	\ar@<.5ex> [l] ^-{\pi_{\alpha}\circ t_2}
	\ar@<-.5ex> [l] _-{\pi_{\alpha}\circ s_2}
}
\]

\item Coherent adjunction or elimination of a redundant $2$-cell with its collapsible $3$-cell: 
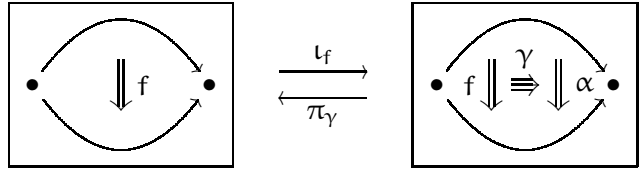
\[
\boxed{\xymatrix@C=5em{
\bullet
	\ar@/^5ex/ [r] _-{}="src"
	\ar@/_5ex/ [r] ^-{}="tgt"
\ar@2 "src"!<0pt,-15pt>;"tgt"!<0pt,15pt> ^-*+{f}
& \bullet
}}
\quad
\xymatrix@C=3em {
\strut 
	\ar@<1ex> [r] ^-{\iota_f} 
&\strut
	\ar@<1ex> [l] ^-{\pi_{\gamma}}
}
\quad
\boxed{\xymatrix@C=5em{
\bullet
	\ar@/^5ex/ [r] _-{}="src"
	\ar@/_5ex/ [r] ^-{}="tgt"
\ar@2 "src"!<-12.5pt,-15pt>;"tgt"!<-12.5pt,15pt> _-*+{f} ^-{}="src2"
\ar@2 "src"!<12.5pt,-15pt>;"tgt"!<12.5pt,15pt> ^-*+{\alpha} _-{}="tgt2"
\ar@3 "src2"!<7.5pt,0pt>;"tgt2"!<-7.5pt,0pt> ^-*+{\gamma}
& \bullet
}}
\]
The coherent adjunction $\iota_{f} : \tck{\Sigma} \ifl \tck{\Sigma}(\alpha)(\gamma)$ is the canonical inclusion. The coherent elimination $\pi_{\gamma} : \tck{\Sigma} \pfl \tck{\Sigma}/\gamma$ maps $\alpha$ to $f$ and $\gamma$ to $1_f$, leaving the other cells unchanged. The $(3,1)$-category $\tck{\Sigma}/\gamma$ is freely generated by the following $(3,1)$-polygraph $\Sigma/\gamma$:
\[
\xymatrix@C=4em{
\Sigma_0
& \Sigma_1^*
	\ar@<.5ex> [l] ^-{t_0}
	\ar@<-.5ex> [l] _-{s_0}
& \tck{(\Sigma_2\setminus\ens{\alpha})}
	\ar@<.5ex> [l] ^-{t_1}
	\ar@<-.5ex> [l] _-{s_1}
& \tck{(\Sigma_3\setminus\ens{\gamma})}.
	\ar@<.5ex> [l] ^-{\pi_{\gamma}\circ t_2}
	\ar@<-.5ex> [l] _-{\pi_{\gamma}\circ s_2}
}
\]

\item Coherent adjunction or elimination of a redundant $3$-cell:
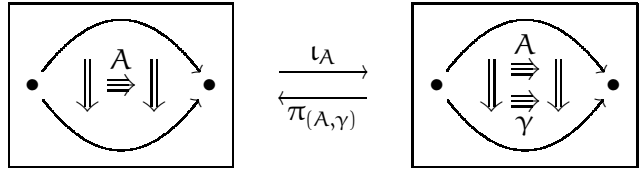
\[
\boxed{\xymatrix@C=5em{
\bullet
	\ar@/^5ex/ [r] _-{}="src"
	\ar@/_5ex/ [r] ^-{}="tgt"
\ar@2 "src"!<-12.5pt,-15pt>;"tgt"!<-12.5pt,15pt> _-*+{} ^-{}="src2"
\ar@2 "src"!<12.5pt,-15pt>;"tgt"!<12.5pt,15pt> ^-*+{} _-{}="tgt2"
\ar@3 "src2"!<7.5pt,0pt>;"tgt2"!<-7.5pt,0pt> ^-*+{A}
& \bullet
}}
\quad
\xymatrix@C=3em {
\strut 
	\ar@<1ex> [r] ^-{\iota_A} 
&\strut
	\ar@<1ex> [l] ^-{\pi_{(A,\gamma)}}
}
\quad
\boxed{\xymatrix@C=5em{
\bullet
	\ar@/^5ex/ [r] _-{}="src"
	\ar@/_5ex/ [r] ^-{}="tgt"
\ar@2 "src"!<-12.5pt,-15pt>;"tgt"!<-12.5pt,15pt> _-{} ^-{}="src2"
\ar@2 "src"!<12.5pt,-15pt>;"tgt"!<12.5pt,15pt> ^-{} _-{}="tgt2"
\ar@3 "src2"!<7.5pt,6pt>;"tgt2"!<-7.5pt,6pt> ^-*+{A}
\ar@3 "src2"!<7.5pt,-6pt>;"tgt2"!<-7.5pt,-6pt> _-*+{\gamma}
& \bullet
}}
\]
The coherent adjunction $\iota_{A} : \tck{\Sigma} \ifl \tck{\Sigma}(\gamma)$ is the canonical inclusion. The coherent elimination $\pi_{(A,\gamma)} : \tck{\Sigma} \pfl \tck{\Sigma}/(A,\gamma)$ maps $\gamma$ to $A$, leaving the other cells unchanged. The $(3,1)$-category $\tck{\Sigma}/(A,\gamma)$ is freely generated by the following $(3,1)$-polygraph $\Sigma/(A,\gamma)$:
\[
\xymatrix@C=4em{
\Sigma_0
& \Sigma_1^*
	\ar@<.5ex> [l] ^-{t_0}
	\ar@<-.5ex> [l] _-{s_0}
& \tck{\Sigma_2}
	\ar@<.5ex> [l] ^-{t_1}
	\ar@<-.5ex> [l] _-{s_1}
& \tck{(\Sigma_3\setminus\ens{\gamma})}.
	\ar@<.5ex> [l] ^-{t_2}
	\ar@<-.5ex> [l] _-{s_2}
}
\]
\end{enumerate}
If $\Sigma$ and $\Upsilon$ are $(3,1)$-polygraphs, a \emph{(finite) Tietze transformation from $\Sigma$ to $\Upsilon$} is a $3$-functor $F:\tck{\Sigma}\fl\tck{\Upsilon}$ that decomposes into a (finite) sequence of elementary Tietze transformations.

\begin{example}[The reduced standard coherent presentation]
\label{Example:ReducedStandardCoherentPresentation}
Let $\C$ be a category. One can reduce the standard coherent presentation $\Std_3(\C)$ of $\C$, given in Example~\ref{Example:StandardCoherentPresentation} into the smaller \emph{reduced standard coherent presentation} $\Std'_3(\C)$ of~$\C$. It is obtained from $\Std_3(\C)$ by a Tietze transformation that performs the following coherent eliminations, the resulting coherent presentation of the category $\C$ being detailed in~\cite[4.1.6]{GuiraudMalbos12}:
\begin{itemize}
\item the $3$-cells $\alpha_{1_x,u,v}$, $\alpha_{u,1_y,v}$ and $\alpha_{u,v,1_z}$, since they are parallel to composites of $\lambda$s and $\rho$s, 
\item the $2$-cells $\gamma_{1_x,u}$ and the $3$-cells $\lambda_u$,
\item the $2$-cells $\gamma_{u,1_x}$ and the $3$-cells $\rho_u$,
\item the $1$-cells $\rep{1}_x$ and the $2$-cells $\iota_x$.
\end{itemize}

\end{example}

\begin{theorem}
\label{Theorem:TietzeTransformations}
Two (finite) $(3,1)$-polygraphs $\Sigma$ and $\Upsilon$ are Tietze equivalent if, and only if, there exists a (finite) Tietze transformation between them. As a consequence, if $\Sigma$ is a coherent presentation of a category $\C$ and if there exists a Tietze transformation from $\Sigma$ to $\Upsilon$, then $\Upsilon$ is a coherent presentation of~$\C$.
\end{theorem}

\begin{proof}
Let us prove that, if two $(3,1)$-polygraphs are related by a Tietze transformation, then they are Tietze-equivalent. Since isomorphisms of categories and equivalence of $2$-categories compose, it is sufficient to check the result for each one of the six types of elementary Tietze transformations on a fixed $(3,1)$-polygraph $\Sigma$. By definition, the $3$-functors $\pi\circ\iota$ are all equal to the identity of $\tck{\Sigma}$ and the $3$-functors $\iota\circ\pi$ induce identities on the presented category. Moreover, the latter induce the following $2$-functors on the presented $2$-category $\cl{\Sigma}$:
\[
\cl{\iota}_{u}\circ\cl{\pi}_{\alpha} \:\simeq\: 1_{\cl{\Sigma}},
\qquad\qquad
\cl{\iota}_{f}\circ\cl{\pi}_{A} \:=\: 1_{\cl{\Sigma}},
\qquad\qquad
\cl{\iota}_{A}\circ\cl{\pi}_{(A,\gamma)} \:=\: 1_{\cl{\Sigma}}.
\]
Indeed, the first isomorphism is the identity on every cell, except on $x$ which is mapped to $\cl{\alpha}$. The second and third isomorphisms are, in fact, identities since they do not change the equivalence classes of $2$-cells modulo $3$-cells.

Conversely, let $\Sigma$ and $\Upsilon$ be Tietze-equivalent $(3,1)$-polygraphs. We fix an equivalence $F:\cl{\Sigma}\fl\cl{\Upsilon}$ of $2$-categories that induces an isomorphism on the presented categories. We choose a weak inverse $G:\cl{\Upsilon}\fl\cl{\Sigma}$ and pseudonatural isomorphisms $\sigma:GF\dfl 1_{\cl{\Sigma}}$ and $\tau:FG\dfl 1_{\cl{\Upsilon}}$, in such a way that the quadruple $(F,G,\sigma,\tau)$ is an adjoint equivalence, which is always feasible~\cite[Chap.~IV, \textsection~4, Theorem~1]{MacLane98}. This means that the following ``triangle identities'' hold:
\[
\xymatrix@!C@C=3em{
FGF
	\ar@2@/^3ex/ [r] ^{F\sigma} _{}="src"
	\ar@2@/_3ex/ [r] _{\tau F} ^{}="tgt"
& F
\ar@{} "src";"tgt" |{\sm =}
}
\qquad\qquad
\xymatrix@!C@C=3em{
GFG
	\ar@2@/^3ex/ [r] ^{G\tau} _{}="src"
	\ar@2@/_3ex/ [r] _{\sigma G} ^{}="tgt"
& G
\ar@{} "src";"tgt" |{\sm =}
}
\]
Let us lift the $2$-functor $F$ to a $3$-functor $\rep{F}:\tck{\Sigma}\fl\tck{\Upsilon}$, defined as $F$ on the $0$-cells and $1$-cells. For every $2$-cell $\alpha:u\dfl v$ of $\Sigma$, we choose a representative $\rep{F}(\alpha):F(u)\dfl F(v)$ of $F(\cl{\alpha})$ in $\tck{\Upsilon}$ and, then, we extend $\rep{F}$ by functoriality to every $2$-cell of $\tck{\Sigma}$. For a $3$-cell $\gamma:f\tfl g$ of $\Sigma$, we have $\cl{f}=\cl{g}$ by definition of $\cl{\Sigma}$, so that $F(\cl{f})=F(\cl{g})$ holds in $\cl{\Upsilon}$, meaning that there exists a $3$-cell in $\tck{\Upsilon}$ from $\rep{F}(f)$ to $\rep{F}(g)$: we take it as a value for $\rep{F}(\gamma)$ and we extend $\rep{F}$ to every $3$-cell of $\tck{\Sigma}$ by functoriality. We proceed similarly with~$G$ to get a $3$-functor $\rep{G}:\tck{\Upsilon}\fl\tck{\Sigma}$. 

Then, for a $1$-cell $x$ of $\Sigma$, we choose a representative $\rep{\sigma}_x:GF(x)\dfl x$ of $\sigma_x$ in~$\tck{\Sigma}$ and we extend it to every $1$-cell by functoriality. If $\alpha:u\dfl v$ is a $2$-cell of $\Sigma$, the naturality condition satisfied by $\sigma$ on $\cl{\alpha}$ lifts to an arbitrarily chosen $3$-cell of $\Sigma$
\[
\xymatrix@!C@R=1em@C=2em{
& GF(v)
	\ar@2@/^/ [dr] ^{\rep{\sigma}_v}
	\ar@3 []!<0pt,-15pt>;[dd]!<0pt,15pt> ^{\rep{\sigma}_{\alpha}}
\\
GF(u) 
	\ar@2@/^/ [ur] ^{\rep{G}\rep{F}(\alpha)}
	\ar@2@/_/ [dr] _{\rep{\sigma}_u}
&& v
\\
& u
	\ar@2@/_/ [ur] _{\alpha}
}
\]
We proceed similarly with $\tau$. The conditions for the adjoint equivalence also lift to a $3$-cell $\lambda_x$ of~$\tck{\Upsilon}$ for every $1$-cell $x$ of $\Sigma$ and to a $3$-cell $\rho_y$ of $\tck{\Sigma}$ for every $1$-cell $y$ of $\Upsilon$:
\[
\xymatrix@!C@C=3em{
FGF(x)
	\ar@2@/^4ex/ [r] ^{\rep{F}(\rep{\sigma}_x)} _{}="src"
	\ar@2@/_4ex/ [r] _{\rep{\tau}_{F(x)}} ^{}="tgt"
& F(x)
\ar@3 "src"!<0pt,-10pt>;"tgt"!<0pt,10pt> ^-{\lambda_x}
}
\qquad\qquad
\xymatrix@!C@C=3em{
GFG(y)
	\ar@2@/^4ex/ [r] ^{\rep{G}(\rep{\tau}_y)} _{}="src"
	\ar@2@/_4ex/ [r] _{\rep{\sigma}_{G(y)}} ^{}="tgt"
& G(y)
\ar@3 "src"!<0pt,-10pt>;"tgt"!<0pt,10pt> ^-{\rho_y}
}
\]
Now, let us build a Tietze transformation from $\Sigma$ to $\Upsilon$. We start by constructing a $(3,1)$-polygraph~$\Xi$ that contains both $\Sigma$ and $\Upsilon$, together with coherence cells that correspond to the Tietze equivalence. The $(3,1)$-polygraph $\Xi$ has the same $0$-cells as $\Sigma$ (and as $\Upsilon$) and it contains the $1$-cells, $2$-cells and $3$-cells of~$\Sigma$ and $\Upsilon$, plus the following cells:
\begin{itemize}

\item Two $2$-cells $\phi_x:F(x)\dfl x$ and $\psi_y:G(y)\dfl y$, for all $1$-cells $x$ of $\Sigma$ and $y$ of $\Upsilon$. Using the fact that $F$ is a functor that preserves the $0$-cells, we extend $\phi$ to every $1$-cell $u$ of $\tck{\Sigma}$ by functoriality, \ie{} by $\phi_{1_p} = 1_{1_p}$ and $\phi_{uu'} = \phi_u \phi_{u'}$, to get a $2$-cell $\phi_u:F(u)\dfl u$ for every $1$-cell $u$ of~$\tck{\Sigma}$. We proceed similarly with $\psi$ to define a $2$-cell $\psi_v:G(v)\dfl v$ of $\tck{\Xi}$ for every $1$-cell $v$ of~$\tck{\Upsilon}$.

\item Two $3$-cells $\phi_{\alpha}$ and $\psi_{\beta}$, for all $2$-cells $\alpha:u\dfl u'$ and $\beta:v\dfl v'$, with the following shapes:
\[
\xymatrix@!C {
& F(u)
	\ar@2 [r] ^{\rep{F}(\alpha)} _{}="src"
& F(u')
	\ar@2@/^/ [dr] ^{\phi_{u'}}
\\
u
	\ar@2@/^/ [ur] ^{\phi_u^-} 
	\ar@2@/_/ [rrr] _{\alpha} ^{}="tgt"
\ar@3 "src"!<0pt,-10pt>;"tgt"!<0pt,10pt> ^{\phi_{\alpha}}
&&& u'
}
\qquad
\xymatrix@!C {
& G(v)
	\ar@2 [r] ^{\rep{G}(\beta)} _{}="src"
& G(v')
	\ar@2@/^/ [dr] ^{\psi_{v'}}
\\
v
	\ar@2@/^/ [ur] ^{\psi_v^-} 
	\ar@2@/_/ [rrr] _{\beta} ^{}="tgt"
\ar@3 "src"!<0pt,-10pt>;"tgt"!<0pt,10pt> ^{\psi_{\beta}}
&&& v'
}
\]
We use the $2$-functoriality of the sources and targets of $\phi_{\alpha}$ and $\psi_{\beta}$ to extend $\phi$ and $\psi$ to every $2$-cells $f$ of $\tck{\Sigma}$ and $g$ of $\tck{\Upsilon}$, respectively.

\item Two $3$-cells $\xi_x$ and $\eta_y$, for all $1$-cells $x$ of $\Sigma$ and $y$ of $\Upsilon$, with the following shapes:
\[
\xymatrix @!C @C=3em {
& GF(x)
	\ar@2@/^/ [dr] ^{\rep{\sigma}_x}
	\ar@3 []!<0pt,-15pt>;[d]!<0pt,5pt> ^{\xi_x}
\\
F(x)
	\ar@2@/^/ [ur] ^{\psi^-_{F(x)}}
	\ar@2@/_/ [rr] _{\phi_x}
&& x
}
\qquad\qquad
\xymatrix @!C @C=3em {
& FG(y)
	\ar@2@/^/ [dr] ^{\rep{\tau}_y}
	\ar@3 []!<0pt,-15pt>;[d]!<0pt,5pt> ^{\eta_y}
\\
G(y)
	\ar@2@/^/ [ur] ^{\phi^-_{G(y)}}
	\ar@2@/_/ [rr] _{\psi_y}
&& y
}
\]
We then extend $\epsilon$ and $\eta$ to all $1$-cells $u$ of $\tck{\Sigma}$ and $v$ of $\tck{\Upsilon}$, respectively.
\end{itemize}
We construct a Tietze transformation $\Phi$ from $\Sigma$ to $\Xi$ step-by-step, as follows. 
\begin{itemize}
\item Adjunction of the cells of $\Upsilon$. For every $1$-cell $y$ of $\Upsilon$, we apply $\iota_{G(y)}$ to coherently add $y$ and $\psi_y: G(y)\dfl y$. Then, for every $2$-cell $\beta:v\dfl v'$ of $\Upsilon$, we apply $\iota_{\psi_v^-\star_1\rep{G}(\beta)\star_1 \psi_{v'}}$ to coherently add $\beta$ and $\psi_{\beta}$. Then, we add every $3$-cell $\delta:g\tfl g'$ of $\Upsilon$ with $\iota_{B}$, where $B$ is the $3$-cell of~$\tck{\Xi}$ defined by
\[
B\:=\: \psi_g^- \star_2 \big( \psi_v^- \star_1 G(\delta) \star_1 \psi_{v'} \big) \star_2 \psi_{g'}
\]
and pictured as follows:
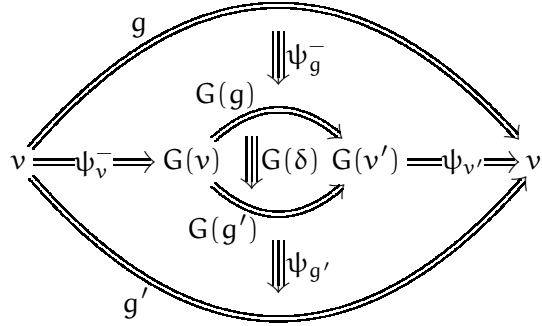
\[
\xymatrix@!C @C=3em {
v
	\ar@2@/^12ex/ [rrr] ^(0.25){g} _{}="src1"
	\ar@2 [r] |-{\psi_v^-}
	\ar@2@/_12ex/ [rrr] _(0.25){g'} _{}="tgt3"
&	G(v)
	\ar@2@/^4ex/ [r] ^(0.3){G(g)} ^{}="tgt1" _{}="src2"
	\ar@2@/_4ex/ [r] _(0.3){G(g')} ^{}="tgt2" _{}="src3"
& G(v')
	\ar@2 [r] |-{\psi_{v'}}
& v'
\ar@3 "src1"!<0pt,-10pt>;"tgt1"!<0pt,10pt> ^{\psi_g^-}
\ar@3 "src2"!<-10pt,-10pt>;"tgt2"!<-10pt,10pt> ^{G(\delta)}
\ar@3 "src3"!<0pt,-10pt>;"tgt3"!<0pt,10pt> ^{\psi_{g'}}
}
\]

\item Adjunction of the coherence cells for $\Sigma$. For every $1$-cell $x$, we apply $\iota_{\psi^-_{F(x)}\star_1\rep{\sigma}_x}$ to coherently add the $2$-cell $\phi_x$ and the $3$-cell $\xi_x$. Then, for every $2$-cell $\alpha:u\dfl u'$ of $\Sigma$, we add the $3$-cell~$\phi_{\alpha}$ with~$\iota_A$, where $A$ is the $3$-cell of $\tck{\xi}$ defined by
\[
A \:=\: 
	\big( \phi_u^- \star_1 \xi_u \star_1 \rep{\sigma}_u^- \star_1 \psi_{F(u)} 
	\star_1 \psi_{\rep{F}(\alpha)}^-
	\star_1 \xi_{u'}^- \big)
	\:\star_2\: 
	\big( \rep{\sigma}_u^- \star_1 \rep{\sigma}_{\alpha} \big)
\]
and pictured as follows, where we abusively simplify the labels of $3$-cells for readability:
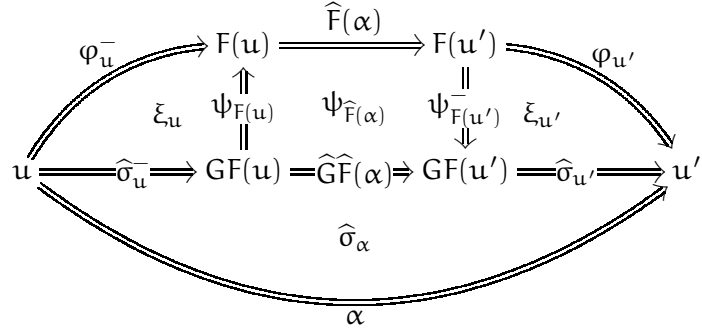
\[
\xymatrix @!C @C=4em @R=2.75em {
& F(u) 
	\ar@2 [r] ^{\rep{F}(\alpha)}
	\ar@{} [dr] |{\psi_{\rep{F}(\alpha)}}
& F(u')
	\ar@2@/^3ex/ [dr] ^{\phi_{u'}} _{}="src2"
	\ar@2 [d] |{\psi^-_{F(u')}}
\\
u 
	\ar@2@/^3ex/ [ur] ^{\phi_u^-} _{}="src1"
	\ar@2 [r] |{\rep{\sigma}^-_u}
	\ar@2@/_10ex/ [rrr] _{\alpha} ^{}="tgt3"
& GF(u)
	\ar@2 [u] |{\psi_{F(u)}}
	\ar@2 [r] |{\rep{G}\rep{F}(\alpha)} _{}="src3"
& GF(u')
	\ar@2 [r] |{\rep{\sigma}_{u'}}
& u'
\ar@{} "src1";"2,2" |-{\xi_u}
\ar@{} "src2";"2,3" |-{\xi_{u'}}
\ar@{} "src3";"tgt3" |{\rep{\sigma}_{\alpha}}
}
\]

\item Adjunction of the last coherence cells for $\Upsilon$. For every $1$-cell $y$ of $\Upsilon$, we add the $3$-cell $\eta_y$ with~$\iota_C$, where $C$ is the $3$-cell of $\tck{\Xi}$ defined by
\[
C \:=\: 
	\big( \phi_{G(y)}^- \star_1 \xi_{G(y)} \star_1 \rep{\sigma}_{G(y)}^- \star_1 \psi_{FG(y)} \star_1 \psi_{\rep{\tau}_y}^- \big)
	\:\star_2\:
	\big( \rep{\sigma}_{G(y)}^- \star_1 \rho_y \star_1 \psi_y \big)
\]
and pictured, in a simplified way, as follows
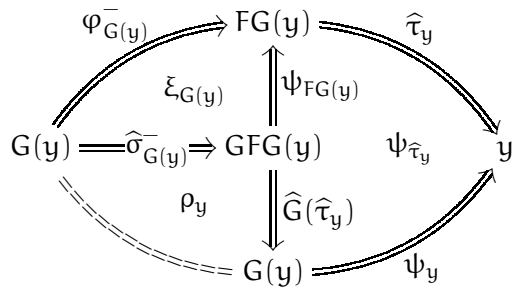
\[
\xymatrix@!C @C=4em @R=2.75em {
& FG(y)
	\ar@2@/^3ex/ [dr] ^{\rep{\tau}_y}
\\
G(y)
	\ar@2@/^3ex/ [ur] ^{\phi^-_{G(y)}} _{}="src1"
	\ar@2 [r] |{\rep{\sigma}^-_{G(y)}}
	\ar@2{=}@/_3ex/ [dr] ^{}="src2"
& GFG(y)
	\ar@2 [u] _{\psi_{FG(y)}}
	\ar@2 [d] ^{\rep{G}(\rep{\tau}_y)}
	\ar@{} [r] |-{\psi_{\rep{\tau}_y}}
& y
\\
& G(y)
	\ar@2@/_3ex/ [ur] _{\psi_y}
\ar@{} "src1";"2,2" |-{\xi_{G(y)}}
\ar@{} "src2";"2,2" |-{\rho_y}
}
\]
\end{itemize}
As a result, we get a Tietze transformation $\Phi$ from $\Sigma$ to $\Xi$. Since the construction and the result are totally symmetric in $\Sigma$ and $\Upsilon$, and since the Tietze transformation $\Phi$ contains coherent adjunctions only, we also get a Tietze transformation $\Psi$ from $\Xi$ to $\Upsilon$. By composition, we get a Tietze transformation from~$\Sigma$ to $\Upsilon$. To conclude, we note that both $\Phi$ and $\Psi$ are finite when both $\Sigma$ and $\Upsilon$ are.

Finally, if $\Sigma$ is a coherent presentation of a category $\C$, then the $2$-category it presents is a cofibrant approximation of $\C$ by Theorem~\ref{Theorem:CoherentPresentationsCofibrantApproximations}. Moreover, if there exists a Tietze transformation from $\Sigma$ to~$\Upsilon$, they are Tietze-equivalent by the first part of the proof. Thus, the categories presented by $\Sigma$ and~$\Upsilon$ are isomorphic (to $\C$), and the $2$-categories they present are equivalent, hence weakly equivalent. As a consequence, the $2$-category presented by $\Upsilon$ is also a cofibrant approximation of $\C$ so that, by Theorem~\ref{Theorem:CoherentPresentationsCofibrantApproximations}, we conclude that $\Upsilon$ is a coherent presentation of~$\C$.
\end{proof}

\subsubsection{Higher Nielsen transformations}
\label{Subsubsection:Nielsen}

We introduce higher-dimensional analogues of Nielsen transformations to perform replacement of cells in $(3,1)$-polygraphs. The \emph{elementary Nielsen transformations} on a $(3,1)$-polygraph $\Sigma$ are the following operations:
\begin{enumerate}
\item The replacement of a $2$-cell by a formal inverse (including in the source and target of every $3$-cell).
\item The replacement of a $3$-cell  by a formal inverse.
\item The replacement of a $3$-cell $\gamma : f \tfl g$ by a $3$-cell $\tilde{\gamma}:h\star_1 f\star_1 k \tfl h\star_1 g\star_1 k$, where $h$ and $k$ are $2$-cells of $\tck{\Sigma}$.
\end{enumerate}
Each one of those three elementary Nielsen transformations is a Tietze transformation. For example, the last one is the composition of the following elementary Tietze transformations:
\begin{itemize}
\item the coherent adjunction $\iota_{h\star_1\gamma\star_1 k}$ of the $3$-cell $\tilde{\gamma}:h\star_1 f\star_1 k \tfl h\star_1 g\star_1 k$,
\item the coherent elimination $\pi_{h^-\star_1\tilde{\gamma}\star_1k ^-}$ of $\gamma$.
\end{itemize}
The replacement of a $2$-cell $\alpha:u\dfl v$ by a formal inverse $\tilde{\alpha}:v\dfl u$ is the composition of:
\begin{itemize}
\item the coherent adjunction $\iota_{\alpha^-}$ of the $2$-cell $\tilde{\alpha}:v\dfl u$ and a $3$-cell $ \gamma:\alpha^-\tfl\tilde{\alpha}$,
\item the Nielsen transformation that replaces $\gamma$ with $\tilde{\gamma}:\tilde{\alpha}^-\tfl\alpha$ by composition with $\alpha$ on one side and by $\tilde{\alpha}^-$ on the other side,
\item the coherent elimination $\pi_{\tilde{\gamma}}$ of  $\alpha$ and $\tilde{\gamma}$. 
\end{itemize}

In what follows, we perform coherent eliminations of cells that are collapsible only up to a Nielsen transformation (a composition of elementary ones). If $f$ is Nielsen-equivalent to a collapsible cell $\tilde{f}$, we abusively denote by $\pi_{f}$ the corresponding coherent elimination, with a precision about the eliminated cell $t(\tilde{f})$ when it is not clear from the context. In a similar way, if $(A, B)$ is a noncollapsible $3$-sphere of~$\tck{\Sigma}$, we denote by $\pi_{(A,B)}$ the potential coherent elimination corresponding to a collapsible $3$-sphere~$\tck{\Sigma}$ obtained from $(A,B)$ by composition with $2$-cells and $3$-cells of $\tck{\Sigma}$.

\subsection{Homotopical completion}
\label{SubsectionHomotopicalCompletion}

In this section, we recall notions of rewriting theory for $2$-polygraphs from~\cite[4.1]{GuiraudMalbos09} and~\cite[4.1]{GuiraudMalbos12}, together with Squier's completion to compute coherent presentations from convergent presentations. Then we extend Squier's completion to terminating $2$-polygraphs thanks to Knuth-Bendix's completion~\cite{KnuthBendix70}. 

\subsubsection{Rewriting properties of \pdf{2}-polygraphs}

A  \emph{rewriting step} of a $2$-polygraph $\Sigma$ is a $2$-cell of the free $2$-category $\Sigma^*$ with shape
\[
\xymatrix@C=4em{
{y}
	\ar [r] ^-{w} 
& {x}
	\ar@/^3ex/ [r] ^-{u} ^{}="src"
	\ar@/_3ex/ [r] _-{v} ^{}="tgt"
	\ar@2 "src"!<0pt,-10pt>;"tgt"!<0pt,10pt> ^-{\alpha}
& {x'}
	\ar [r] ^-{w'}
& {y'}
}
\]
where $\alpha:u\dfl v$ is a $2$-cell of $\Sigma$ and $w$ and $w'$ are $1$-cells of $\Sigma^*$. A \emph{normal form} is a $1$-cell that is the source of no rewriting step. 

We say that $\Sigma$ \emph{terminates} if it has no infinite rewriting sequence (no infinite sequence of composable rewriting steps). In that case, the relations $s(f)>t(f)$ for $f$ a rewriting step define a \emph{termination order}: this is a well-founded order relation on the $1$-cells that is compatible with the composition. Another example of termination order is the deglex order that first compares the length and, then, uses a lexicographic order on the words of same length. In fact, the existence of a termination order is sufficient to prove termination. 

A \emph{branching of $\Sigma$} is a (non-ordered) pair $(f,g)$ of $2$-cells of $\Sigma^*$ with a common source, also called the source of the branching. We say that $\Sigma$ is \emph{confluent} if all of its branchings are confluent, that is, for every branching $(f,g)$, there exist $2$-cells $f'$ and $g'$ in $\Sigma^*$, as in the following diagram: 
\[
\xymatrix @R=1em @C=3em{
& {v}
	\ar@2@/^/ [dr] ^-{f'}
\\
{u}
	\ar@2@/^/ [ur] ^-{f}
	\ar@2@/_/ [dr] _-{g}
&& {u'}
\\
& {w}
	\ar@2@/_/ [ur] _-{g'}
}
\]
A branching $(f,g)$ is local if $f$ and $g$ are rewriting steps. The local branchings are classified as follows:
\begin{itemize}
\item \emph{aspherical} branchings have shape $(f,f)$,
\item \emph{Peiffer} branchings have shape $(f v, u g)$, where $u=s(f)$ and $v=s(g)$,
\item \emph{overlap} branchings are all the other cases.
\end{itemize}
Local branchings are ordered by inclusion of their sources, and a minimal overlap branching is called \emph{critical}. Under the termination hypothesis, confluence is equivalent to confluence of critical branchings. 

We say that $\Sigma$ is \emph{convergent} if it terminates and is confluent. Such a~$\Sigma$ is called a \emph{convergent presentation} of the category $\cl{\Sigma}$, and of any category that is isomorphic to $\cl{\Sigma}$. In that case, every $1$-cell $u$ of $\Sigma^*$ has a unique normal form, denoted by $\rep{u}$, so that we have $\cl{u}=\cl{v}$ in $\cl{\Sigma}$ if, and only if, $\widehat{u}=\widehat{v}$ holds in~$\Sigma^*$. This extends to a section $\cl{\Sigma}\ifl\Sigma^*$ of the canonical projection, sending a $1$-cell $u$ of~$\cl{\Sigma}$ to the unique normal form of its representative $1$-cells in $\Sigma^*$, still denoted by $\rep{u}$. 
A $(3,1)$-polygraph is \emph{convergent} if its underlying $2$-polygraph is.

\subsubsection{Squier's completion for convergent polygraphs}

Let us assume that $\Sigma$ is convergent. A \emph{family of generating confluences of~$\Sigma$} is a cellular extension of $\tck{\Sigma}$ that contains exactly one $3$-cell  
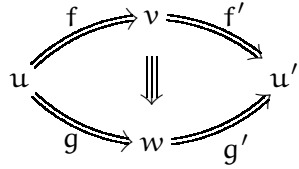
\[
\xymatrix @R=1em@C=3em @!C{
& {v}
	\ar@2 @/^/ [dr] ^{f'}
	\ar@3 []!<0pt,-15pt>;[dd]!<0pt,15pt> 
\\
{u}
	\ar@2 @/^/ [ur] ^{f}
	\ar@2 @/_/ [dr] _{g}
&& {u'}
\\
& {w}
	\ar@2 @/_/ [ur] _{g'}
}
\]
for every critical branching $(f,g)$ of $\Sigma$. Such a family always exists by confluence but it is not necessarily unique. Indeed, the $3$-cell can be directed in the reverse way and, for a given branching $(f,g)$, one can have several possible $2$-cells $f'$ and~$g'$ with the required shape (see~\cite[4.3.2]{GuiraudMalbos12} for a constructive version, based on normalisation strategies). 
We call \emph{Squier's completion of $\Sigma$} the $(3,1)$-polygraph 
obtained from~$\Sigma$ by adjunction of a chosen family of generating confluences of $\Sigma$. The following result is due to Squier, we refer to ~\cite[Theorem~4.4.2]{GuiraudMalbos13} for a proof in our language.

\begin{theorem}[{\cite[Theorem~5.2]{Squier94}}]
\label{Theorem:SquierCompletion}
For every convergent presentation~$\Sigma$ of a category~$\C$, Squier's completion of~$\Sigma$ is a coherent presentation of~$\C$.
\end{theorem}

\subsubsection{Homotopical completion}
\label{Subsubsection:HomotopicalCompletion}

Let $\Sigma$ be a terminating $2$-polygraph, equipped with a total termination order $\leq$. The \emph{homotopical completion of $\Sigma$} is the $(3,1)$-polygraph $\Sr(\Sigma)$ obtained from $\Sigma$ by successive application of Knuth-Bendix's and Squier's completions. In fact, both constructions can be interleaved to compute $\Sr(\Sigma)$, as we describe here.

One considers each critical branching $(f,g)$ of $\Sigma$. There are two possible situations, shown below, depending on whether $(f,g)$ is confluent or not:
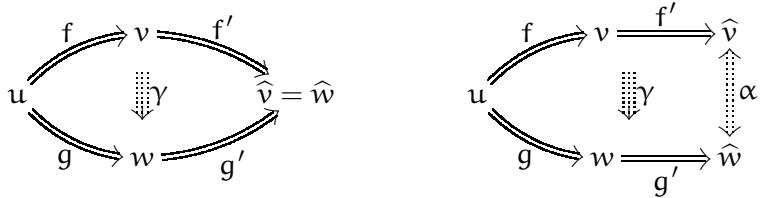
\[
\xymatrix @C=3em @R=1em {
& {v}
	\ar@2@/^/ [dr] ^-{f'}
	\ar@3{.>} []!<0pt,-15pt>;[dd]!<0pt,15pt> ^-{\gamma}
\\
{u}
	\ar@2@/^/ [ur] ^-{f}
	\ar@2@/_/ [dr] _-{g}
&& {\rep{v}=\rep{w}}
\\
& {w}
	\ar@2@/_/ [ur] _-{g'}
}
\qquad\qquad
\xymatrix @C=3em @R=1em {
& {v}
	\ar@2 [r] ^-{f'}
	\ar@3{.>} []!<10pt,-15pt>;[dd]!<10pt,15pt> ^-{\gamma}
& {\rep{v}}
	\ar@2{<.>} [dd] ^-{\alpha}
\\
{u}
	\ar@2@/^/ [ur] ^-{f}
	\ar@2@/_/ [dr] _-{g}
\\
& {w}
	\ar@2 [r] _-{g'}
& {\rep{w}}	
}
\]
If $(f,g)$ is confluent, the left case occurs and one adds the dotted $3$-cell $\gamma$ to $\Sigma$. Otherwise, one performs a Tietze transformation on $\Sigma$ to coherently add the $2$-cell $\alpha$ and the $3$-cell $\gamma$. To preserve termination, the $2$-cell $\alpha$ is directed from $\rep{v}$ to $\rep{w}$ if $\rep{v}>\rep{w}$ and in the reverse direction otherwise. To be formal, the coherent adjunction would add a $3$-cell $\gamma$ with target $\alpha$, but we implicitly perform a Nielsen transformation for convenience.

The potential adjunction of additional $2$-cells $\alpha$ can create new critical branchings, whose confluence must also be examined, possibly generating the adjunction of additional $2$-cells and $3$-cells. This defines an increasing sequence of $(3,1)$-polygraphs, where $\Sigma^{n+1}$ is obtained by completion of the critical branchings of $\Sigma^n$:
\[
(\Sigma,\emptyset) 
	\:=\: \Sigma^0 
	\:\subseteq\: \Sigma^1 
	\:\subseteq\: \cdots 
	\:\subseteq\: \Sigma^n 
	\:\subseteq\: \Sigma^{n+1}
	\:\subseteq\: \cdots
\]
The $(3,1)$-polygraph $\Sr(\Sigma)$ is defined as the union of this increasing sequence. If the $2$-polygraph $\Sigma$ is already confluent, the homotopical completion is exactly Squier's completion. 
As a consequence of  Theorem~\ref{Theorem:SquierCompletion}, we get that the potentially infinite $(3,1)$-polygraph $\Sr(\Sigma)$ satisfies the following properties.

\begin{theorem}
\label{Theorem:HomotopicalCompletion}
For every terminating presentation $\Sigma$ of a category $\C$, the homotopical completion $\Sr(\Sigma)$ of $\Sigma$ is a coherent convergent presentation of $\C$.
\end{theorem}

\begin{example}
\label{Example:KNCompletion}
From~\cite{KapurNarendran85}, we consider the presentation $\Sigma = \big(s, t, a \;;\; ta \odfl{\alpha} as \,,\, st \odfl{\beta} a\big)$ of $\B_3^+=\B^+(S_3)$, obtained from Artin's presentation by coherent adjunction of the Coxeter element $st$ and the $2$-cell $\beta$. The deglex order generated by $t>s>a$ proves the termination of $\Sigma$. The homotopical completion of $\Sigma$ is the $(3,1)$-polygraph
\[
\Sr(\Sigma) \:=\: \big( s, t, a \;;\; ta \odfl{\alpha} as \,,\, st \odfl{\beta} a \,,\, sas \odfl{\gamma} aa \,,\, saa \odfl{\delta} aat \;;\; A, B, C,D \big)
\]
where $A$, $B$, $C$ and $D$ are the following $3$-cells, induced by completion of critical pairs $(\beta a,s\alpha)$ and $(\gamma t, sa\beta)$:
\[
  \vcenter{
    \xymatrix@R=1.5em@C=2em{
      & aa
      \\
      sta
      \ar@2@/^/ [ur] ^{\beta a} _(0.66){}="src"
      \ar@2@/_/ [dr] _{s\alpha} ^(0.66){}="tgt"
      \\
      & sas
      \ar@2@/_/ [uu] _{\gamma}
      \ar@3 "src"!<0pt,-15pt>;"tgt"!<0pt,15pt> ^{A}
    }
  }
  \quad
  \vcenter{
    \xymatrix@R=1.5em@C=2em{
      & aat
      \\
      sast
      \ar@2@/^/ [ur] ^{\gamma t} _(0.66){}="src"
      \ar@2@/_/ [dr] _{sa\beta} ^(0.658){}="tgt"
      \\
      & saa
      \ar@2@/_/ [uu] _{\delta}
      \ar@3 "src"!<0pt,-15pt>;"tgt"!<0pt,15pt> ^{B}
    }
  }
  \quad
  \vcenter{
    \xymatrix@R=1.5em@C=1em{
      & aaas
      \ar@3 []!<0pt,-20pt>;[dd]!<0pt,20pt> ^{C}
      \\
      sasas
      \ar@2@/^/ [ur] ^{\gamma as}
      \ar@2@/_/ [dr] _{sa\gamma}
      && aata
      \ar@2@/_/ [ul] _{aa\alpha}
      \\
      & saaa
      \ar@2@/_/ [ur] _{\delta a}
    }
  }
  \quad
  \vcenter{
    \xymatrix@R=1.5em@C=1em{
      & aaaa
      \ar@3 []!<5pt,-20pt>;[dd]!<5pt,20pt> ^{D}
      & aaast
      \ar@2 [l] _{aaa\beta}
      \\
      sasaa
      \ar@2@/^/ [ur] ^{\gamma aa}
      \ar@2@/_/ [dr] _{sa\delta}
      \\
      & saaat
      \ar@2 [r] _{\delta at}
      & aatat
      \ar@2@/_/ [uu] _{aa\alpha t}
    }
  }
\]
\end{example}

\subsection{Homotopical reduction}

\subsubsection{Generic homotopical reduction}
\label{Subsubsection:GenericHomotopicalReduction}

Let $\Sigma$ be a $(3,1)$-polygraph. A \emph{collapsible part of $\Sigma$} is a triple $\Gamma=(\Gamma_2,\Gamma_3,\Gamma_4)$ made of a family $\Gamma_2$ of $2$-cells of $\Sigma$, a family $\Gamma_3$ of $3$-cells of $\Sigma$ and a family~$\Gamma_4$ of $3$-spheres of $\tck{\Sigma}$, such that the following conditions are satisfied:
\begin{itemize}
\item every $\gamma$ of every $\Gamma_k$ is collapsible (potentially up to a Nielsen transformation),
\item no $\gamma$ of any $\Gamma_k$ is redundant for some element of $\Gamma_{k+1}$,
\item there exists well-founded order relations on the $1$-cells, $2$-cells and $3$-cells of $\Sigma$ such that, for every~$\gamma$ in every $\Gamma_k$, the target of $\gamma$ is strictly greater than every generating $(k-1)$-cell that occurs in the source of~$\gamma$.
\end{itemize}
In that case, the recursive assignment
\[
\pi_{\Gamma} (x) \:=\: 
\begin{cases}
\pi_{\Gamma} (s(\gamma)) &\text{if $x=t(\gamma)$ for $\gamma$ in $\Gamma$} \\
1_{\pi_{\Gamma}(s(\gamma))} &\text{if $x=\gamma$ is in $\Gamma$} \\
x &\text{otherwise}
\end{cases}
\]
defines a Tietze transformation $\pi_{\Gamma}:\tck{\Sigma} \fl \tck{\Sigma}/\Gamma$ by well-founded induction, called the \emph{homotopical reduction of $\Sigma$ with respect to $\Gamma$}. The target $(3,1)$-category is freely generated by the $(3,1)$-polygraph~$\Sigma/\Gamma$ obtained from $\Sigma$ by removing the cells of $\Gamma$ and of the corresponding redundant cells, and by replacement of the source and target maps of $\Sigma$ by their compositions with~$\pi_{\Gamma}$. Moreover, by construction, the $(3,1)$-polygraph $\Sigma/\Gamma$ is Tietze-equivalent to $\Sigma$.

\subsubsection{Generating triple confluences}
\label{Subsubsection:TripleConfluences}

The coherent elimination of $3$-cells of a $(3,1)$-polygraph~$\Sigma$ by homotopical reduction requires a collapsible set of $3$-spheres of $\tck{\Sigma}$. When $\Sigma$ is convergent and coherent, its triple critical branchings generate a convenient way to build such a set.

We recall from~\cite{GuiraudMalbos12} that a \emph{local triple branching} is a triple $(f,g,h)$ of rewriting steps with a common source. Like branchings, local triple branchings are classified into three families:
\begin{itemize}
\item \emph{aspherical} triple branchings have two of their $2$-cells equal,
\item \emph{Peiffer} triple branchings have at least one of their $2$-cells that form a Peiffer branching with the other two,
\item \emph{overlap} triple branchings are the remaining local triple branchings.
\end{itemize}
Local triple branchings are ordered by inclusion of their sources and a minimal overlap triple branching is called \emph{critical}. 

If $\Sigma$ is a coherent and convergent $(3,1)$-polygraph, a \emph{triple generating confluence of $\Sigma$} is a $3$-sphere
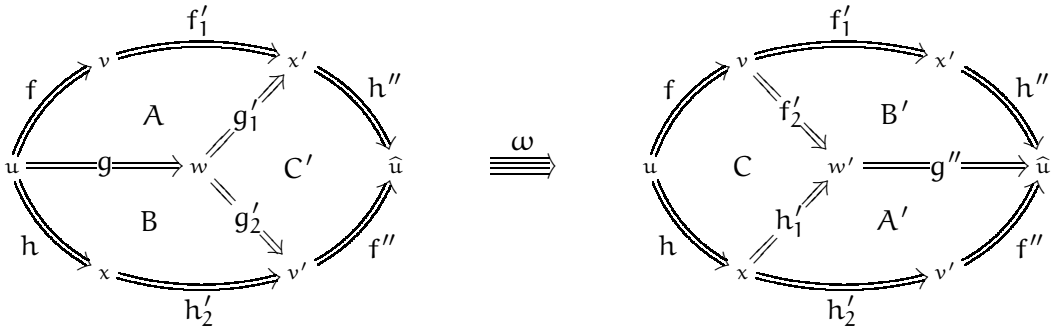
\[
\xymatrix {
& {\sm v} 
	\ar@2 @/^/ [rr] ^-{f'_1}
	\ar@{} [dr] |-{A}
&& {\sm x'}
	\ar@2 @/^/ [dr] ^-{h''}
&&&&& {\sm v}
	\ar@2 @/^/ [rr] ^-{f'_1}
	\ar@2 [dr] |-{f'_2}
&& {\sm x'}
	\ar@2 @/^/ [dr] ^-{h''}
\\
{\sm u}
	\ar@2 @/^/ [ur] ^-{f}
	\ar@2 [rr] |-{g}
	\ar@2 @/_/ [dr] _-{h}
&& {\sm w} 
	\ar@2 [ur] |-{g'_1}
	\ar@2 [dr] |-{g'_2}
	\ar@{} [rr] |-{C'}
&& {\sm \rep{u}}
& \strut
	\ar@4 [r] ^-*+{\omega}
&& {\sm u}
	\ar@2 @/^/ [ur] ^-{f}
		\ar@{} [rr] |-{C}
	\ar@2 @/_/ [dr] _-{h}
&& {\sm w'}
	\ar@2 [rr] |-{g''}
	\ar@{} [ur] |-{B'}
	\ar@{} [dr] |-{A'}
&& {\sm \rep{u}}
\\
& {\sm x}
	\ar@2 @/_/ [rr] _-{h'_2}
	\ar@{} [ur] |-{B}
&& {\sm v'}
	\ar@2 @/_/ [ur] _-{f''}
&&&&& {\sm x}
	\ar@2 [ur] |-{h'_1}
	\ar@2 @/_/ [rr] _-{h'_2}
&& {\sm v'}
	\ar@2 @/_/ [ur] _-{f''}
}
\]
where $(f,g,h)$ is a triple critical branching of $\Sigma$ and the other cells are obtained as follows. First, we consider the branching $(f,g)$: we use confluence to get $f'_1$ and $g'_1$ and coherence to get the $3$-cell~$A$. We proceed similarly with the branchings $(g,h)$ and $(f,h)$. Then, we consider the branching $(f'_1,f'_2)$ and we use convergence to get~$g''$ and~$h''$ with $\rep{u}$ as common target, plus the $3$-cell $B'$ by coherence. We do the same operation with $(h'_1,h'_2)$ to get $A'$. Finally, we build the $3$-cell $C'$ to relate the parallel $2$-cells $g'_1\star_1 h''$ and $g'_2\star_1 f''$. 

\subsubsection{Homotopical completion-reduction}
\label{Subsubsection:HomotopicalCompletionReduction}

In the applications we consider, homotopical reduction is applied to the homotopical completion $\Sr(\Sigma)$ of a terminating $2$-polygraph $\Sigma$. This induces a collapsible part $\Gamma$ of $\Sr(\Sigma)$ made of
\begin{itemize}
\item some of the generating triple confluences of $\Sr(\Sigma)$,
\item the $3$-cells coherently adjoined with a $2$-cell by homotopical completion to reach confluence, 
\item some collapsible $2$-cells or $3$-cells already present in the initial presentation~$\Sigma$.
\end{itemize}
If $\Sigma$ is a terminating $2$-polygraph, the \emph{homotopical completion-reduction of $\Sigma$} is the $(3,1)$-polygraph
\[
\Rr(\Sigma) \:=\: \pi_{\Gamma}(\Sr(\Sigma))
\]
obtained from the homotopical completion of $\Sigma$ by homotopical reduction with respect to some collapsible part $\Gamma$ of $\Sr(\Sigma)$. The definition and the notation should depend on $\Gamma$, and we make them precise in each application we consider.

\begin{theorem}
\label{Theorem:HomotopicalCompletionReduction}
For every terminating presentation $\Sigma$ of a category $\C$, the homotopical completion-reduction $\Rr(\Sigma)$ of $\Sigma$ is a coherent presentation of $\C$.
\end{theorem}

\begin{example}
In Example~\ref{Example:KNCompletion}, we have obtained a coherent convergent presentation $\Sr(\Sigma)$ of $\B_3^+$ by homotopical completion. We consider the collapsible part $\Gamma$ of $\Sr(\Sigma)$ consisting of the two generating triple confluences
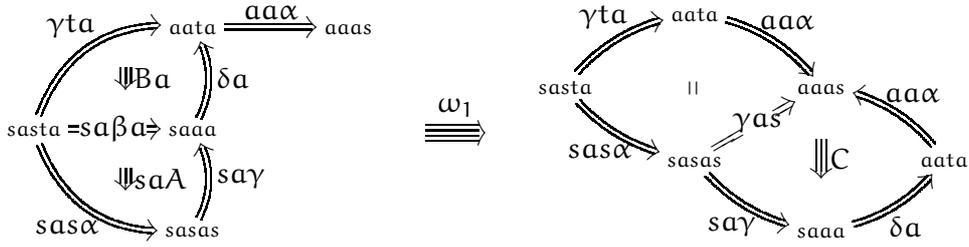
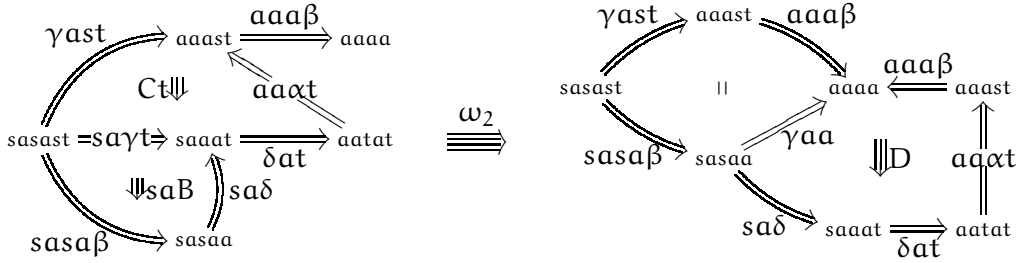
\[
\begin{array}{r@{\quad}c@{\quad}l}
  \vcenter{
    \xymatrix@R=2.5em@C=3em{
      & {\sm aata}
      \ar@2 [r] ^{aa\alpha}
      \ar@3 []!<-25pt,-15pt>;[d]!<-25pt,15pt> ^{Ba}
      & {\sm aaas}
      \\
      {\sm sasta}
      \ar@2@/^3ex/ [ur] ^{\gamma ta} _(0.68){}="src1"
      \ar@2 [r] |{sa\beta a} ^(0.574){}="tgt1" _(0.574){}="src2"
      \ar@2@/_3ex/ [dr] _{sas\alpha} ^(0.68){}="tgt2"
      & {\sm saaa}
      \ar@2@/_/ [u] _{\delta a}
      \ar@3 []!<-25pt,-15pt>;[d]!<-25pt,15pt> ^{saA}
      \\
      & {\sm sasas}
      \ar@2@/_/ [u] _{sa\gamma}
    }
  }
  &
  \vcenter{\xymatrix@C=2em{\strut \ar@4 [r] ^-*+{\omega_1} &\strut}}
  &
  \vcenter{
    \xymatrix@R=1.5em@C=2em{
      & {\sm aata}
      \ar@2@/^/ [dr] ^{aa\alpha}
      \\
      {\sm sasta}
      \ar@2@/^/ [ur] ^{\gamma ta}
      \ar@2@/_/ [dr] _{sas\alpha}
      \ar@{} [rr] |{\shortparallel}
      && {\sm aaas}
      \ar@3 []!<0pt,-20pt>;[dd]!<0pt,20pt> ^{C}
      \\
      & {\sm sasas}
      \ar@2 [ur] |{\gamma as}
      \ar@2@/_/ [dr] _{sa\gamma}
      && {\sm aata}
      \ar@2@/_/ [ul] _{aa\alpha}
      \\
      && {\sm saaa}
      \ar@2@/_/ [ur] _{\delta a}
    }
  }
\end{array}
\]
and
\[
\begin{array}{r@{\quad}c@{\quad}l}
  \vcenter{
    \xymatrix@R=2.5em@C=3em{
      & {\sm aaast}
      \ar@2 [r] ^{aaa\beta}
      \ar@3 []!<-10pt,-15pt>;[d]!<-10pt,15pt> _{Ct}
      & {\sm aaaa}
      \\
      {\sm sasast}
      \ar@2@/^3ex/ [ur] ^{\gamma ast}
      \ar@2 [r] |{sa\gamma t} _(0.582){}="src"
      \ar@2@/_3ex/ [dr] _{sasa\beta} ^(0.68){}="tgt"
      & {\sm saaat}
      \ar@2 [r] _{\delta at}
      \ar@3 []!<-25pt,-15pt>;[d]!<-25pt,15pt> ^{saB}
      & {\sm aatat}
      \ar@2 [ul] |{aa\alpha t}
      \\
      & {\sm sasaa}
      \ar@2@/_/ [u] _{sa\delta}
    }
  }
  &
  \vcenter{\xymatrix@C=2em{\strut \ar@4 [r] ^-*+{\omega_2} &\strut}}
  &
  \vcenter{
    \xymatrix@R=1.5em@C=1.9em{
      & {\sm aaast}
      \ar@2@/^/ [dr] ^{aaa\beta}
      \\
     {\sm  sasast}
      \ar@2@/^/ [ur] ^{\gamma ast}
      \ar@2@/_/ [dr] _{sasa\beta}
      \ar@{} [rr] |{\shortparallel}
      && {\sm aaaa}
      \ar@3 []!<10pt,-20pt>;[dd]!<10pt,20pt> ^{D}
      & {\sm aaast}
      \ar@2 [l] _{aaa\beta}
      \\
      & {\sm sasaa}
      \ar@2 [ur] _{\gamma aa}
      \ar@2@/_/ [dr] _{sa\delta}
      \\
      && {\sm saaat}
      \ar@2 [r] _{\delta at}
      & {\sm aatat}
      \ar@2 [uu] |{aa\alpha t}
    }
  }
\end{array}
\]
together with the $3$-cells $A$ and $B$ coherently adjoined with the $2$-cell $\gamma$ and $\delta$ during homotopical completion and the $2$-cell $\beta:st\dfl a$ that defines the redundant generator $a$. We have that $\omega_1$, $\omega_2$, $A$, $B$ and~$\beta$ are collapsible (up to a Nielsen transformation), with respective redundant cells $C$, $D$, $\gamma$, $\delta$ and $a$. We conclude that $\Gamma$ is collapsible with the orders
\[
D > C > B > A,
\qquad
\delta > \gamma > \beta > \alpha,
\qquad
a > t > s.
\]
Thus the homotopical reduction of $\Sr(\Sigma)$ with respect to $\Gamma$ is the $(3,1)$-polygraph
\[
\Rr(\Sigma) \:=\: \big( s, t \;;\; tst \dfl sts \;;\; \emptyset \big).
\]
By Theorem~\ref{Theorem:HomotopicalCompletionReduction}, we recover that the monoid $\B_3^+$ admits a coherent presentation made of Artin's presentation and no $3$-cell.
\end{example}

\section{Garside's coherent presentation of Artin monoids}
\label{Section:GarsideCoherentPresentation}

Recall that a \emph{Coxeter group} is a group $\W$ that admits a presentation with a finite set $S$ of generators and with one relation
\begin{equation}
(st)^{m_{st}} \:=\: 1, \qquad \text{with $m_{st}\in\Nb\amalg\ens{\infty}$},
\label{cox2}
\end{equation}
for every $s$ and $t$ in $S$, with the following requirements and conventions:
\begin{itemize}
\item $m_{st}=\infty$ means that there is, in fact, no relation between $s$ and $t$,
\item $m_{st}=1$ if, and only if, $s=t$.
\end{itemize}
The last requirement implies that $s^2=1$ holds in $\W$ for every $s$ in $S$. As a consequence, the group $\W$ can also be seen as the monoid with the same presentation. Let us note that a given Coxeter group can have several generating sets that fit the given scheme, but we always assume that such a set $S$ has been fixed and comes equipped with a total order.

Following~\cite[(1.1)]{BrieskornSaito72}, we denote by $\angle{st}^n$ the element of length $n$ in the free monoid $S^*$, obtained by multiplication of alternating copies of $s$ and $t$. Formally, this element is defined by induction on $n$ as follows:
\[
\angle{st}^0 \:=\: 1
\qquad\text{and}\qquad
\angle{st}^{n+1} \:=\: s\angle{ts}^n.
\]
When $s\neq t$ and $m_{st}<\infty$, we use this notation and the relations $s^2=t^2=1$ to write~\eqref{cox2} as a \emph{braid relation}:
\begin{equation}
\angle{st}^{m_{st}} \:=\: \angle{ts}^{m_{st}}.
\label{cox2'}
\end{equation}

A \emph{reduced expression} of an element $u$ of $\W$ is a representative of minimal length of $u$ in the free monoid~$S^*$. The \emph{length of $u$} is denoted by $l(u)$ and defined as the length of any of its reduced expressions. The Coxeter group $\W$ is finite if, and only if, it admits an element of maximal length,~\cite[Theorem~5.6]{BrieskornSaito72}; in that case, this element is unique, it is called the \emph{longest element of $\W$} and is denoted by $w_0(S)$. For $I\subseteq S$, the subgroup of $\W$ spanned by the elements of $I$ is denoted by $\W_I$. It is a Coxeter group with generating set $I$. If $\W_I$ is finite, we denote by $w_0(I)$ its longest element.

We recall that the \emph{Artin monoid} associated to $\W$ is the monoid denoted by $\B^+(\W)$, generated by $S$ and subject to the braid relations~\eqref{cox2'}. This presentation, seen as a $2$-polygraph, is denoted by $\Art_2(\W)$ and called \emph{Artin's presentation}: this is the same as the one of $\W$, except for the relations $s^2=1$. 

In this section, we fix a Coxeter group $\W$ and we apply the homotopical completion-reduction method to get a coherent presentation for the Artin monoid $\B^+(\W)$. 

\subsection{Garside's presentation of Artin monoids}
\label{Subsection:Recollections}

We recall some arithmetic properties on Artin monoids, observed by Garside for braid monoids in~\cite{Garside69} and generalised by Brieskorn and Saito in~\cite{BrieskornSaito72}. Garside's presentation is explicitly given in~\cite[1.4.5]{Deligne97} for spherical Artin monoids and in~\cite[Proposition~1.1]{Michel99} for any Artin monoid. We refer to~\cite{GeckPfeiffer00} for proofs.

\subsubsection{Length notation and divisibility}

For every $u$ and $v$ in $\W$, we have $l(uv)\leq l(u)+l(v)$ and we use distinct graphical notations depending on whether the equality holds or not:
\[
\typedeux{1} 
	\quad\Leftrightarrow\quad
l(uv)=l(u)+l(v),
\]
\[
\typedeux{0}
	\quad\Leftrightarrow\quad
l(uv)<l(u)+l(v).
\]
When $w=uv$ holds in $\W$ with \typedeux{1}, we write $w\doteq uv$. We generalise the notation for a greater number of elements of $\W$. For example, in the case of three elements $u$, $v$ and $w$ of $\W$, we write \typetrois{?} when both equalities $l(uv)=l(u)+l(v)$ and $l(vw)=l(v)+l(w)$ hold. This case splits in the following two mutually exclusive subcases:
\[
\typetrois{1} 
	\quad\Leftrightarrow\quad
\begin{cases}
\:\typetrois{} \\
l(uvw)=l(u)+l(v)+l(w),
\end{cases}
\]
\[
\typetrois{0}
	\quad\Leftrightarrow\quad
\begin{cases}
\:\typetrois{} \\
l(uvw)<l(u)+l(v)+l(w).
\end{cases}
\]
If $u$ and $v$ are two elements of $\B^+(\W)$, we say that~\emph{$u$ is a divisor of $v$} and that \emph{$v$ is a multiple of $u$} if there exists an element $u'$ in $\B^+(\W)$ such that $uu'=v$. In that case, the element~$u'$ is uniquely defined and called the \emph{complement of $u$ in $v$}~\cite[Proposition~2.3]{BrieskornSaito72}. Moreover, if $v$ is in~$\W$, seen as an element of $\B^+(\W)$ by the canonical embedding (given by Matsumoto's theorem, see~\cite[Theorem~1.2.2]{GeckPfeiffer00}), then we also have~$u$ and $u'$ in $\W$ and $uu'\doteq v$. If two elements $u$ and $v$ of $\B^+(\W)$ have a common multiple, then they have a least common multiple, lcm for short~\cite[Proposition~4.1]{BrieskornSaito72}.

\subsubsection{Garside's coherent presentation}
\label{Subsubsection:GarsideCoherentPresentation}

Let $\W$ be a Coxeter group. We call \emph{Garside's presentation of $\B^+(\W)$} the $2$-polygraph $\Gar_2(\W)$ whose $1$-cells are the elements of $\W\setminus\ens{1}$ and with one $2$-cell
\[
\alpha_{u,v} \::\: u|v \:\dfl \: uv
\]
whenever $l(uv)=l(u)+l(v)$ holds. Here, we write $uv$ for the product in $\W$ and $u|v$ for the product in the free monoid over $\W$. 
We denote by $\Gar_3(\W)$ the extended presentation of $\B^+(\W)$ obtained from $\Gar_2(\W)$ by adjunction of one $3$-cell
\[
\xymatrix @!C @R=1.5em {
& uv|w
	\ar@2 @/^/ [dr] ^-{\alpha_{uv,w}}
	\ar@3 []!<0pt,-25pt>;[dd]!<0pt,25pt> ^-{A_{u,v,w}}
\\
u|v|w
	\ar@2 @/^/ [ur] ^-{\alpha_{u,v}|w}
	\ar@2 @/_/ [dr] _-{u|\alpha_{v,w}}
&& uvw
\\
& u|vw
	\ar@2 @/_/ [ur] _-{\alpha_{u,vw}}
}
\]
for every $u$, $v$ and $w$ of $\W\setminus\ens{1}$ with \typetrois{1}. 

\begin{theorem}
\label{Theorem:GarsideCoherentPresentation}
For every Coxeter group $\W$, the Artin monoid $\B^+(\W)$ admits $\Gar_3(\W)$ as a coherent presentation.
\end{theorem}

\noindent
The $(3,1)$-polygraph $\Gar_3(\W)$ is called the \emph{Garside's coherent presentation} of the Artin monoid $\B^+(\W)$. Theorem \ref{Theorem:GarsideCoherentPresentation} is proved in the following section by homotopical completion-reduction of $\Gar_2(\W)$.

\subsection{Homotopical completion-reduction of Garside's presentation} 
\label{Subsection:GarsideCompletionReduction}

Let us define a termination order on the $2$-polygraph $\Gar_2(\W)$. Let $<$ denote the strict order on the elements of the free monoid $\W^*$ that first compares their length as elements of $\W^*$, and then the length of their components, starting from the right. For example, we have that $u_1|u_2<v_1|v_2|v_3$ (first condition) and $uv|w < u|vw$ if \typetrois{} (second condition). The order relation $\leq$ generated by $<$ by adding reflexivity is a termination order on $\Gar_2(\W)$: for every $2$-cell $\alpha_{u,v}$ of $\Gar_2(\W)$, we have $u|v>uv$. Hence the $2$-polygraph $\Gar_2(\W)$ terminates, so that its homotopical completion is defined.

\begin{proposition}
\label{Proposition:GarsideCoherentConvergentPresentation}
For every Coxeter group $\W$, the Artin monoid $\B^+(\W)$ admits, as a coherent convergent presentation, the $(3,1)$-polygraph $\Sr(\Gar_2(\W))$ with one $0$-cell, one $1$-cell for every element of $\W\setminus\ens{1}$, the $2$-cells
\[
\xymatrix @C=2.5em @!C @R=1.5em {
u|v 
	\ar@2 [r] ^-*+{\alpha_{u,v}} 
& uv
} 
\qquad\text{and}\qquad
\xymatrix @C=2.5em @!C @R=1.5em {
u|vw 
	\ar@2 [r] ^-*+{\beta_{u,v,w}} 
& uv|w,
} 
\]
respectively for every $u$, $v$ of $\W\setminus\ens{1}$ with \typedeux{1} and every $u$, $v$, $w$ of $\W\setminus\ens{1}$ with \typetrois{0}, and the nine families of $3$-cells $A$, $B$, $C$, $D$, $E$, $F$, $G$, $H$, $I$ given in Figure~\ref{Figure:SquierGarside3Cells}.
\end{proposition}

\begin{figure}[!ht]
\begin{center}
$\vcenter{\xymatrix @!C @R=1.5em {
& uv|w
	\ar@2 @/^/ [dr] ^-{\alpha_{uv,w}}
	\ar@3 []!<0pt,-25pt>;[dd]!<0pt,25pt> ^-{A_{u,v,w}}
\\
u|v|w
	\ar@2 @/^/ [ur] ^-{\alpha_{u,v}|w}
	\ar@2 @/_/ [dr] _-{u|\alpha_{v,w}}
&& uvw
\\
& u|vw
	\ar@2 @/_/ [ur] _-{\alpha_{u,vw}}
}}$
\hspace{\stretch{1}}
$\vcenter{\xymatrix @!C {
u|v|w
	\ar@2 @/^3ex/ [rr] ^-{\alpha_{u,v}|w} ^{}="src"
	\ar@2 @/_/ [dr] _-{u|\alpha_{v,w}}
&& uv|w
\\
& u|vw
	\ar@2 @/_/ [ur] _-{\beta_{u,v,w}}
\ar@3 "src"!<0pt,-15pt>;[]!<0pt,20pt> ^-{B_{u,v,w}}
}}$
\hspace{\stretch{1}}
$\vcenter{\xymatrix @!C @R=1.5em {
& uv|wx
	\ar@2 @/^/ [dr] ^-{\beta_{uv,w,x}}
	\ar@3 []!<0pt,-25pt>;[dd]!<0pt,25pt> ^-{C_{u,v,w,x}}
\\
u|v|wx
	\ar@2 @/^/ [ur] ^-{\alpha_{u,v}|wx}
	\ar@2 @/_/ [dr] _-{u|\beta_{v,w,x}}
&& uvw|x
\\
& u|vw|x
	\ar@2 @/_/ [ur] _-{\alpha_{u,vw}|x}
}}$
\end{center}

\begin{center}
$\vcenter{\xymatrix @!C @C=3em {
u|v|wx
	\ar@2 @/^3ex/ [rrr] ^-{\alpha_{u,v}|wx} ^{}="s"
	\ar@2 @/_/ [dr] _-{u|\beta_{v,w,x}}
&&& uv|wx
\\
& u|vw|x
	\ar@2 [r] _-{\beta_{u,v,w}|x} _{}="t"
& uv|w|x
	\ar@2 @/_/ [ur] _-{uv|\alpha_{w,x}}
\ar@3 "s"!<0pt,-15pt>;"t"!<0pt,15pt> ^-{D_{u,v,w,x}}
}}$
\hspace{\stretch{1}}
$\vcenter{\xymatrix @!C @R=1.5em {
& uv|w|x
	\ar@2 @/^/ [dr] ^-{uv|\alpha_{w,x}}
\ar@3 []!<0pt,-25pt>;[dd]!<0pt,25pt> ^-{E_{u,v,w,x}}
\\
u|vw|x
	\ar@2 @/^/ [ur] ^-{\beta_{u,v,w}|x}
	\ar@2 @/_/ [dr] _-{u|\alpha_{vw,x}}
&& uv|wx
\\
& u|vwx
	\ar@2 @/_/ [ur] _-{\beta_{u,v,wx}}
}}$
\end{center}

\begin{center}
$\vcenter{\xymatrix @!C @C=1em @R=1.5em {
&& uv|w|xy
	\ar@2 @/^/ [drr] ^-{uv|\alpha_{w,xy}}
\\
u|vw|xy
	\ar@2 @/^/ [urr] ^-{\beta_{u,v,w}|xy}
	\ar@2 @/_/ [dr] _-{u|\beta_{vw,x,y}}
&&&& uv|wxy
\\
& u|vwx|y
	\ar@2 [rr] _-{\beta_{u,v,wx}|y} _-{}="t"
&& uv|wx|y
	\ar@2 @/_/ [ur] _-{uv|\alpha_{wx,y}}
\ar@3 "1,3"!<0pt,-25pt>;"t"!<0pt,25pt> ^-{F_{u,v,w,x,y}}
}}$
\hspace{\stretch{1}}
$\vcenter{\xymatrix @!C @R=1.5em {
& uv|w|xy
	\ar@2 @/^/ [dr] ^-{uv|\beta_{w,x,y}}
\ar@3 []!<0pt,-25pt>;[dd]!<0pt,25pt> ^-{G_{u,v,w,x,y}}
\\
u|vw|xy
	\ar@2 @/^/ [ur] ^-{\beta_{u,v,w}|xy}
	\ar@2 @/_/ [dr] _-{u|\beta_{vw,x,y}}
&& uv|wx|y
\\
& u|vwx|y
	\ar@2 @/_/ [ur] _-{\beta_{u,v,wx}|y}
}}$
\end{center}

\begin{center}
\strut\hspace{\stretch{1}}
$\vcenter{\xymatrix @!C {
& uv | xy
	\ar@2 @/^/ [dr] ^-{\beta_{uv,x,y}}
\\
u | vxy 
	\ar@2 @/^/ [ur] ^-{\beta_{u,v,xy}}
	\ar@2 @/_3ex/ [rr] _-{\beta_{u,vx,y}} _{}="t"
&& uvx | y
\ar@3 "1,2"!<0pt,-20pt>;"t"!<0pt,15pt> ^-{H_{u,v,x,y}}
}}$
\hspace{\stretch{2}}
$\vcenter{\xymatrix @R=1.5em {
& uv_1 | w_1 \:=\: uv_1 | x_1y
	\ar@2 @/^3ex/ [dr] ^-{\beta_{uv_1,x_1,y}}
	\ar@3 []!<0pt,-30pt>;[dd]!<0pt,30pt> ^-{I_{u,v_1,w_1,v_2,w_2}}
\\
*\txt{$u | v_1w_1$ \\ $=$ \\ $u | v_2w_2$}
	\ar@2 @/^3ex/ [ur] ^-{\beta_{u,v_1,w_1}}
	\ar@2 @/_3ex/ [dr] _-{\beta_{u,v_2,w_2}}
&& *\txt{$uv_1x_1 | y$ \\ $=$ \\ $uv_2x_2 | y$}
\\
& uv_2 | w_2 \:=\: uv_2 | x_2y
	\ar@2 @/_3ex/ [ur] _-{\beta_{uv_2,x_2,y}}
}}$
\hspace{\stretch{1}}\strut
\end{center}
\caption{The $3$-cells of the homotopical completion of Garside's presentation}
\label{Figure:SquierGarside3Cells}
\end{figure}
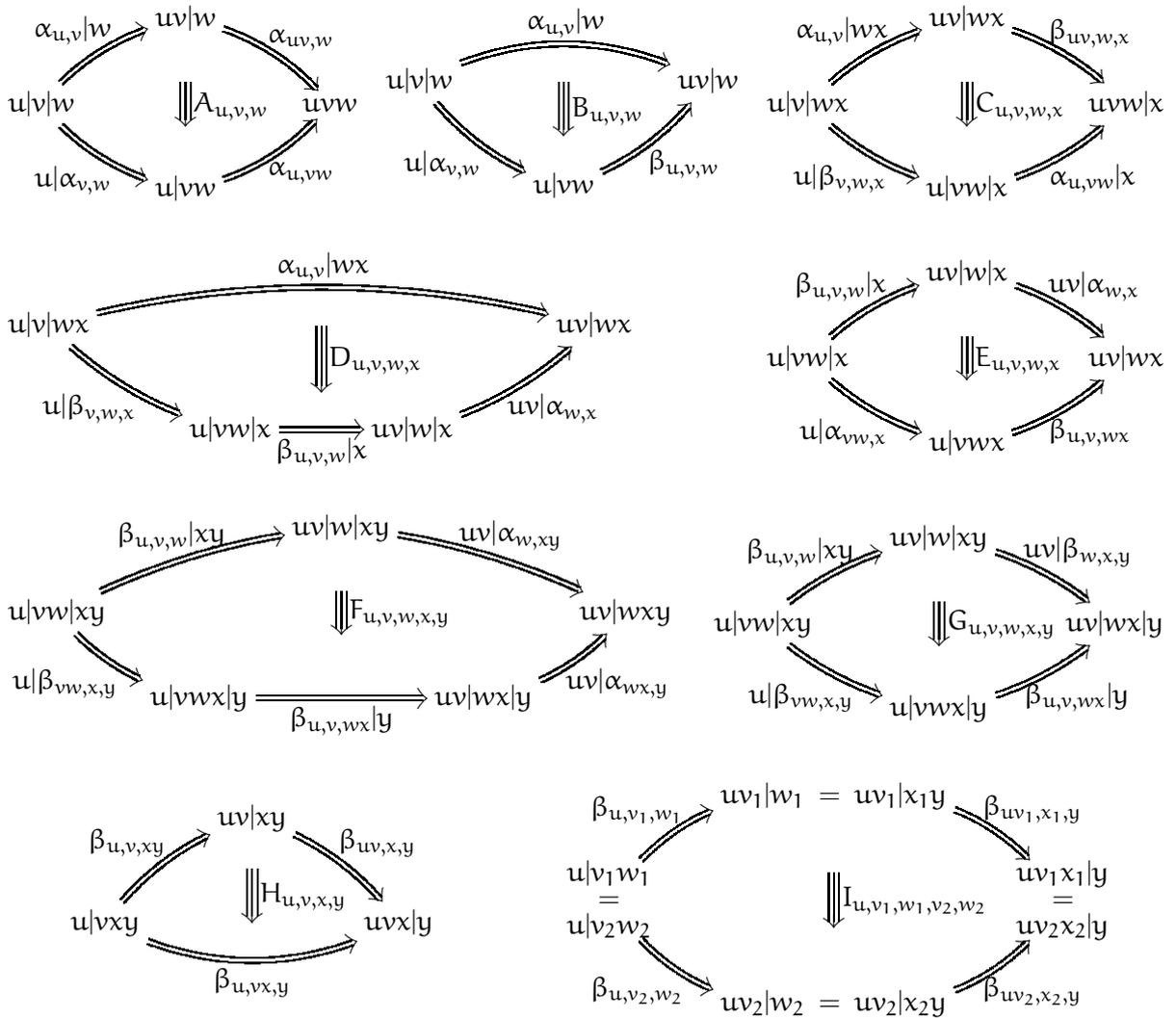

The $3$-cells of Figure~\ref{Figure:SquierGarside3Cells} are families indexed by all the possible elements of $\W\setminus\ens{1}$, deduced by the involved $2$-cells. For example, there is one $3$-cell $A_{u,v,w}$ for every $u$, $v$, $w$ with \typetrois{1}, and one $3$-cell $F_{u,v,w,x,y}$ for every $u$, $v$, $w$, $x$, $y$ with \typecinq{0}{1}{1}{}{0}{}.

\begin{proof}
The $2$-polygraph $\Gar_2(\W)$ has exactly one critical branching for every $u$, $v$ and~$w$ of $\W\setminus\ens{1}$ such that \typetrois{?}:
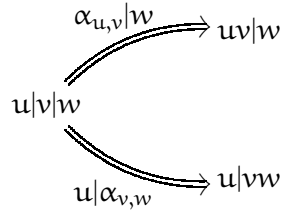
\[
\xymatrix @!C @C=4em @R=1em {
& uv|w
\\
u|v|w
	\ar@2 @/^2.5ex/ [ur] ^-{\alpha_{u,v}|w}
	\ar@2 @/_2.5ex/ [dr] _-{u|\alpha_{v,w}}
\\
& u|vw
}
\]
Then there are two possibilities. If \typetrois{1}, the branching is confluent, adjoining the $3$-cell $A_{u,v,w}$. Otherwise, we have \typetrois{0} and the branching is not confluent, thus homotopical completion coherently adjoins the $2$-cell $\beta_{u,v,w}$ and the $3$-cell $B_{u,v,w}$. The family $\beta$ of $2$-cells creates new critical branchings, each one being confluent and conducting to the adjunction of one or several $3$-cells. The sources of all the $2$-cells $\alpha$ and $\beta$ have size~$2$ in the free monoid over $\W\setminus\ens{1}$. As a consequence, there are two main cases for the critical branchings that involve at least one $2$-cell~$\beta$. 

The first case occurs when the sources of the $2$-cells of $\Gar_2(\W)$ that generate the branching overlap on one element of $\W\setminus\ens{1}$. The source of such a branching has size $3$, with one $2$-cell of the branching reducing the leftmost two generating $1$-cells and the other one reducing the rightmost two. This leaves three main cases of branchings:
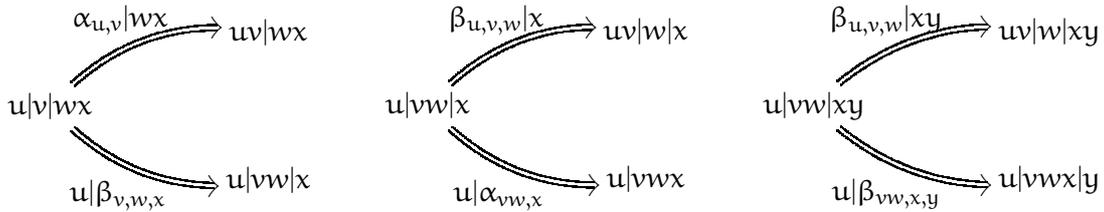
\[
\vcenter{\xymatrix @!C @C=4em @R=1em {
& uv | wx 
\\
u | v | wx
	\ar@2 @/^2.5ex/ [ur] ^-{\alpha_{u,v} | wx}
	\ar@2 @/_2.5ex/ [dr] _-{u | \beta_{v,w,x}}
\\
& u | vw | x
}}
\qquad
\vcenter{\xymatrix @!C @C=4em @R=1em {
& uv | w | x 
\\
u | vw | x
	\ar@2 @/^2.5ex/ [ur] ^-{\beta_{u,v,w} | x}
	\ar@2 @/_2.5ex/ [dr] _-{u | \alpha_{vw,x}}
\\
& u | vwx
}}
\qquad
\vcenter{\xymatrix @!C @C=4em @R=1em {
& uv | w | xy 
\\
u | vw | xy
	\ar@2 @/^2.5ex/ [ur] ^-{\beta_{u,v,w} | xy}
	\ar@2 @/_2.5ex/ [dr] _-{u | \beta_{vw,x,y}}
\\
& u | vwx | y
}}
\]
The first branching occurs when \typequatre{?}{0}{}, splitting into the two disjoint possibilities \typequatre{1}{0}{} and \typequatre{0}{0}{}, respectively corresponding to the $3$-cells $C_{u,v,w,x}$ and $D_{u,v,w,x}$. The second branching appears when \typequatre{0}{1}{} and corresponds to the $3$-cell $E_{u,v,w,x}$. The third branching happens when \typecinq{0}{1}{?}{}{}{}, with the extra condition that $l(vwxy) < l(vw) + l(xy)$ since $vw|xy$ is the source of the $2$-cell $\beta_{vw,x,y}$: this situation splits into the two disjoint possibilities \typecinq{0}{1}{1}{}{0}{} and \typecinq{0}{1}{0}{}{}{}, respectively corresponding to the $3$-cells $F_{u,v,w,x,y}$ and $G_{u,v,w,x,y}$.

The second main case occurs when the $2$-cells of $\Gar_2(\W)$ that generate the branching have the same source. Since one of those $2$-cells must be a $\beta$, the source must have shape $u|v_1w_1$ with \typetroisbase{$u$}{$v_1$}{$w_1$}, preventing the other $2$-cell to be an $\alpha$. The only remaining possibility is to have a different decomposition $v_1w_1=v_2w_2$, with \typetroisbase{$u$}{$v_2$}{$w_2$}{0}, so that the branching is as follows:
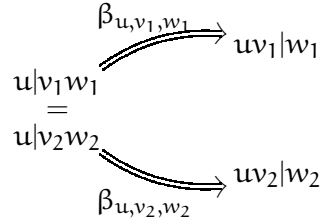
\[
\xymatrix @!C @C=4em @R=0em {
& uv_1 | w_1
\\
*\txt{$u | v_1w_1$ \\ $=$ \\ $u | v_2w_2$}
	\ar@2 @/^3ex/ [ur] ^-{\beta_{u,v_1,w_1}}
	\ar@2 @/_3ex/ [dr] _-{\beta_{u,v_2,w_2}}
\\
& uv_2 | w_2
}
\] 
The properties of Artin monoids ensure that we have the following relations in $\B^+(\W)$:
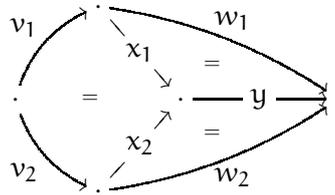
\[
\xymatrix @C=2em {
& \cdot
	\ar@/^/ [drrr] ^-{w_1}
	\ar [dr] |-*+{x_1} 
	\ar@{} [drrr] _{\sm =}
\\
\cdot
	\ar@/^/ [ur] ^-{v_1}
	\ar@/_/ [dr] _-{v_2}
	\ar@{} [rr] |(.45){\sm =}
&& \cdot
	\ar [rr] |-*+{y}
&& \cdot 
\\
& \cdot
	\ar [ur] |-*+{x_2}
	\ar@/_/ [urrr] _-{w_2}
	\ar@{} [urrr] ^{\sm =}
}
\]
Indeed, we note that the elements $v_1$ and $v_2$ have a common multiple since $v_1w_1=v_2w_2$. Hence, they admit an lcm. The elements $x_1$ and $x_2$ are respectively defined as the complements of $v_1$ and $v_2$ in their lcm. The element $y$ is the complement of the lcm $v_1x_1=v_2x_2$ of $v_1$ and $v_2$ in their common multiple $v_1w_1=v_2w_2$. By uniqueness of the complements of $v_1$ and $v_2$ in $v_1w_1=v_2w_2$, we get $w_1=x_1y$ and $w_2=x_2y$. Moreover, we have \typetroisbase{$v_1$}{$x_1$}{$y$}{1} and \typetroisbase{$v_2$}{$x_2$}{$y$}{1}. Finally, from the hypothesis \typetroisbase{$u$}{$v_1$}{$w_1$}, we get that $y\neq 1$. Then, there are two possible subcases for the confluence diagram, depending on $x_1$ and $x_2$. The first subcase is when we have either $x_1=1$ or $x_2=1$. We note that both cannot happen at the same time, otherwise $v_1=v_2$ and $w_1=w_2$, so that the branching would be aspherical and not critical. We get the $3$-cell $H_{u,v,x,y}$ if $x_2=1$, inducing $v_2=v_1x_1$, $w_1=x_1y$ and $w_2=y$, with $v=v_1$ and $x=x_1$. The second subcase, when $x_1\neq 1$ and $x_2\neq 1$ gives the $3$-cell $I_{u,v_1,w_1,v_2,w_2}$.
\end{proof}

\subsubsection{Homotopical reduction of \pdf{\Sr(\Gar_2(\W))}}

We consider the following generating triple confluences, associated to some of the triple critical branchings of $\Sr(\Gar_2(\W))$:
\renewcommand{\objectstyle}{\scriptstyle}
\renewcommand{\labelstyle}{\scriptstyle}
 \begin{itemize}
\item The $3$-sphere $\omega^C_{u,v,w,x}$ in the case \typequatre{1}{0}{}:
\[
\!\!\!\!\!\!\!\!\!\!
\!\!\!\!\!\!\!
\vcenter{\xymatrix @!C @C=1em{
& uv | w | x
	\ar@2@/^/ [rr] ^-{\alpha_{uv,w} | x}
	\ar@{} [dr] |{\displaystyle A_{u,v,w} | x}
&& uvw | x
\\
u | v | w | x
	\ar@2@/^/ [ur] ^{\alpha_{u,v} | w | x}
	\ar@2 [rr] |-{u | \alpha_{v,w} | x} ^-{}="s"
	\ar@2@/_/ [dr] _{u | v | \alpha_{w,x}}
&& u | vw | x
	\ar@2@/_/ [ur] _{\alpha_{u,vw} | x}
\\
& u | v | wx
	\ar@2@/_/ [ur] _{u | \beta_{v,w,x}} 
\ar@{} "s";[] |(0.4){\displaystyle u|B_{v,w,x}}
}}
\quad\qfl\quad
\vcenter{\xymatrix @!C @C=1em{
& uv | w | x
	\ar@2 [dr] |-{uv | \alpha_{w,x}}
	\ar@2 @/^3ex/ [drrr] ^-{\alpha_{uv,w}|x} ^-{}="s"
	\ar@{} [dd] |-{\displaystyle =}
\\
u | v | w | x
	\ar@2@/^/ [ur] ^-{\alpha_{u,v} | w | x}
	\ar@2@/_/ [dr] _-{u | v | \alpha_{w,x}}
&& uv | wx
	\ar@2 [rr] |-{\beta_{uv,w,x}}
	\ar@{} [dr] |-{\displaystyle C_{u,v,w,x}}
&& uvw | x
\\
& u | v | wx
	\ar@2 [ur] |-{\alpha_{u,v} | wx}
	\ar@2@/_/ [rr] _-{u | \beta_{v,w,x}} 
&& u | vw | x
	\ar@2@/_/ [ur] _-{\alpha_{u,vw} | x}
\ar@{} "s";"2,3" |(0.33){\displaystyle B_{uv,w,x}}
}}
\]
 
\item The $3$-sphere $\omega^D_{u,v,w,x}$ in the case \typequatre{0}{0}{}:
\[
\!\!\!\!\!\!\!\!\!\!
\!\!\!\!\!\!\!
\vcenter{\xymatrix @!C @C=1em{
& uv | w | x
	\ar@2@/^/ [rr] ^-{uv | \alpha_{w,x}}
&& uv | wx
\\
u | v | w | x
	\ar@2@/^/ [ur] ^-{\alpha_{u,v} | w | x}
	\ar@2 [rr] |-{u | \alpha_{v,w} | x} ^-{}="s" _-{}="t"
	\ar@2@/_/ [dr] _-{u | v | \alpha_{w,x}}
&& u | vw | x
	\ar@2@/_/ [ul] _(0.33){\beta_{u,v,w} | x}
\\
& u | v | wx
	\ar@2@/_/ [ur] _-{u | \beta_{v,w,x}} 
\ar@{} "1,2";"t" |(0.66){\displaystyle B_{u,v,w} | x}
\ar@{} "s";"3,2" |(0.33){\displaystyle u | B_{v,w,x}}
}}
\quad\qfl\quad
\vcenter{\xymatrix@C=1em{
& uv | w | x
	\ar@2@/^/ [dr] ^-{uv | \alpha_{w,x}}
	\ar@{} [dd] |-{\displaystyle =}
\\
u | v | w | x
	\ar@2@/^/ [ur] ^{\alpha_{u,v} | w | x}
	\ar@2@/_/ [dr] _-{u | v | \alpha_{w,x}}
&& uv | wx
	\ar@{} [dr] |-{\displaystyle D_{u,v,w,x}}
&& uv | w | x
	\ar@2@/_/ [ll] _-{uv | \alpha_{w,x}}
\\
& u | v | wx
	\ar@2 [ur] |-{\alpha_{u,v} | wx}
	\ar@2@/_/ [rr] _-{u | \beta_{v,w,x}} 
&& u | vw | x
	\ar@2@/_/ [ur] _-{\beta_{u,v,w} | x}
}}
\]

\item The $3$-sphere $\omega^E_{u,v,w,x}$ in the case \typequatre{0}{1}{}:
\[
\!\!\!\!\!\!\!\!\!\!
\!\!\!\!\!\!\!
\vcenter{\xymatrix @!C @C=1em{
& uv | w | x
	\ar@2 @/^3ex/ [drrr] ^-{uv | \alpha_{w,x}}
\\
u | v | w | x
	\ar@2@/^/ [ur] ^-{\alpha_{u,v} | w | x}
	\ar@2 [rr] |-{u | \alpha_{v,w} | x} _-{}="t"
	\ar@2@/_/ [dr] _-{u | v | \alpha_{w,x}}
	\ar@{} "1,2";"t" |(0.75){\displaystyle B_{u,v,w} | x} 
&& u | vw | x
	\ar@2 [ul] |-{\beta_{u,v,w} | x}
	\ar@2 [dr] |-{u | \alpha_{vw,x}}
	\ar@{} [dl] |-{\displaystyle u | A_{v,w,x}}
	\ar@{} [rr] ^(0.4){\displaystyle E_{u,v,w,x}}
&& uv | wx
\\
& u | v | wx
	\ar@2@/_/ [rr] _-{u | \alpha_{v,wx}} 
&& u | vwx
	\ar@2@/_/ [ur] _-{\beta_{u,v,wx}}
}}
\quad\qfl\quad
\vcenter{\xymatrix @!C @C=1em{
& uv | w | x
	\ar@2@/^/ [dr] ^-{uv | \alpha_{w,x}}
	\ar@{} [dd] |-{\displaystyle =}
&& 
\\
u | v | w | x
	\ar@2@/^/ [ur] ^-{\alpha_{u,v} | w | x}
	\ar@2@/_/ [dr] _-{u | v | \alpha_{w,x}}
&& uv | wx
\\
& u | v | wx
	\ar@2 [ur] |-{\alpha_{u,v} | wx}
	\ar@2@/_/ [rr] _-{u | \alpha_{v,wx}} _-{}="t"
	\ar@{} "2,3";"t" |(0.75){\displaystyle B_{u,v,wx} }	
&& u | vwx
	\ar@2@/_/ [ul] _-{\beta_{u,v,wx}}
}}
\]

\item The $3$-sphere $\omega^F_{u,v,w,x,y}$ in the case \typecinq{0}{1}{1}{}{0}{}:
\[
\!\!\!\!\!\!\!\!\!\!
\!\!\!\!\!\!\!
\vcenter{\xymatrix @!C @C=1em{
& uv | w | x | y
	\ar@2@/^/ [rr] ^-{uv | \alpha_{w,x} | y}
	\ar@{} [dr] |(0.4){\displaystyle E_{u,v,w,x} | y}
&& uv | wx | y
	\ar@2@/^/ [d] ^-{uv | \alpha_{wx,y}}
\\
u | vw | x | y
	\ar@2@/^/ [ur] ^-{\beta_{u,v,w} | x | y}
	\ar@2 [rr] |-{u | \alpha_{vw,x} | y} ^-{}="s"
	\ar@2@/_/ [dr] _-{u | vw | \alpha_{x,y}}
&& u | vwx | y
	\ar@2 [ur] |(0.4){\beta_{u,v,wx} | y}
& uv | wxy
\\
& u | vw | xy
	\ar@2@/_/ [ur] _-{u | \beta_{vw,x,y}} 
\ar@{} "s";[] |(0.33){\displaystyle u|B_{vw,x,y}}
}}
\:\qfl\:
\vcenter{\xymatrix@C=1em{
& uv | w | x | y
	\ar@2 [rr] ^-{uv | \alpha_{w,x} | y}
	\ar@2 [dr] |-{uv | w | \alpha_{x,y}}
&& uv | wx | y
	\ar@2 [dr] |-{uv | \alpha_{wx,y}}
	\ar@{} [dl] |-{\displaystyle uv | A_{w,x,y}}
\\
u | vw | x | y
	\ar@2@/^/ [ur] ^-{\beta_{u,v,w} | x | y}
	\ar@2@/_/ [dr] _-{u | vw | \alpha_{x,y}}
	\ar@{} [rr] |-{\displaystyle =}
&& uv | w | xy
	\ar@2 [rr] |-{uv | \alpha_{w,xy}}
&& uv | wxy
\\
& u | vw | xy
	\ar@2 [ur] |-{\beta_{u,v,w} | xy}
	\ar@2@/_/ [dr] _-{u | \beta_{vw,x,y}}
\\
&& u | vwx | y
	\ar@2 [rr] _-{\beta_{u,v,wx} | y} _-{}="t"
&& uv | wx | y
	\ar@2 [uu] |-{uv | \alpha_{wx,y}}
\ar@{} "2,3";"4,3" ^-{\displaystyle F_{u,v,w,x,y}}
}}
\]

\item The $3$-sphere $\omega^G_{u,v,w,x,y}$ in the case \typecinq{0}{1}{0}{}{}{}:
\[
\!\!\!\!\!\!\!\!\!\!
\!\!\!\!\!\!\!
\vcenter{\xymatrix @!C @C=1em{
& uv | w | xy
	\ar@2@/^/ [rr] ^-{uv | \beta_{w,x,y}}
&& uv | wx | y
\\
u | v | w | xy
	\ar@2@/^/ [ur] ^-{\alpha_{u,v} | w | xy}
	\ar@2 [rr] |-{u | \alpha_{v,w} | xy} _-{}="t"
	\ar@2@/_/ [dr] _-{u | v | \beta_{w,x,y}}
&& u | vw | xy
	\ar@2 [ul] _-{\beta_{u,v,w} | xy}
	\ar@2 [dr] |-{u | \beta_{vw,x,y}}
	\ar@{} [dl] |-{\displaystyle u | C_{v,w,x,y}}
	\ar@{} [ur] |(0.33){\displaystyle G_{u,v,w,x,y}}
\\
& u | v | wx | y
	\ar@2@/_/ [rr] _-{u | \alpha_{v,wx} | y} 
&& u | vwx | y
 	\ar@2@/_/ [uu] _-{\beta_{u,v,wx} | y} 
 \ar@{} "1,2";"t" |(0.7){\displaystyle B_{u,v,w} | xy}
}}
\quad\qfl\quad
\vcenter{\xymatrix@C=1em{
& uv | w | xy
	\ar@2@/^/ [dr] ^-{uv | \beta_{w,x,y}}
	\ar@{} [dd] |-{\displaystyle =}
\\
u | v | w | xy
	\ar@2@/^/ [ur] ^-{\alpha_{u,v} | w | xy}
	\ar@2@/_/ [dr] _-{u | v | \beta_{w,x,y}}
&& uv | wx | y
\\
& u | v | wx | y
 	\ar@2 [ur] |-{\alpha_{u,v} | wx | y}
	\ar@2@/_/ [rr] _-{u | \alpha_{v,wx} | y} _-{}="t"
&& u | vwx | y
	\ar@2@/_/ [ul] _-{\beta_{u,v,wx} | y}
\ar@{} "2,3";"t" |(0.75){\displaystyle B_{u,v,wx} | y}
}}
\]

\item The $3$-sphere $\omega^H_{u,v,w,x}$ in the case \typequatre{1}{1}{0}:
\[
\!\!\!\!\!\!\!\!\!\!
\!\!\!\!\!\!\!
\vcenter{\xymatrix @!C @C=1em{
& uv | w | x
	\ar@2@/^/ [rr] ^-{\alpha_{uv,w} | x}
	\ar@{} [dr] |-{\displaystyle A_{u,v,w} | x}
&& uvw | x
\\
u | v | w | x
	\ar@2@/^/ [ur] ^-{\alpha_{u,v} | w | x}
	\ar@2 [rr] |-{u | \alpha_{v,w} | x}
	\ar@2@/_/ [dr] _-{u | v | \alpha_{w,x}}
&& u | vw | x
	\ar@2 [ur] |-{\alpha_{u,vw} | x}
	\ar@2 [dr] |-{u | \alpha_{vw,x}}
	\ar@{} [dl] |-{\displaystyle u | A_{v,w,x}}
\\
& u | v | wx
	\ar@2@/_/ [rr] _-{u | \alpha_{v,wx}} 
&& u | vwx
 	\ar@2 @/_6ex/ [uu] |(0.66){\beta_{u,vw,x}} _-{}="t"
\ar@{} "2,3";"t" |-{\displaystyle B_{u,vw,x}}
}}
\quad\qfl\quad
\vcenter{\xymatrix @!C @C=1em{
& uv | w | x
	\ar@2@/^/ [rr] ^-{\alpha_{uv,w} | x} ^-{}="s1"
	\ar@2 [dr] |-{uv | \alpha_{w,x}}
	\ar@{} [dd] |-{\displaystyle =}
&& uvw | x
\\
u | v | w | x
	\ar@2@/^/ [ur] ^-{\alpha_{u,v} | w | x}
	\ar@2@/_/ [dr] _-{u | v | \alpha_{w,x}}
&& uv | wx
	\ar@2 [ur] |-{\beta_{uv,w,x}}
\\
& u | v | wx
 	\ar@2 [ur] |-{\alpha_{u,v} | wx}
	\ar@2@/_/ [rr] _-{u | \alpha_{v,wx}} _-{}="t1"
&& u | vwx
	\ar@2 [ul] |-{\beta_{u,v,wx}}
	\ar@2 @/_6ex/ [uu] |(0.66){\beta_{u,vw,x}} _-{}="t2"
\ar@{} "s1";"2,3" |(0.25){\displaystyle B_{uv,w,x}}
\ar@{} "2,3";"t1" |(0.75){\displaystyle B_{u,v,wx}}
\ar@{} "2,3";"t2" |-{\displaystyle H_{u,v,w,x}}
}}
\]

\item The $3$-sphere $\omega^I_{u,v_1,w_1,v_2,w_2}$ in the case \typetroisbase{$u$}{$v_1$}{$w_1$}{0} and \typetroisbase{$u$}{$v_2$}{$w_2$}{0} with $v_1w_1=v_2w_2$:
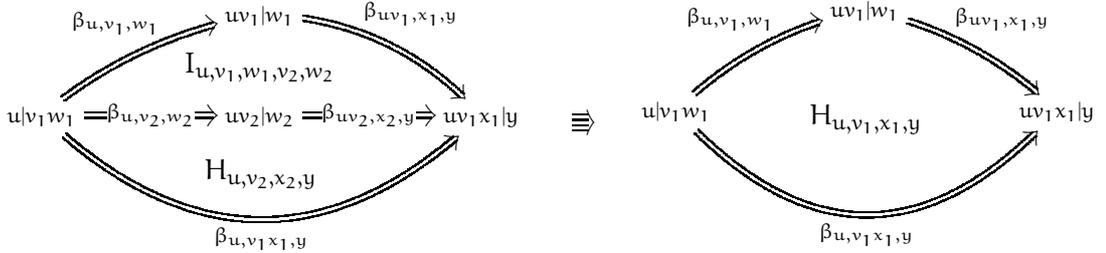
\[
\vcenter{\xymatrix@C=2em {
&& uv_1 | w_1 
	\ar@2@/^2ex/ [drr] ^*+{\beta_{uv_1,x_1,y}}
	\ar@{} [d] |-{\displaystyle I_{u,v_1,w_1,v_2,w_2}}
\\
u | v_1w_1
	\ar@2@/^/ [urr] ^*+{\beta_{u,v_1,w_1}}
	\ar@2 [rr] |-{\beta_{u,v_2,w_2}}
	\ar@2 @/_8ex/ [rrrr] _-{\beta_{u,v_1x_1,y}} _-{}="t"
&& uv_2 | w_2
	\ar@2 [rr] |-{\beta_{uv_2,x_2,y}}
	\ar@{} [];"t" |-{\displaystyle H_{u,v_2,x_2,y}}
&& uv_1x_1 | y
}}
\quad\qfl\quad
\vcenter{\xymatrix@C=1.5em {
&& uv_1 | w_1 
	\ar@2@/^/ [drr] ^*+{\beta_{uv_1,x_1,y}}
\\
u | v_1w_1
	\ar@2@/^/ [urr] ^*+{\beta_{u,v_1,w_1}}
	\ar@2 @/_8ex/ [rrrr] _-{\beta_{u,v_1x_1,y}} _-{}="t"
&&&& uv_1x_1 | y
\ar@{} "1,3";"t" |-{\displaystyle H_{u,v_1,x_1,y}}
}}
\]
\end{itemize}
\renewcommand{\objectstyle}{\displaystyle}
\renewcommand{\labelstyle}{\displaystyle}
We consider the collapsible part $\Gamma$ of $\Sr(\Gar_2(\W))$ made of each of those $3$-spheres and all the $3$-cells $B_{u,v,w}$, with the order $I>H>\cdots>C$. The homotopical reduction of $\Sr(\Gar_2(\W))$ with respect to $\Gamma$ is exactly Garside's coherent presentation $\Gar_3(\W)$, ending the proof of Theorem~\ref{Theorem:GarsideCoherentPresentation}.

\subsection{Garside's coherent presentation for Garside monoids}

Garside monoids have been introduced as a generalisation of spherical Artin monoids by Dehornoy and Paris~\cite{DehornoyParis99,Dehornoy02} to abstract the arithmetic properties observed by Garside on braid monoids~\cite{Garside69} and by Brieskorn-Saito and Deligne on spherical Artin monoids~\cite{BrieskornSaito72,Deligne72}. We refer the reader to \cite{DehornoyDigneGodelleKrammerMichel13} for a unified treatment of Garside structure.

We fix a Garside monoid~$\M$ and we follow~\cite{GebhardtGonzales10} for most of the terminology and notation. 

\subsubsection{Recollections on Garside monoids}

In the monoid~$\M$, all elements $u$ and $v$ admit a greatest common divisor $u\wedge v$. Moreover, the monoid~$\M$ has a Garside element, denoted by $w_0$, such that the set~$\W$ of its divisors generates~$\M$. The complement of an element $u$ of $\W$ in $w_0$ is denoted by $\dr(u)$. A pair $(u,v)$ of elements of $\W$ is \emph{left-weighted} if we have $\dr(u)\wedge v=1$. For each pair $(u,v)$ of elements of~$\W$, there exists a unique left-weighted pair $(u',v')$ of elements of~$\W$ such that $uv=u'v'$ holds in $\M$: we take $u'=u(\dr(u)\wedge v)$ and~$v'$ to be the complement of $\dr(u)\wedge v$ in $v$. The operation transforming $(u,v)$ into $(u',v')$ is called \emph{local sliding}. It induces a computational process that transforms any element~$u$ of~$\W^*$ into its \emph{(left) normal form} by a finite sequence of local slidings, thereafter represented by dashed arrows:
\[
\xymatrix{
u 
	\ar@{-->} [r]
& (\cdots)	
	\ar@{-->} [r]
& {\rep{u}}.
}
\]
Moreover, two elements $u$ and $v$ of $\W^*$ represent the same element of $\M$ if, and only if, they have the same normal form, so that they are linked by a finite sequence of local slidings and their inverses:
\[
\xymatrix{
u 
	\ar@{-->} [r] 
& {\rep{u}}
& v.
	\ar@{-->} [l]
}
\]

\subsubsection{Garside's presentation}

First, let us note that, since the set $\W$ of divisors of $w_0$ generates $\M$, then so does $\W\setminus\ens{1}$. Given two elements $u$ and $v$ of~$\W\setminus\ens{1}$, we use the notations \typedeux{1} and \typedeux{0} to mean
\[
\typedeux{1} 
	\quad\Leftrightarrow\quad
\dr(u)\wedge v = 1,
\]
\[
\typedeux{0}
	\quad\Leftrightarrow\quad
\dr(u)\wedge v \neq 1.
\]
We define \emph{Garside's presentation of $\M$} as the $2$-polygraph $\Gar_2(\M)$ with one $0$-cell, one $1$-cells for every element of $\W\setminus\ens{1}$ and one $2$-cell
\[
\xymatrix @C=2.5em @!C {
u|v 
	\ar@2 [r] ^-{\alpha_{u,v}} 
& uv
} 
\]
for every $u$ and $v$ in $\W\setminus\ens{1}$ such that \typedeux{1} holds. 

Let us check that Garside's presentation is, indeed, a presentation of the monoid $\M$. If \typedeux{1} holds, transforming $u|v$ into $uv$ is a local sliding since $uv$ is the normal form of $u|v$, so that each $2$-cell~$\alpha_{u,v}$ is an instance of local sliding. Conversely, if $u|vw$ is transformed into $uv|w$ by local sliding, this implies, in particular, that both \typedeux{1} and \typedeuxbase{$v$}{$w$}{1} hold. Thus, the composite $2$-cell
\[
\xymatrix @C=3em @!C{
& u|v|w 
	\ar@2@/^/ [dr] ^-*+{\alpha_{u,v}|w}
\\
u|vw 
	\ar@2@/^/ [ur] ^-*+{u|\alpha^-_{v,w}}
	\ar@{-->} [rr] 
&& uv|w
}
\]
corresponds to the local sliding transformation applied to $u|vw$. We define \emph{Garside's coherent presentation} $\Gar_3(\M)$ as done in~\ref{Subsubsection:GarsideCoherentPresentation} for Artin monoids. The proof of Theorem~\ref{Theorem:GarsideCoherentPresentation} adapts in a straightforward way to this case.

\begin{theorem}
\label{Theorem:GarsideGarside}
Every Garside monoid $\M$ admits $\Gar_3(\M)$ as a coherent presentation.
\end{theorem}

\section{Artin's coherent presentation of Artin monoids}
\label{Section:ArtinCoherentPresentation}

Let $\W$ be a Coxeter group with a totally ordered set $S$ of generators. In this section, we use the homotopical reduction method on Garside's coherent presentation $\Gar_3(\W)$ to contract it into a smaller coherent presentation associated to Artin's presentation.

\subsection{Artin's coherent presentation}

We call \emph{Artin's presentation} of the Artin monoid $\B^+(\W)$ the $2$-polygraph $\Art_2(\W)$ with one $0$-cell, the elements of $S$ as $1$-cell and one $2$-cell
\[
\gamma_{s,t} \::\: \angle{ts}^{m_{st}} \:\dfl\: \angle{st}^{m_{st}}
\]
for every $t>s$ in $S$ such that $m_{st}$ is finite.

We recall that, if $I$ is a subset of $S$, then $I$ has an lcm if, and only if, the subgroup $\W_I$ of $\W$ spanned by~$I$ is finite. In that case, the lcm of $I$ is the longest element $w_0(I)$ of $\W_I$. This implies that, if an element $u$ of $\W$ admits reduced expressions $s_1u_1$, \dots, $s_nu_n$ where $s_1$, \dots, $s_n$ are in~$S$, then the subgroup $W_{\ens{s_1,\dots,s_n}}$ is finite and its longest element $w_0(s_1,\dots,s_n)$ is a divisor of $u$. As a consequence, the element $u$ has a unique reduced expression of the shape $w_0(s_1,\dots,s_n)u'$. 

The main theorem of this section extends $\Art_2(\W)$ into \emph{Artin's coherent presentation} of the Artin monoid $\B^+(\W)$.

\begin{theorem}
\label{Theorem:ArtinCoherentPresentation}
For every Coxeter group $\W$, the Artin monoid $\B^+(\W)$ admits the coherent presentation $\Art_3(\W)$ made of Artin's presentation $\Art_2(\W)$ and one $3$-cell $Z_{r,s,t}$ for all elements $t>s>r$ of $S$ such that the subgroup $\W_{\ens{r,s,t}}$ is finite.
\end{theorem}

We note that Artin's coherent presentation has exactly one $k$-cell, $0\leq k\leq 3$, for every subset~$I$ of~$S$ of rank $k$ such that the subgroup $\W_I$ is finite. In~\ref{Subsection:ReductionGarside}, we use homotopical reduction on Garside's coherent presentation $\Gar_3(\W)$ to get a homotopy basis of Artin's presentation. The precise shape of the $3$-cells is given in~\ref{Subsection:Zamolodchikov}. 

\subsection{Homotopical reduction of Garside's coherent presentation}
\label{Subsection:ReductionGarside}

We consider Garside's coherent presentation $\Gar_3(\W)$ of $\B^+(\W)$. The homotopical reduction in the proof of Theorem~\ref{Theorem:GarsideCoherentPresentation} has coherently eliminated some redundant $3$-cells, thanks to generating triple confluences of $\Sr(\Gar_2(\W))$. This convergent $(3,1)$-polygraph has other triple critical branchings. In particular, the critical triple branchings created by three $2$-cells $\alpha$, whose sources are the $u|v|w|x$ with \typequatre{1}{1}{1}, generate the following family $\Gar_4(\W)$ of $4$-spheres $\omega_{u,v,w,x}$ of $\tck{\Gar_3(\W)}$:
\renewcommand{\objectstyle}{\scriptstyle}
\renewcommand{\labelstyle}{\scriptstyle}
\[
\vcenter{\xymatrix @!C @C=1em{
& uv | w | x
	\ar@2@/^/ [rr] ^-{\alpha_{uv,w} | x}
	\ar@{} [dr] |-{\displaystyle A_{u,v,w} | x}
&& uvw | x
	\ar@2@/^/ [dr] ^-{\alpha_{uvw,x}}
	\ar@{} [dd] |-{\displaystyle A_{u,vw,x}}
\\
u | v | w | x
	\ar@2@/^/ [ur] ^-{\alpha_{u,v} | w | x}
	\ar@2 [rr] |-{u | \alpha_{v,w} | x}
	\ar@2@/_/ [dr] _-{u | v | \alpha_{w,x}}
&& u | vw | x
	\ar@2 [ur] |-{\alpha_{u,vw} | x}
	\ar@2 [dr] |-{u | \alpha_{vw,x}}
	\ar@{} [dl] |-{\displaystyle u | A_{v,w,x}}
&& uvwx
\\
& u | v | wx
	\ar@2@/_/ [rr] _-{u | \alpha_{v,wx}} 
&& u | vwx
	\ar@2@/_/ [ur] _-{\alpha_{u,vwx}}
}}
\:\qfl\:
\vcenter{\xymatrix @!C @C=1em{
& uv | w | x
	\ar@2@/^/ [rr] ^-{\alpha_{uv,w} | x}
	\ar@2 [dr] |-{uv | \alpha_{w,x}}
	\ar@{} [dd] |-{\displaystyle =}
&& uvw | x
	\ar@2@/^/ [dr] ^-{\alpha_{uvw,x}}
	\ar@{} [dl] |-{\displaystyle A_{uv,w,x}}
\\
u | v | w | x
	\ar@2@/^/ [ur] ^-{\alpha_{u,v} | w | x}
	\ar@2@/_/ [dr] _-{u | v | \alpha_{w,x}}
&& uv | wx
	\ar@2 [rr] |-{\alpha_{uv,wx}}
	\ar@{} [dr] |-{\displaystyle A_{u,v,wx}}
&& uvwx
\\
& u | v | wx
	\ar@2 [ur] |-{\alpha_{u,v} | wx}
	\ar@2@/_/ [rr] _-{u | \alpha_{v,wx}} 
&& u | vwx
	\ar@2@/_/ [ur] _-{\alpha_{u,vwx}}
}}
\]
\renewcommand{\objectstyle}{\displaystyle}%
\renewcommand{\labelstyle}{\displaystyle}%
To construct a collapsible part of $\Gar_3(\W)$, we use the indexing families of the cells of $\Gar_3(\W)$ and the $3$-spheres of $\Gar_4(\W)$ to classify and compare them.

\subsubsection{The classification}

If $u$ is an element of $\W\setminus\ens{1}$, the \emph{smallest divisor of $u$} is denoted by $d_u$ and defined as the smallest element of $S$ that is a divisor of $u$. Let $(u_1,\dots,u_n)$ be a family of elements of $\W\setminus\ens{1}$ such that
\[
l(u_1\cdots u_n) \:=\: l(u_1)+\cdots+l(u_n).
\]
For every $k\in\ens{1,\dots,n}$, we write $s_k=d_{u_1\cdots u_k}$. We note that $s_1\geq s_2\geq \cdots\geq s_n$ since each $s_k$ divides $u_1\cdots u_l$ for $l\geq k$. Moreover, the elements $s_1$,\dots, $s_k$ have $u_1\cdots u_k$ as common multiple, so that their lcm $w_0(s_1,\dots,s_k)$ exists and divides $u_1\cdots u_k$, and each subgroup $\W_{s_1,\dots,s_k}$ is finite. Thus, we have the following diagram, where each arrow $u\fl v$ means that $u$ is a divisor of $v$:
\[
\xymatrix @C=3em {
w_0(s_1) \ar [r] \ar [d]
&	w_0(s_1,s_2) \ar [r] \ar [d]
& (\cdots)	\ar [r]
& w_0(s_1,\dots,s_{n-1}) \ar [r] \ar [d]
& w_0(s_1,\dots,s_n) \ar [d]
\\
u_1	\ar [r]
& u_1u_2 \ar [r]
& (\cdots) \ar [r]
& u_1\cdots u_{n-1} \ar [r]
& u_1\cdots u_n
}
\]
If every vertical arrow is an equality, we say that $(u_1,\dots,u_n)$ is \emph{essential}. Since each~$u_k$ is different from~$1$, this implies that no horizontal arrow is an equality, so that $s_1>\cdots>s_n$ holds. Moreover, we have $u_1=s_1$ and, by uniqueness of the complement, we get that each $u_{k+1}$ is the complement of $w_0(s_1,\dots,s_k)$ in $w_0(s_1,\dots,s_{k+1})$. Thus, the family $(u_1,\dots,u_n)$ is uniquely determined by the elements $s_1$, \dots, $s_n$ of $S$ such that $s_1>\cdots>s_n$. 

Otherwise, there exists a minimal $k$ in $\ens{1,\dots,n}$ such that $(u_1,\dots,u_k)$ is not essential, \ie, such that $u_1\cdots u_k \neq w_0(s_1,\dots,s_k)$. If $k\geq 2$, there are two possibilities, depending on whether $w_0(s_1,\dots,s_{k-1})$ and $w_0(s_1,\dots,s_k)$ are equal or not, which is equivalent to the equality $s_{k-1}=s_k$ since $s_1>\dots>s_{k-1}\geq s_k$. If $s_{k-1}=s_k$, we say that $(u_1,\dots,u_n)$ is \emph{collapsible}. If $s_{k-1}>s_k$, then we have $u_k\doteq vw$ (\ie, $u_k=vw$ and \typedeuxbase{$v$}{$w$}{1}), with $v$ and~$w$ in $\W\setminus\ens{1}$ such that $(u_1,\dots,u_{k-1},v)$ is essential: we say that $(u_1,\dots,u_n)$ is \emph{redundant}. 

Finally, if $k=1$ and $(u_1)$ is not essential, we have $u_1=s_1 w$ with $w$ in $\W\setminus\ens{1}$ and we say that $(u_1)$ is redundant.

By construction, the family $(u_1,\dots,u_n)$ is either essential, collapsible or redundant. This induces a partition of the cells of $\Gar_3(\W)$ and the spheres of $\Gar_4(\W)$ in three parts.

\subsubsection{The well-founded order}

Finally, we define a mapping 
\[
\Phi(u_1,\dots,u_n) \:=\: \big( l(u_1\cdots u_n),\; d_{u_1},\; l(u_1),\; d_{u_1u_2},\; l(u_1u_2),\; \dots\; ,\; d_{u_1\cdots u_{n-1}},\; l(u_1\cdots u_{n-1}) \big)
\]
of every family $(u_1,\dots,u_n)$ of elements of $\W\setminus\ens{1}$ such that $l(u_1\cdots u_n)=l(u_1)+\cdots+l(u_n)$ into $\Nb\times (S\times\Nb)^{n-1}$. We equip the target set with the well-founded lexicographic order generated by the natural order on $\Nb$ and the fixed order on $S$. We compare families $(u_1,\dots,u_n)$ of elements of $\W\setminus\ens{1}$ such that $l(u_1\cdots u_n) = l(u_1)+\cdots+l(u_n)$ by ordering to their images through $\Phi$. 

The cells of $\Gar_3(\W)$ are then compared according to their indices.

\subsubsection{The collapsible part of \pdf{\Gar_3(\W)}}

We define $\Gamma$ as the collection of all the $2$-cells and $3$-cells of $\Gar_3(\W)$ and all the $3$-spheres of $\Gar_4(\W)$ whose indexing family is collapsible. Let us check that $\Gamma$ is a collapsible part of $\Gar_3(\W)$.

The $2$-cells of $\Gamma$ are the $\alpha_{s,u} : s|u \dfl su$ with $s=d_{su}$. Each one is collapsible, the corresponding redundant $1$-cell is $su$ and we have $su>s$ and $su>u$ because $l(su)>l(s)$ and $l(su)>u$.

The $3$-cells of $\Gamma$ are the
\[
\vcenter{\xymatrix @!C @R=1.5em {
& su|v
	\ar@2 @/^/ [dr] ^-{\alpha_{su,v}}
	\ar@3 []!<0pt,-25pt>;[dd]!<0pt,25pt> ^-{A_{s,u,v}}
\\
s|u|v
	\ar@2 @/^/ [ur] ^-{\alpha_{s,u}|v}
	\ar@2 @/_/ [dr] _-{s|\alpha_{u,v}}
&& suv
\\
& s|uv
	\ar@2 @/_/ [ur] _-{\alpha_{s,uv}}
}}
\]
with either (a) $s=d_{su}$ or (b) $s>d_{su}=d_{suv}$ and $su=w_0(s,d_{su})$. Those $3$-cells are collapsible up to a Nielsen transformation, and the corresponding redundant $2$-cells are: (a) $\alpha_{su,v}$; or (b) $\alpha_{s,uv}$. By hypothesis, the indexing pairs $(su,v)$ and $(s,uv)$ are redundant, so that none of those $2$-cells is in $\Gamma$. We check that each redundant $2$-cell is strictly greater than the other $2$-cells appearing in the source and target of $A_{s,u,v}$. For both cases (a) and (b), we observe that $\alpha_{su,v}$ and $\alpha_{s,uv}$ are always strictly greater than $\alpha_{s,u}$ and $\alpha_{u,v}$ since $l(suv)>l(su)$ and $l(suv)>l(uv)$. Then, we proceed by case analysis:
\begin{enumerate}[\quad(a)]
\item $\alpha_{su,v}>\alpha_{s,uv}$ since $s=d_{su}$ and $l(su)>l(s)$ 
\item $\alpha_{s,uv}>\alpha_{su,v}$ since $s>d_{su}$.
\end{enumerate}
Finally, the $3$-spheres of $\Gamma$ are the $\omega_{s,u,v,w}$
\renewcommand{\objectstyle}{\scriptstyle}
\renewcommand{\labelstyle}{\scriptstyle}
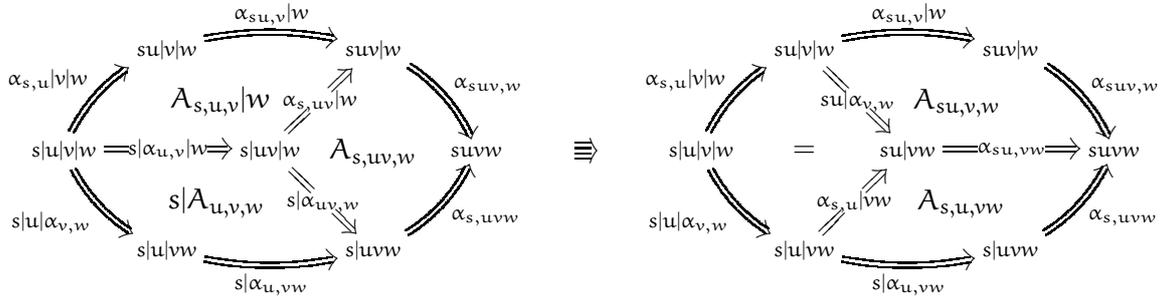
\[
\xymatrix @!C @C=0.75em{
& su | v | w
	\ar@2@/^/ [rr] ^-{\alpha_{su,v} | w}
	\ar@{} [dr] |-{\displaystyle A_{s,u,v} | w}
&& suv | w
	\ar@2@/^/ [dr] ^-{\alpha_{suv,w}}
	\ar@{} [dd] |-{\displaystyle A_{s,uv,w}}
\\
s | u | v | w
	\ar@2@/^/ [ur] ^-{\alpha_{s,u} | v | w}
	\ar@2 [rr] |-{s | \alpha_{u,v} | w}
	\ar@2@/_/ [dr] _-{s | u | \alpha_{v,w}}
&& s | uv | w
	\ar@2 [ur] |-{\alpha_{s,uv} | w}
	\ar@2 [dr] |-{s | \alpha_{uv,w}}
	\ar@{} [dl] |-{\displaystyle s | A_{u,v,w}}
&& suvw
\\
& s | u | vw
	\ar@2@/_/ [rr] _-{s | \alpha_{u,vw}} 
&& s | uvw
	\ar@2@/_/ [ur] _-{\alpha_{s,uvw}}
}
\quad
\raisebox{-36pt}{\qfl}
\quad
\xymatrix @!C @C=0.75em{
& su | v | w
	\ar@2@/^/ [rr] ^-{\alpha_{su,v} | w}
	\ar@2 [dr] |-{su | \alpha_{v,w}}
	\ar@{} [dd] |-{\displaystyle =}
&& suv | w
	\ar@2@/^/ [dr] ^-{\alpha_{suv,w}}
	\ar@{} [dl] |-{\displaystyle A_{su,v,w}}
\\
s | u | v | w
	\ar@2@/^/ [ur] ^-{\alpha_{s,u} | v | w}
	\ar@2@/_/ [dr] _-{s | u | \alpha_{v,w}}
&& su | vw
	\ar@2 [rr] |-{\alpha_{su,vw}}
	\ar@{} [dr] |-{\displaystyle A_{s,u,vw}}
&& suvw
\\
& s | u | vw
	\ar@2 [ur] |-{\alpha_{s,u} | vw}
	\ar@2@/_/ [rr] _-{s | \alpha_{u,vw}} 
&& s | uvw
	\ar@2@/_/ [ur] _-{\alpha_{s,uvw}}
}
\]
\renewcommand{\objectstyle}{\displaystyle}%
\renewcommand{\labelstyle}{\displaystyle}%
with one of the following: 
\[
\begin{cases}
\text{(a) $s=d_{su}$,} \\
\text{(b) $s>d_{su}=d_{suv}$ and $su=w_0(s,d_{su})$,} \\
\text{(c) $s>d_{su}>d_{suv}=d_{suvw}$ and $suv=w_0(s,d_{su},d_{suv})$.}
\end{cases}
\]
Those $3$-cells are collapsible up to a Nielsen transformation, and the corresponding redundant $3$-cells are (a) $A_{su,v,w}$, (b) $A_{s,uv,w}$ or (c) $A_{s,u,vw}$. By hypothesis, the indexing triples $(su,v,w)$, $(s,uv,w)$ and $(s,u,vw)$ are redundant, so that none of those $3$-cells is in $\Gamma$. We observe that $A_{su,v,w}$, $A_{s,uv,w}$ and $A_{s,u,vw}$ are always strictly greater than $A_{s,u,v}$ and $A_{u,v,w}$ since $l(suvw)>l(suv)$ and $l(suvw)>l(uvw)$. Then, we proceed by case analysis:
\begin{enumerate}[\quad(a)]
\item $A_{su,v,w}>A_{s,uv,w}$ and $A_{su,v,w}>A_{s,u,vw}$ since $s=d_{su}$ and $l(su)>l(s)$.
\item $A_{s,uv,w}>A_{su,v,w}$ since $s>d_{su}$ and $A_{s,uv,w}>A_{s,u,vw}$ since $d_{suv}=d_{su}$ and $l(suv)>l(su)$.
\item $A_{s,u,vw}>A_{su,v,w}$ since $s>d_{su}$ and $A_{s,u,vw}>A_{s,uv,w}$ since $d_{su}>d_{suv}$.
\end{enumerate}

\subsubsection{The homotopical reduction}

The homotopical reduction of $\Gar_3(\W)$ with respect to $\Gamma$ is the Tietze transformation $\pi=\pi_{\Gamma}$ that coherently eliminates all the collapsible cells of $\Gamma$ with their corresponding redundant cell. According to the partition of the cells of $\Gar_3(\W)$, this only leaves the essential cells, \emph{i.e.}, those whose indexing family is essential, with source and target replaced by their image through $\pi$.

In particular, the essential $1$-cells are the elements of $S$. By definition of $\Gamma$, the $3$-functor $\pi$ maps a $1$-cell $u$ of $\Gar_3(\W)$ to the element $s\pi(v)$ of $S^*$ if $u\doteq sv$ and $s=d_u$. This gives, by induction,
\[
\pi(u) \:=\: s_1\cdots s_n
\]
for $s_1$, \dots, $s_n$ in $S$ such that $u\doteq s_1\cdots s_n$ and $s_i=d_{s_i\cdots s_n}$. This is sufficient to conclude that the underlying $2$-polygraph of $\Gar_3(\W)/\Gamma$ is (isomorphic to) Artin's presentation of $\B^+(\W)$. 

The essential $2$-cells are the $\alpha_{s,u}$ such that $s>d_{su}$ and $su = w_0(s,d_{su})$. Hence, there is one such $2$-cell for every $t>s$ in $S$ such that $\W_{\ens{s,t}}$ is finite, \ie, such that $m_{st}$ is finite, and its image through $\pi$ has shape
\[
\angle{ts}^{m_{st}} \:\dfl\: \angle{st}^{m_{st}}.
\]

Finally, the essential $3$-cells are the $A_{s,u,v}$ such that $s>d_{su}>d_{suv}$, $su= w_0(s,d_{su})$ and $suv=w_0(s,d_{su},d_{suv})$. Hence, there is one such $3$-cell for every $t>s>r$ in $S$ such that $\W_{\ens{r,s,t}}$ is finite. If we denote by $Z_{r,s,t}$ the image of the corresponding $3$-cell $A_{s,u,v}$ through $\pi$, this concludes the proof of Theorem~\ref{Theorem:ArtinCoherentPresentation}.

\subsection{The \pdf{3}-cells of Artin's coherent presentation}
\label{Subsection:Zamolodchikov}

Let us compute the sources and targets of the $3$-cells $Z_{r,s,t}$ of Artin's coherent presentation. The $3$-cell $Z_{r,s,t}$ is the image through the Tietze transformation $\pi$ of the corresponding essential $3$-cell $A_{t,u,v}$, with $u$ the complement of $t$ in $w_0(s,t)$ and $v$ the complement of $w_0(s,t)$ in $w_0(r,s,t)$. Since the $3$-cell $A_{t,u,v}$ is entirely determined by its source, the shape of the $3$-cell $Z_{r,s,t}$ is determined by the Coxeter type of the parabolic subgroup $\W_{\ens{r,s,t}}$. According to the classification of finite Coxeter groups~\cite[Chap.~VI, \textsection~4, Theorem~1]{BourbakiLie4-6}, there are five cases, shown below:
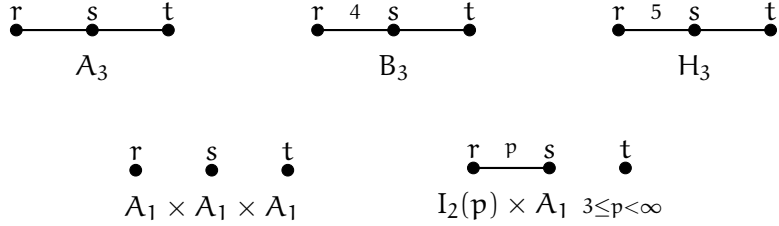
\[
\begin{tikzpicture}
\foreach \i in {1, 2, 3} {
	\coordinate (\i) at (\i, 0) ;
}
\draw [fill = black] (1) circle (0.2em) node [above] {$r$} ;
\draw [fill = black] (2) circle (0.2em) node [above] {$s$} ;
\draw [fill = black] (3) circle (0.2em) node [above] {$t$} ;
\draw [thick] (1) -- (2) ;
\draw [thick] (2) -- (3) ;
\node at (2, -0.5) {$A_3$} ;
\end{tikzpicture}
\qquad\qquad
\begin{tikzpicture}
\foreach \i in {1, 2, 3} {
	\coordinate (\i) at (\i, 0) ;
}
\draw [fill = black] (1) circle (0.2em) node [above] {$r$} ;
\draw [fill = black] (2) circle (0.2em) node [above] {$s$} ;
\draw [fill = black] (3) circle (0.2em) node [above] {$t$} ;
\draw [thick] (1) -- node [above] {$\sm 4$} (2) ;
\draw [thick] (2) -- (3) ;
\node at (2, -0.5) {$B_3$} ;
\end{tikzpicture}
\qquad\qquad
\begin{tikzpicture}
\foreach \i in {1, 2, 3} {
	\coordinate (\i) at (\i, 0) ;
}
\draw [fill = black] (1) circle (0.2em) node [above] {$r$} ;
\draw [fill = black] (2) circle (0.2em) node [above] {$s$} ;
\draw [fill = black] (3) circle (0.2em) node [above] {$t$} ;
\draw [thick] (1) -- node [above] {$\sm 5$} (2) ;
\draw [thick] (2) -- (3) ;
\node at (2, -0.5) {$H_3$} ;
\end{tikzpicture}
\]
\[
\begin{tikzpicture}
\foreach \i in {1, 2, 3} {
	\coordinate (\i) at (\i, 0) ;
}
\draw [fill = black] (1) circle (0.2em) node [above] {$r$} ;
\draw [fill = black] (2) circle (0.2em) node [above] {$s$} ;
\draw [fill = black] (3) circle (0.2em) node [above] {$t$} ;
\node at (2, -0.5) {$A_1\times A_1\times A_1$} ;
\end{tikzpicture}
\qquad\qquad
\begin{tikzpicture}
\foreach \i in {1, 2, 3} {
	\coordinate (\i) at (\i, 0) ;
}
\draw [fill = black] (1) circle (0.2em) node [above] {$r$} ;
\draw [fill = black] (2) circle (0.2em) node [above] {$s$} ;
\draw [fill = black] (3) circle (0.2em) node [above] {$t$} ;
\draw [thick] (1) -- node [above] {$\sm p$} (2) ;
\node at (2, -0.5) {$I_2(p)\times A_1 \:\: \sm{3\leq p<\infty}$} ;
\end{tikzpicture}
\]
Note that we use the numbering conventions of ~\cite[Theorem~1.1]{GeckPfeiffer00}. The resulting $3$-cells are given in Figures~\ref{Figure:ZABAAA} and~\ref{Figure:ZHI}. The rest of this section explains their computation, mainly based on the images of the $2$-cells of $\Gar_3(\W)$ through $\pi$. We detail the cases of the Coxeter types $A_1\times A_1\times A_1$ and~$A_3$. A Python script, based on the PyCox library~\cite{Geck12}, can be used to compute Garside's and Artin's coherent presentations for spherical Artin monoids\footnote{\url{http://www.pps.univ-paris-diderot.fr/~guiraud/cox/cox.zip}}. The $3$-cells $Z_{r,s,t}$ are also given, in ``string diagrams'', in~\cite[Definition~4.3]{EliasWilliamson13}.
\begin{figure}[!p] 
\[
\xymatrix @!C @C=3em{
& strsrt
	\ar@2 [r] ^{s\gamma_{rt}s\gamma_{rt}^-}
& srtstr
	\ar@2 [r] ^{sr\gamma_{st}r}
	\ar@3 []!<0pt,-60pt>;[dddd]!<0pt,60pt> ^{\:Z_{r,s,t}}
& srstsr
	\ar@2@/^2ex/ [dr] ^-{\gamma_{rs}tsr}
\\
stsrst
	\ar@2@/^2ex/ [ur] ^-{st\gamma_{rs}t}
&&&& rsrtsr
\\
tstrst
	\ar@2 [u] ^{\gamma_{st}rst}
	\ar@2 [d] _{ts\gamma_{rt}st}
&&&& rstrsr
	\ar@2 [u] _{rs\gamma_{rt}sr}
\\
tsrtst
	\ar@2@/_2ex/ [dr] _-{tsr\gamma_{st}}
&&&& rstsrs
	\ar@2 [u] _{rst\gamma_{rs}}
\\
& tsrsts
	\ar@2 [r] _{t\gamma_{rs}ts}
& trsrts
	\ar@2 [r] _{\gamma_{rt}s\gamma_{rt}^-s}
& rtstrs
	\ar@2@/_2ex/ [ur] _-{\:r\gamma_{st}rs}
}
\]

\renewcommand{\objectstyle}{\scriptstyle} 
\renewcommand{\labelstyle}{\scriptstyle} 
\[
\xymatrix @!C @R=1.5em @C=2em{
& srtsrtstr
	\ar@2 [r] ^{srts\gamma_{rt}^-str}
& srtstrstr
	\ar@2 [r] ^{sr\gamma_{st}rs\gamma_{rt}}
& srstsrsrt
	\ar@2 [r] ^{srst\gamma_{rs}t}
	\ar@3 []!<0pt,-65pt>;[dddddd]!<0pt,65pt> ^{\displaystyle \: Z_{r,s,t}}
& srstrsrst
	\ar@2 [r] ^{srs\gamma_{rt}srst}
& srsrtsrst	
	\ar@2@/^2ex/ [dr] ^-{\gamma_{rs}tsrst}
\\
strsrstsr
	\ar@2@/^2ex/ [ur] ^-{s\gamma_{rt}sr\gamma_{st}^-r}
&&&&&& rsrstsrst
\\
stsrsrtsr
	\ar@2 [u] ^{st\gamma_{rs}tsr}
&&&&&& rsrtstrst
	\ar@2 [u] _{rsr\gamma_{st}rst}
\\
tstrsrtsr
	\ar@2 [u] ^{\gamma_{st}rsrtsr}
	\ar@2 [d] _{ts\gamma_{rt}s\gamma_{rt}^-sr}
&&&&&& rsrtsrtst
	\ar@2 [u] _{rsrts\gamma_{rt}^-st}
\\
tsrtstrsr
	\ar@2 [d] _{tsr\gamma_{st}rsr}
&&&&&& rstrsrsts
	\ar@2 [u] _{rs\gamma_{rt}sr\gamma_{st}^-}
\\
tsrstsrsr
	\ar@2@/_2ex/ [dr] _-{tsrst\gamma_{rs}}
&&&&&& rstsrsrts
	\ar@2 [u] _{rst\gamma_{rs}ts}
\\
& tsrstrsrs
	\ar@2 [r] _{tsrs\gamma_{rt}srs}
& tsrsrtsrs
	\ar@2 [r] _{t\gamma_{rs}tsrs}
& trsrstsrs
	\ar@2 [r] _{\gamma_{rt}sr\gamma_{st}^-rs}
& rtsrtstrs
	\ar@2 [r] _{rts\gamma_{rt}^-strs}
& rtstrstrs
	\ar@2@/_2ex/ [ur] _-{\quad r\gamma_{st}rs\gamma_{rt}s}
}
\]
\renewcommand{\objectstyle}{\displaystyle}%
\renewcommand{\labelstyle}{\displaystyle}%

\[
\xymatrix @!C @C=3em{
& str
	\ar@2 [r] ^{s\gamma_{rt}} _{}="src"
& srt
	\ar@2@/^2ex/ [dr] ^{\gamma_{rs}t}
\\
tsr
	\ar@2@/^2ex/ [ur] ^{\gamma_{st}r}
	\ar@2@/_2ex/ [dr] _{t\gamma_{rs}}
&&& rst
\\
& trs
	\ar@2 [r] _{\gamma_{rt}s} ^{}="tgt"
& rts
	\ar@2@/_2ex/ [ur] _{r\gamma_{st}}
\ar@3 "src"!<0pt,-25pt>;"tgt"!<0pt,25pt> ^{\:Z_{r,s,t}}
}	
\]
\caption{The $3$-cells $Z_{r,s,t}$ for Coxeter types $A_3$, $B_3$ and $A_1\times A_1\times A_1$}
\label{Figure:ZABAAA}
\end{figure}
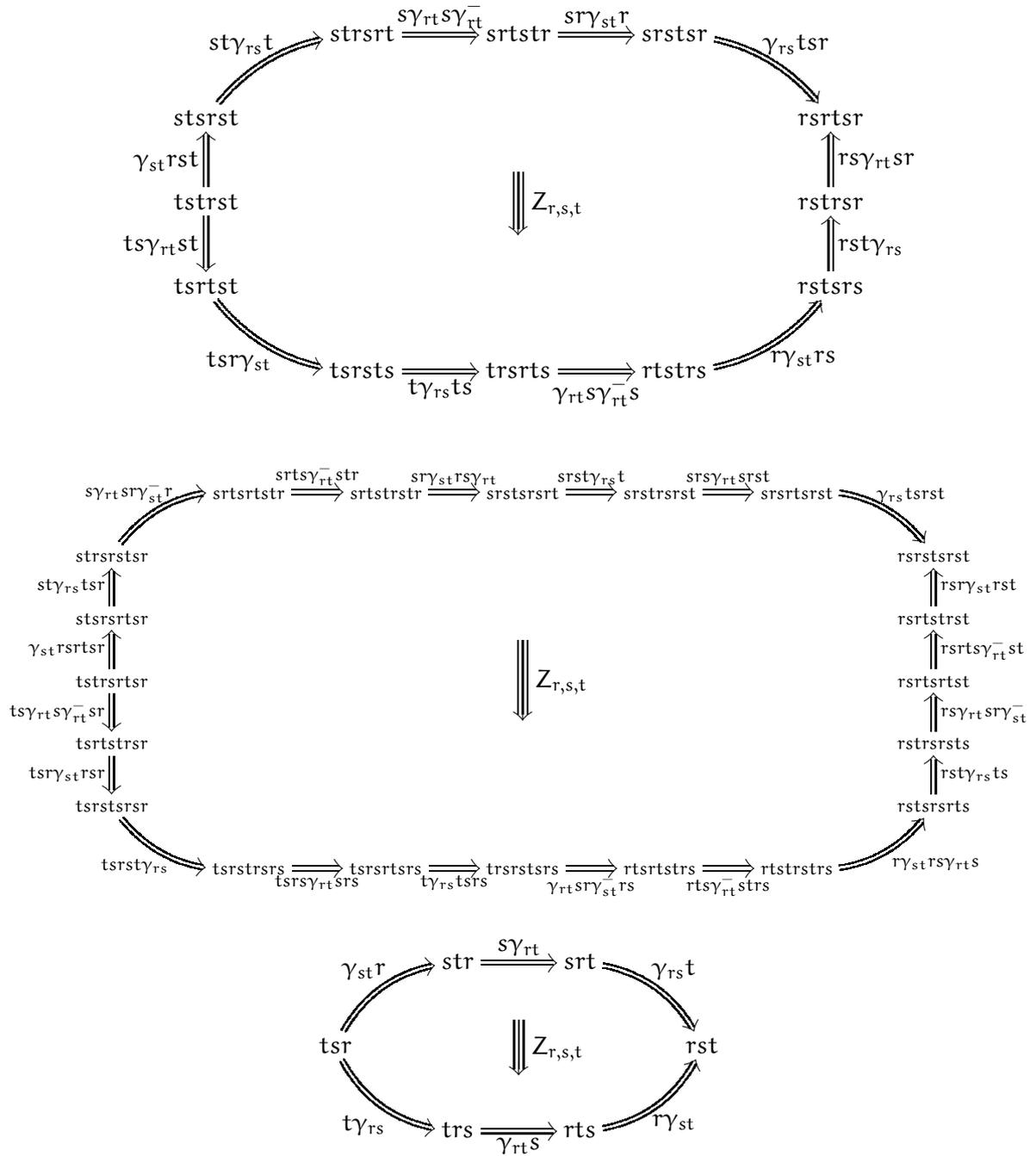

\begin{figure}[!p]
\renewcommand{\objectstyle}{\scriptstyle} 
\renewcommand{\labelstyle}{\scriptstyle} 
\[
\xymatrix @!C @R=1.2em @C=1.5em{
& srstrsrsrtsrsrt
	\ar@2 [r]
& srsrtsrstrsrsrt
	\ar@2 [r] _{}="src"
& srsrtsrstsrsrst
	\ar@2 [r]
& srsrtsrtstrsrst
	\ar@2@/^2ex/ [dr]
\\
srstsrsrstsrsrt
	\ar@2@/^2ex/ [ur]
&&&&& srsrtstrsrtsrst
	\ar@2 [d]
\\
srtstrsrtstrsrt
	\ar@2 [u]
&&&&& srsrstsrsrtsrst
	\ar@2 [d]
\\
srtsrtstrsrtstr
	\ar@2 [u]
&&&&& rsrsrtsrsrtsrst
\\
srtsrstsrsrstsr
	\ar@2 [u]
&&&&& rsrstrsrsrtsrst
	\ar@2 [u]
\\
srtsrstrsrsrtsr
	\ar@2 [u]
&&&&& rsrstsrsrstsrst
	\ar@2 [u]
\\
strsrsrtsrsrtsr
	\ar@2 [u]
&&&&& rsrtstrsrtstrst
	\ar@2 [u]
\\
stsrsrstsrsrtsr
	\ar@2 [u]
&&&&& rsrtsrtstrsrtst
	\ar@2 [u]
\\
tstrsrstsrsrtsr
	\ar@2 [u]
	\ar@2 [d]
&&&&& rsrtsrstsrsrsts
	\ar@2 [u]
\\
tsrtsrstsrstrsr
	\ar@2 [d]
&&&&& rsrtsrstrsrsrts
	\ar@2 [u]
\\
tsrtsrtstrstrsr
	\ar@2 [d]
&&&&& rstrsrsrtsrsrts
	\ar@2 [u]
\\
tsrtstrsrtstrsr
	\ar@2 [d]
&&&&& rstsrsrstsrsrts
	\ar@2 [u]
\\
tsrstsrsrstsrsr
	\ar@2 [d]
&&&&& rtstrsrtstrsrts
	\ar@2 [u]
\\
tsrstrsrsrtsrsr
	\ar@2 [d]
&&&&& rtsrtstrsrtstrs
	\ar@2 [u]
\\
tsrsrtsrstrsrsr
	\ar@2 [d]
&&&&& rtsrstsrsrstsrs
	\ar@2 [u]
\\
tsrsrtsrstsrsrs
	\ar@2@/_2ex/ [dr]
&&&&& rtsrstrsrsrtsrs
	\ar@2 [u]
\\
& tsrsrtsrtstrsrs
	\ar@2 [r]
& tsrsrtstrsrtsrs
	\ar@2 [r] ^{}="tgt"
& tsrsrstsrsrtsrs
	\ar@2 [r]
& trsrsrtsrsrtsrs
	\ar@2@/_2ex/ [ur]
\ar@3 "src"!<0pt,-180pt>;"tgt"!<0pt,180pt> ^{\displaystyle \: Z_{r,s,t}}
}
\]
\renewcommand{\objectstyle}{\displaystyle}
\renewcommand{\labelstyle}{\displaystyle}

\[
\xymatrix @!C @C=4em {
& st\angle{rs}^{p-1}
	\ar@2 [r] ^-{s\gamma_{rt}\angle{rs}^{p-2}}
& (\cdots)
	\ar@2 [r]
	\ar@3 []!<0pt,-30pt>;[dd]!<0pt,30pt> ^{\:Z_{r,s,t}}
& \angle{sr}^p t
	\ar@2@/^3ex/ [dr] ^{\gamma_{rs}t}
\\
t\angle{sr}^p
	\ar@2@/^3ex/ [ur] ^{\gamma_{st}\angle{rs}^{p-1}}
	\ar@2@/_3ex/ [dr] _{t\gamma_{rs}}
&&&& \angle{rs}^p t
\\
& t\angle{rs}^p
	\ar@2 [r] _-{\gamma_{rt}\angle{sr}^{p-1}}
& rt\angle{sr}^{p-1}
	\ar@2 [r] _-{r\gamma_{st}\angle{sr}^{p-2}}
& (\cdots)
	\ar@2@/_3ex/ [ur] 
}
\]
\caption{The $3$-cells $Z_{r,s,t}$ for Coxeter type $H_3$ and $I_2(p)\times A_1$, $p\geq 3$}
\label{Figure:ZHI}
\end{figure}
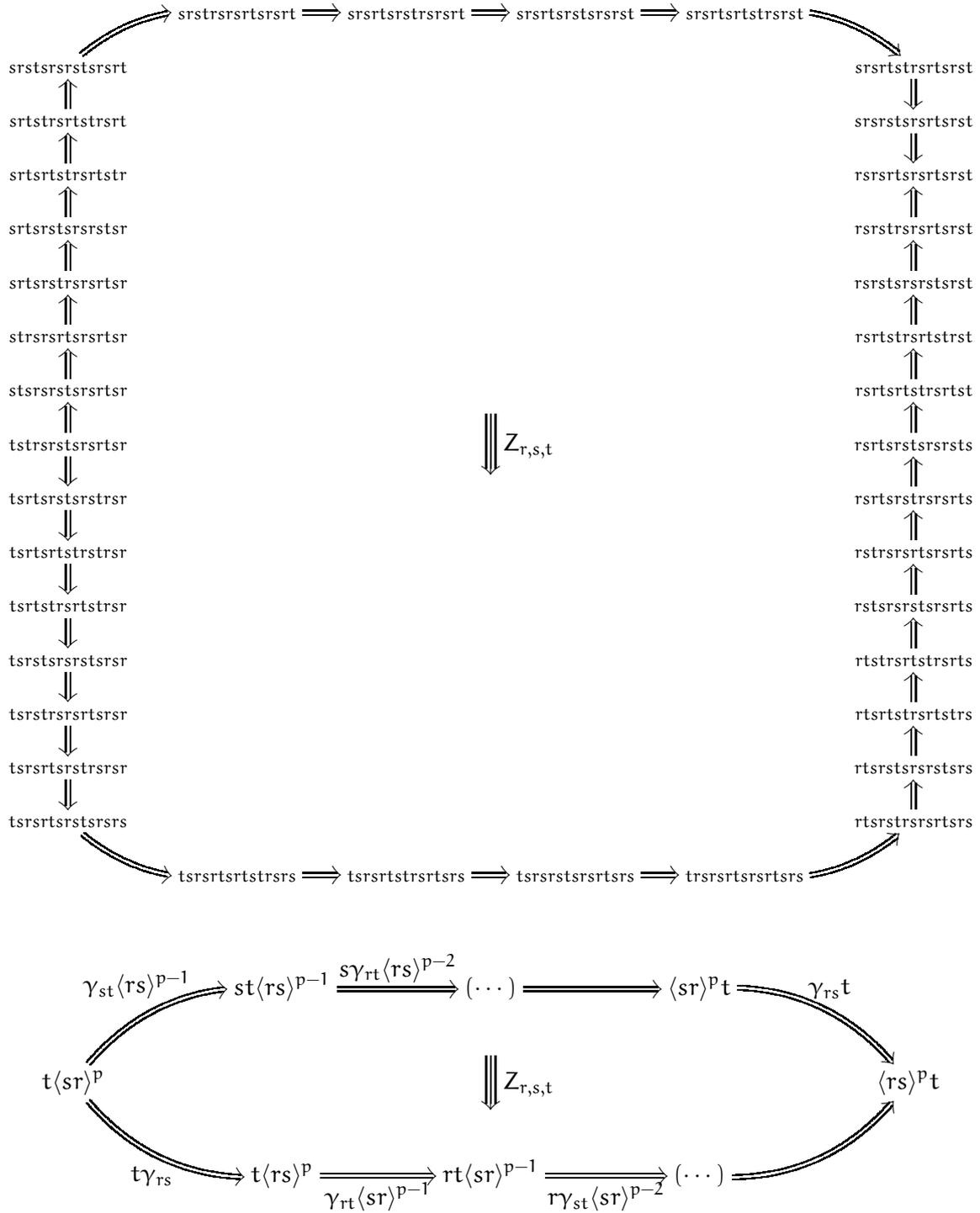

\subsubsection{Projection of the \pdf{2}-cells of Garside's presentation}

By construction, the image of a $2$-cell $\alpha_{u,v}$ of $\Gar_3(\W)$ through $\pi$ is given by induction, depending on whether it is essential, collapsible or redundant. 

The essential $2$-cells are the $\alpha_{t,u}$ such that $t>s$ and $u$ is the complement of $t$ in $w_0(s,t)$, where $s=d_{tu}$. The image of $\alpha_{t,u}$ is the corresponding braid relation:
\[
\pi(\alpha_{t,u}) \:=\: \gamma_{s,t}.
\]
The collapsible $2$-cells are the $\alpha_{s,u}$ such that $s=d_{su}$, mapped to the identity of $\pi(su)$. Finally, there are two disjoint cases of redundant $2$-cells : (a) $\alpha_{su,v}$ with $s=d_{su}$ and (b) $\alpha_{s,uv}$ with $s>d_{su}=d_{suv}$. They are mapped through $\pi$ to the source of the $3$-cell $A_{s,u,v}$ after the appropriate Nielsen transformation, giving the following inductive formulas:
\[
\begin{array}{lc}
\text{(a) $s=d_{su}$}
\quad&
\vcenter{\xymatrix @C=2.5em @R=1.5em {
\pi(su)\pi(v)
	\ar@2@/^3ex/ [rrrr] ^{\pi(\alpha_{su,v})} _{}="src"
	\ar@2@/_/ [dr] _(0.33)*+{\pi(\alpha_{s,u}^-)\pi(v)} 
&&&& \pi(suv)
\\
& \pi(s)\pi(u)\pi(v)
	\ar@2 [rr] _{\pi(s)\pi(\alpha_{u,v})} ^{}="tgt"
&& \pi(s)\pi(uv)
	\ar@2@/_/ [ur] _-{\pi(\alpha_{s,uv})}
\ar@{} "src";"tgt" |-{=}
}}
\\~\\
\text{(b) $s>d_{su} = d_{suv}$}
\quad&
\vcenter{\xymatrix @C=2.5em @R=1.5em {
\pi(s)\pi(uv)
	\ar@2@/^3ex/ [rrrr] ^{\pi(\alpha_{s,uv})} _{}="src"
	\ar@2@/_/ [dr] _(0.33)*+{\pi(s)\pi(\alpha_{u,v}^-)} 
&&&& \pi(suv)
\\
& \pi(s)\pi(u)\pi(v)
	\ar@2 [rr] _{\pi(\alpha_{s,u})\pi(v)} ^{}="tgt"
&& \pi(su)\pi(v)
	\ar@2@/_/ [ur] _-{\pi(\alpha_{su,v})}
\ar@{} "src";"tgt" |-{=}
}}
\end{array}
\]

\subsubsection{The Coxeter type \pdf{A_1\times A_1\times A_1}}

For $t>s>r$ in $S$ such that $\W_{\ens{r,s,t}}$ is of type $A_1\times A_1\times A_1$, the corresponding essential $3$-cell of $\Gar_3(\W)$ is as follows:
\[
\xymatrix @!C @R=1.5em {
& st|r
	\ar@2 @/^/ [dr] ^-{\alpha_{st,r}}
	\ar@3 []!<0pt,-25pt>;[dd]!<0pt,25pt> ^-{A_{t,s,r}}
\\
t|s|r
	\ar@2 @/^/ [ur] ^-{\alpha_{t,s}|r}
	\ar@2 @/_/ [dr] _-{t|\alpha_{s,r}}
&& rst
\\
& t|rs
	\ar@2 @/_/ [ur] _-{\alpha_{t,rs}}
}
\]
The image $Z_{r,s,t}$ of $A_{t,s,r}$ through $\pi$ is given by the inductive application of $\pi$ to the $2$-cells of its source and target. For the source of $Z_{r,s,t}$, we get $\pi(\alpha_{t,s}|r) = \gamma_{st}r$ and
\begin{align*}
\pi(\alpha_{st,r})
	\:&=\: s\pi(\alpha_{t,r}) \star_1 \pi(\alpha_{s,rt}) \\
	\:&=\: s\gamma_{rt} \star_1 \gamma_{rs}t \star_1 \pi(\alpha_{rs,t}) \\
	\:&=\: s\gamma_{rt} \star_1 \gamma_{rs}t.
\end{align*}
For the target of $Z_{r,s,t}$, we get $\pi(t|\alpha_{s,r}) = t\gamma_{rs}$ and
\begin{align*}
\pi(\alpha_{t,rs})
	\:&=\: \gamma_{rt}s \star_1 \pi(\alpha_{rt,s}) \\
	\:&=\: \gamma_{rt}s \star_1 r\gamma_{st} \star_1 \pi(\alpha_{r,st}) \\
	\:&=\: \gamma_{rt}s \star_1 r\gamma_{st}.
\end{align*}
Hence $Z_{r,s,t}$ is the permutohedron, displayed as the third $3$-cell of Figure~\ref{Figure:ZABAAA}.

\subsubsection{The Coxeter type \pdf{A_3}}

If $\W_{\ens{r,s,t}}$ is of type $A_3$ with $t>s>r$, the corresponding essential $3$-cell of $\Gar_3(\W)$ is as follows:
\[
\xymatrix @!C @R=1.5em {
& sts|rst
	\ar@2 @/^/ [dr] ^-{\alpha_{sts,rst}}
	\ar@3 []!<0pt,-25pt>;[dd]!<0pt,25pt> ^-{A_{t,st,rst}}
\\
t|st|rst
	\ar@2 @/^/ [ur] ^-{\alpha_{t,st}|rst}
	\ar@2 @/_/ [dr] _-{t|\alpha_{st,rst}}
&& rst
\\
& t|rsrts
	\ar@2 @/_/ [ur] _-{\alpha_{t,rsrts}}
}
\]
For the source of $Z_{r,s,t}=\pi(A_{t,st,rst})$, we have $\pi(\alpha_{t,st}|rst)=\gamma_{st}rst$ and
\[
\pi(\alpha_{sts,rst}) 
	\:=\: st\pi(\alpha_{s,rst}) \star_1 s\pi(\alpha_{t,rsrt}) \star_1 \pi(\alpha_{s,rstsr}).
\]
Then, we have $st\pi(\alpha_{s,rst})=st\gamma_{rs}t$ and $\pi(\alpha_{s,rstsr})=\gamma_{rs}tsr$, together with
\[
s\pi(\alpha_{t,rsrt}) 
	\:=\: s\gamma_{rt}srt \star_1 sr\pi(\alpha_{t,srt}) \star_1 s\pi(\alpha_{r,stsr}).
\]
Finally, we get $s\pi(\alpha_{r,stsr})=1_{srstsr}$ and
\[
sr\pi(\alpha_{t,srt})
	\:=\: srt\pi(\alpha_{st,r})^- \star_1 sr\gamma_{st}r \star_1 sr\pi(\alpha_{sts,r})
	\:=\: srts\gamma_{rt}^- \star_1 sr\gamma_{st}r.
\]
Wrapping up all those computations, we get the source of $Z_{r,s,t}$ as displayed at the top of Figure~\ref{Figure:ZABAAA}, where an exchange relation has been applied to contract $s\gamma_{rt}srt \star_1 srts\gamma_{rt}^-$ into $s\gamma_{rt}s\gamma_{rt}^-$. The target of $Z_{r,s,t}$ is obtained by similar computations.

Let us note that we can have $\W_{\ens{r,s,t}}$ of type $A_3$, but with another ordering on the elements $r$, $s$, $t$. For example, if $s>r>t$, the $3$-cell $Z_{s,r,t}$ is the image of $A_{t,rt,srt}$ through $\pi$, obtained, up to a Nielsen equivalence, as follows: one considers the $3$-cell of the case $r>s>t$ with $r$ and~$s$ exchanged, then one replaces every occurrence of the $2$-cell $\gamma_{r,s}$, that is not in $\Art_2(\W)$ since $s<r$, by $\gamma_{s,r}^-$.

\section{Coherent presentations and actions on categories}
\label{Section:Actions}

In this section, we establish the relationship between our results on coherent presentations of monoids and Deligne's notion of an action on a category. In particular, we obtain that Deligne's Theorem~\cite[Theorem~1.5]{Deligne97} is equivalent to Theorem~\ref{Theorem:GarsideCoherentPresentation}. We prove that, up to equivalence, the actions of a monoid~$\M$ on categories are the same as the $2$-functors from $\cl{\Sigma}$ to $\Cat$, where $\Sigma$ is any coherent presentation of~$\M$.

\subsection{\pdf{2}-representations of \pdf{2}-categories}

\subsubsection{\pdf{2}-representations}

We recall from~\cite{Elgueta08} that, given $2$-categories $\Cr$ and $\Dr$, a \emph{$2$-representation of $\Cr$ in $\Dr$} is a pseudofunctor $F:\Cr\fl\Dr$. This is a weakened notion of $2$-functor, specified by:
\begin{itemize}
\item for every $0$-cell $x$ of $\Cr$, a $0$-cell $F(x)$ of $\Dr$,
\item for every $1$-cell $u:x\fl y$ of $\Cr$, a $1$-cell $F(u):F(x)\fl F(y)$ of $\Dr$, 
\item for every $2$-cell $f:u\dfl v$ of $\Cr$, a $2$-cell $F(f):F(u)\dfl F(v)$ of $\Dr$.
\end{itemize}
As for $2$-functors, the data are required to be compatible with vertical composition, in a strict way:
\begin{itemize}
\item for all $2$-cells $f:u\dfl v:x\fl y$ and $g:v\dfl w:x\fl y$ of $\Cr$, we have $F(f\star_1g) = F(f)\star_1 F(g)$,
\item for every $1$-cell $u$ of $\Cr$,  we have $F(1_u)=1_{F(u)}$.
\end{itemize}
The data are also compatible with horizontal composition, but only up to coherent isomorphisms: 
\begin{itemize}
\item for all $1$-cells $u:x\fl y$ and $v:y\fl z$ of $\Cr$, an invertible $2$-cell $F_{u,v} : F(u)F(v) \dfl F(uv)$ of~$\Dr$, natural in $u$ and $v$,
\item for every $0$-cell $x$ of $\Cr$, an invertible $2$-cell $F_{x} : 1_{F(x)} \dfl F(1_x)$ of $\Dr$.
\end{itemize}
Finally, these $2$-cells are required to satisfy the following monoidal coherence relations in $\Dr$:
\begin{itemize}
\item for all $1$-cells $u:x\fl y$, $v:y\fl z$ and $w:z\fl t$ of $\Cr$,
\[
\vcenter{\xymatrix{
& {\sm F(y)}
	\ar [r] ^-{\sm F(v)} 
	\ar @2 []!<0pt,-10pt>;[d]!<0pt,15pt> ^-{F_{u,v}} 
& {\sm F(z)}
	\ar @/^/ [dr] ^-{\sm F(w)}
	\ar @2 []!<0pt,-20pt>;[d]!<0pt,-15pt> ^-{F_{uv,w}} 
\\
{\sm F(x)}
	\ar @/^/ [ur] ^-{\sm F(u)}
	\ar @/_3ex/ [urr] |-{\sm F(uv)} 
	\ar @/_3ex/ [rrr] _-{\sm F(uvw)} 
&&& {\sm F(t)}
}}
\qquad = \qquad
\vcenter{\xymatrix{
& {\sm F(y)}
	\ar [r] ^-{\sm F(v)} 
	\ar @/_3ex/ [drr] |-{\sm F(vw)} 
	\ar @2 []!<0pt,-20pt>;[d]!<0pt,-15pt> _-{F_{u,vw}} 
& {\sm F(z)}
	\ar @/^/ [dr] ^-{\sm F(w)}
	\ar @2 []!<0pt,-10pt>;[d]!<0pt,15pt> _-{F_{v,w}} 	
\\
{\sm F(x)}
	\ar @/^/ [ur] ^-{\sm F(u)}
	\ar @/_3ex/ [rrr] _-{\sm F(uvw)}
&&& {\sm F(t)}
}}
\]
\item for every $1$-cell $u:x\fl y$ of $\Cr$,
\[
\vcenter{\xymatrix{
& {\sm F(x)}
	\ar @/^/ [dr] ^-{\sm F(u)}
\\
{\sm F(x)}
	\ar @/^5ex/ [ur] ^(0.33){\sm 1_{F(x)}} ^-{}="src"
	\ar @/_/ [ur] |(0.4){\sm F(1_x)} _-{}="tgt1"
\ar@2 "src"!<7.5pt,-7.5pt>;"tgt1"!<-7.5pt,7.5pt> ^-{F_x}
	\ar @/_/ [rr] _-{\sm F(u)} _-{}="tgt2"
\ar @2 "1,2"!<0pt,-15pt>;"tgt2"!<0pt,10pt> ^-{F_{1_x,u}}
&& {\sm F(y)}
}}
\quad = \quad 
\xymatrix @C=4em {
{\sm F(x)}
	\ar@/^3.5ex/ [r] ^{\sm F(u)} _{}="src"
	\ar@/_3.5ex/ [r] _{\sm F(u)} ^{}="tgt"
	\ar@2 "src"!<-5pt,-10pt>;"tgt"!<-5pt,10pt> ^-{1_{F(u)}}
& {\sm F(y)}
}
\quad = \quad
\vcenter{\xymatrix{
& {\sm F(y)}
	\ar @/^5ex/ [dr] ^(0.67){\sm 1_{F(y)}} ^-{}="src"
	\ar @/_/ [dr] |(0.6){\sm F(1_y)} _-{}="tgt1"
\ar@2 "src"!<-7.5pt,-7.5pt>;"tgt1"!<7.5pt,7.5pt> _-{F_y}
\\
{\sm F(x)}
	\ar @/^/ [ur] ^-{\sm F(u)} 
	\ar @/_/ [rr] _-{\sm F(u)} _-{}="tgt2"
\ar @2 "1,2"!<0pt,-15pt>;"tgt2"!<0pt,10pt> _-{F_{u,1_y}}
&& {\sm F(y)}
}}
\]
\end{itemize}
As usual with monoidal coherence relations, this implies that, for every sequence $(u_1, \dots, u_n)$ of pairwise composable $1$-cells in $\Cr$, there exists a unique invertible $2$-cell 
\[
F_{u_1,\dots,u_n} \::\: F(u_1)\cdots F(u_n) \:\Longrightarrow\: F(u_1\cdots u_n)
\] 
in $\Dr$ built from the coherence isomorphisms of $F$. A $2$-functor is just a pseudofunctor whose coherence $2$-cells are identities: it can be seen as a \emph{strict} $2$-representation. 

The notion of $2$-representation has been introduced by Elgueta for $2$-groups in~\cite{Elgueta08}. It is also studied by Ganter and Kapranov in~\cite{GanterKapranov08} in the special case of groups. In~\cite{Rouquier08}, Rouquier considers the more general case of $2$-representations of bicategories. Among concrete target $2$-categories for $2$-representations, natural choices are the $2$-categories of $2$-vector spaces, either from Kapranov and Voevodsky~\cite{KapranovVoevodsky94} or from Baez and Crans~\cite{BaezCrans04}, of $2$-Hilbert spaces~\cite{Baez97} or of categories~\cite{Deligne97}.

\subsubsection{Morphisms of \pdf{2}-representations}

If $F,G:\Cr\fl\Dr$ are $2$-representations of $\Cr$ into $\Dr$, a \emph{morphism of $2$-representations from $F$ to $G$} is a pseudonatural transformation $\alpha:F\dfl G$ between the corresponding pseudofunctors:
\begin{itemize}
\item for every $0$-cell $x$ of $\Cr$, a $1$-cell $\alpha_x:F(x)\fl G(x)$ of $\Dr$,
\item for every $1$-cell $u:x\fl y$ of $\Cr$, an invertible $2$-cell of $\Dr$ as follows
\[
\xymatrix@R=1em{
& {\sm F(y)}
	\ar@/^/ [dr] ^-{\sm \alpha_y}
\\
{\sm F(x)}
	\ar@/^/ [ur] ^-{\sm F(u)}
	\ar@/_/ [dr] _-{\sm \alpha_x}
&& {\sm G(y)}
\\
& {\sm G(x)}
	\ar@/_/ [ur] _-{\sm G(u)}
\ar@2 "1,2"!<0pt,-15pt>;"3,2"!<0pt,15pt> _-{\simeq} ^-{\alpha_u}	
}
\]
\end{itemize}
These data must satisfy several coherence relations:
\begin{itemize}
\item for every $2$-cell $f:u\dfl v:x\fl y$ of $\Cr$,
\[
\vcenter{\xymatrix @C=3em @R=2em {
& {\sm F(y)}
	\ar@/^/ [dr] ^-{\sm \alpha_y}
\\
{\sm F(x)}
	\ar@/^5ex/ [ur] ^(0.4){\sm F(u)} ^-{}="src"
	\ar@/_/ [ur] |-{\sm F(v)} _-{}="tgt"
\ar@2 "src"!<2.5pt,-7.5pt>;"tgt"!<-10pt,7.5pt> ^-{F(f)}
	\ar@/_/ [dr] _-{\sm \alpha_x}
&& {\sm G(y)}
\\
& {\sm G(x)}
	\ar@/_/ [ur] _-{\sm G(v)}
\ar@2 "1,2"!<0pt,-25pt>;"3,2"!<0pt,25pt> ^-{\alpha_v}	
}}
\qquad = \qquad
\vcenter{\xymatrix @C=3em @R=2em {
& {\sm F(y)}
	\ar@/^/ [dr] ^-{\sm \alpha_y}
\\
{\sm F(x)}
	\ar@/^/ [ur] ^-{\sm F(u)}
	\ar@/_/ [dr] _-{\sm \alpha_x}
&& {\sm G(y)}
\\
& {\sm G(x)}
	\ar@/^/ [ur] |-{\sm G(u)} ^-{}="src"
	\ar@/_5ex/ [ur] _(0.6){\sm G(v)} _-{}="tgt"
\ar@2 "src"!<10pt,-7.5pt>;"tgt"!<-2.5pt,7.5pt> _-{G(f)}
\ar@2 "1,2"!<0pt,-25pt>;"3,2"!<0pt,25pt> _-{\alpha_u}	
}}
\]
\item for all $1$-cells $u:x\fl y$ and $v:y\fl z$ of $\Cr$,
\[
\vcenter{\xymatrix @C=3em @R=2em{
{\sm F(y)}
	\ar@/^/ [r] ^-{\sm F(v)}
& {\sm F(z) }
	\ar@/^/ [dr] ^-{\sm \alpha_z}
\\
{\sm F(x)}
	\ar@/^/ [u] ^-{\sm F(u)}
	\ar@/_/ [ur] |-{\sm F(uv)} _-{}="tgt"
\ar@2 "1,1"!<12.5pt,-7.5pt>;"tgt"!<-10pt,7.5pt> ^-{F_{u,v}}
	\ar@/_/ [dr] _-{\sm \alpha_x}
&& {\sm G(z)}
\\
& {\sm G(x)}
	\ar@/_/ [ur] _-{\sm G(uv)}
\ar@2 "1,2"!<0pt,-25pt>;"3,2"!<0pt,25pt> ^-{\alpha_{uv}}	
}}
\qquad = \qquad
\vcenter{\xymatrix @!C @C=3em @R=2em {
{\sm F(y)}
	\ar@/^/ [r] ^-{\sm F(v)}
	\ar [dr] |-{\sm \alpha_y} ^-{}="src1"
& {\sm F(z) }
	\ar@/^/ [dr] ^-{\sm \alpha_z}
\ar@2 []!<0pt,-12.5pt>;[d]!<0pt,12.5pt> ^-{\alpha_v}
\\
{\sm F(x)}
	\ar@/^/ [u] ^-{\sm F(u)}
	\ar@/_/ [dr] _-{\sm \alpha_x} _-{}="tgt1"
\ar@2 "src1"!<-4.3pt,-12.5pt>;"tgt1"!<0pt,12.5pt> ^(0.75){\alpha_u}
& {\sm G(y)}
	\ar [r] |-{\sm G(v)}
& {\sm G(z)}
\\
& {\sm G(x)}
	\ar [u] |-{\sm G(u)}
	\ar@/_3ex/ [ur] _-{\sm G(uv)} _-{}="tgt2"
\ar@2 "2,2"!<15pt,-15pt>;"tgt2"!<-10pt,7.5pt> ^-{G_{u,v}}	
}}
\]
\item for every $0$-cell $x$ of $\Cr$,
\[
\vcenter{\xymatrix @C=3em @R=2em {
& {\sm F(x)}
	\ar@/^/ [dr] ^-{\sm \alpha_x}
\\
{\sm F(x)}
	\ar@/^5ex/ [ur] ^(0.4){\sm 1_{F(x)}} ^-{}="src"
	\ar@/_/ [ur] |-{\sm F(1_x)} _-{}="tgt"
\ar@2 "src"!<2.5pt,-7.5pt>;"tgt"!<-10pt,7.5pt> ^-{F_x}
	\ar@/_/ [dr] _-{\sm \alpha_x}
&& {\sm G(x)}
\\
& {\sm G(x)}
	\ar@/_/ [ur] _-{\sm G(1_x)}
\ar@2 "1,2"!<0pt,-25pt>;"3,2"!<0pt,25pt> ^-{\alpha_{1_x}}	
}}
\qquad = \qquad
\xymatrix @C=3em{
{\sm F(x)}
	\ar [r] ^-{\sm \alpha_x}
& {\sm G(x)}
	\ar@/^3ex/ [r] ^-{\sm 1_{G(x)}} ^-{}="src"
	\ar@/_3ex/ [r] _-{\sm G(1_x)} _-{}="tgt"
\ar@2 "src"!<-5pt,-10pt>;"tgt"!<-5pt,10pt> ^-{G_x}
& {\sm G(x)}
}
\]
\end{itemize}

\subsubsection{Categories of \pdf{2}-representations}

If $F,G,H:\Cr\fl\Dr$ are $2$-representations and if $\alpha:F\dfl G$ and $\beta:G\dfl H$ are morphisms of $2$-representations, the composite morphism $\alpha\star\beta:F\dfl H$ is defined by:
\begin{itemize}
\item if $x$ is a $0$-cell of $\Cr$, the $1$-cell $(\alpha\star\beta)_x:F(x)\fl H(x)$ of $\Dr$ is the composite
\[
\xymatrix@C=3em{
{F(x)}
	\ar [r] ^-{\alpha_x}
& {G(x)}
	\ar [r] ^-{\beta_x}
& {H(x)}
}
\]
\item if $u:x\fl y$ is a $1$-cell of $\Cr$, the invertible $2$-cell $(\alpha\star\beta)_u$ of $\Dr$ is defined by
\[
\vcenter{\xymatrix @R=1.5em{
& {\sm F(y)}
	\ar@/^/ [dr] ^-{\sm (\alpha\star\beta)_y}
\\
{\sm F(x)}
	\ar@/^/ [ur] ^-{\sm F(u)}
	\ar@/_/ [dr] _-{\sm (\alpha\star\beta)_x}
&& {\sm H(y)}
\\
& {\sm H(x)}
	\ar@/_/ [ur] _-{\sm H(u)}
\ar@2 "1,2"!<0pt,-15pt>;"3,2"!<0pt,15pt> |-{(\alpha\star\beta)_u}	
}}
\qquad = \qquad
\vcenter{\xymatrix @R=1em{
& {\sm F(y)}
	\ar@/^/ [dr] ^-{\sm \alpha_y}
\\
{\sm F(x)}
	\ar@/^/ [ur] ^-{\sm F(u)}
	\ar@/_/ [dr] _-{\sm \alpha_x}
&& {\sm G(y)}
	\ar@/^/ [dr] ^-{\sm \beta_y}
\\
& {\sm G(x)}
	\ar [ur] |-{\sm G(u)}
	\ar@/_/ [dr] _-{\sm \beta_x}
\ar@2 "1,2"!<0pt,-17.5pt>;"3,2"!<0pt,17.5pt> ^-{\alpha_u}
&& {\sm H(y)}
\\
&& {\sm H(x)}
	\ar@/_/ [ur] _-{\sm H(u)}
\ar@2 "2,3"!<0pt,-17.5pt>;"4,3"!<0pt,17.5pt> ^-{\beta_u}
}}
\]
\end{itemize}
The category of $2$-representations of $\Cr$ into~$\Dr$ is denoted by $2\Rep(\Cr,\Dr)$ and its full subcategory whose objects are the $2$-functors is denoted by $2\Cat(\Cr,\Dr)$. 

\subsubsection{Actions of monoids on categories}

If $\M$ is a monoid, we see it as a $2$-category with exactly one $0$-cell $\bullet$, with the elements of $\M$ as $1$-cells and with identity $2$-cells only. We define the category of \emph{actions of $\M$ on categories} as the category $\Act(\M) = 2\Rep(\M,\Cat)$ of $2$-representations of~$\M$ in~$\Cat$. Expanding the definition, an action~$T$ of~$\M$ is specified by a category $\C=T(\bullet)$,  an endofunctor $T(u):\C\fl\C$ for every element~$u$ of~$\M$, a natural isomorphism $T_{u,v}:T(u)T(v)\dfl T(uv)$ for every pair $(u,v)$ of elements of~$\M$ and a natural isomorphism $T_{\bullet}:1_{\C}\dfl T(1)$ such that:
\begin{itemize}
\item for every triple $(u,v,w)$ of elements of $\M$, the following diagram commutes:
\[
\xymatrix @!C @R=1em @C=1em{
& T(uv)T(w)
	\ar@2 @/^2ex/ [dr] ^-{T_{uv,w}}
	\ar@{} [dd] |-{=}
\\
T(u)T(v)T(w)
	\ar@2 @/^2ex/ [ur] ^-{T_{u,v}T(w)}
	\ar@2 @/_2ex/ [dr] _-{T(u)T_{v,w}}
&& T(uvw)
\\
& T(u)T(vw)
	\ar@2 @/_2ex/ [ur] _-{T_{u,vw}}
}
\]
\item for every element $u$ of $\M$, the following two diagrams commute:
\[
\xymatrix @!C @R=1.5em @C=1em {
& T(1)T(u)
	\ar@2 @/^2ex/ [dr] ^-{T_{1,u}}
\\
T(u)
	\ar@2 @/^2ex/ [ur] ^-{T_{\bullet} T(u)}
	\ar@{=} [rr] _-{}="tgt"
&& T(u)
	\ar@{} "1,2";"tgt" |-{=}
}
\qquad\qquad
\xymatrix @!C @R=1.5em @C=1em {
& T(u)T(1)
	\ar@2 @/^2ex/ [dr] ^-{T_{u,1}}
\\
T(u) 
	\ar@2 @/^2ex/ [ur] ^-{T(u) T_{\bullet}}
	\ar@{=} [rr] _-{}="tgt"
&& T(u)
	\ar@{} "1,2";"tgt" |-{=}
}
\]
\end{itemize}
This definition corresponds to the notion of \emph{unital} action of $\M$ on $\C$ that Deligne considers in~\cite{Deligne97}. For semigroups, he proves that unital actions are equivalent to nonunital actions. For any monoid $\M$, this fact is a consequence of the Tietze equivalence of the standard coherent presentation $\Std_3(\M)$ and the reduced standard coherent presentation $\Std'_3(\M)$, given in Example~\ref{Example:ReducedStandardCoherentPresentation}, together with Theorem~\ref{Theorem:CoherentAction}.

\begin{remark}
If $S$ is an action of $\M$ on a category $\C$ and $T$ is an action of $\M$ on a category~$\D$, by expanding the definition, we get that a morphism of actions $\alpha$ from $S$ to $T$ is specified by a functor $F:\C\fl\D$, corresponding to the component of $\alpha$ at the unique $0$-cell of $\M$, and, for every element $u$ of $\M$, a natural isomorphism $\alpha_u:S(u)F\dfl FT(u)$. These data must satisfy the coherence conditions of a pseudonatural transformation. Those morphisms of actions of monoids on categories differ from the ones of Deligne in~\cite{Deligne97}. Indeed, he considers morphisms between actions of $\M$ on the same category $\C$, such that the functor $F$ is the identity of $\C$, but where the natural transformation $\alpha_u$ is not necessarily an isomorphism: those are the \emph{icons} between the corresponding pseudofunctors, as introduced by Lack in~\cite{Lack10} as a special case of \emph{oplax} natural transformations (defined as pseudonatural transformations whose component $2$-cells are not necessarily invertible). Here we follow Elgueta and consider pseudonatural transformations, but the results and proofs can be adapted to icons or generalised to oplax natural transformations. 
\end{remark}

The main theorem of this section relates the coherent presentations and the $2$-representations of a category. It is a direct consequence of Theorem~\ref{Theorem:CoherentPresentationsCofibrantApproximations} and of Proposition~\ref{Proposition:ActionCofibrantApproximation}, whose proof is the objective of the rest of this section.

\begin{theorem}
\label{Theorem:CoherentAction}
Let $\C$ be a category, let $\Sigma$ be an extended presentation of $\C$. The following assertions are equivalent:
\begin{enumerate}[\bf i)]
\item the $(3,1)$-polygraph $\Sigma$ is a coherent presentation of $\C$;
\item for every $2$-category $\Cr$, there is an equivalence of categories
\[
2\Rep(\C,\Cr) \:\approx\: 2\Cat(\cl{\Sigma},\Cr)
\]
that is natural in $\Cr$.
\end{enumerate}
\end{theorem}

\subsection{\pdf{2}-representations of cofibrant \pdf{2}-categories}
\label{suse:2repCofib}

Let us fix $2$-categories $\Cr$ and $\Dr$, with $\Cr$ cofibrant. Our objective is to define a ``strictification'' functor 
\[
\rep{\cdot} \::\: 2\Rep(\Cr,\Dr) \:\longrightarrow\: 2\Cat(\Cr,\Dr)
\] 
and to prove that it is a quasi-inverse for the canonical inclusion functor of $2\Cat(\Cr,\Dr)$ into $2\Rep(\Cr,\Dr)$.

\subsubsection{Strictification of \pdf{2}-representations}

Let $F:\Cr\fl\Dr$ be a $2$-representation. Let us define the $2$-functor $\rep{F}:\Cr\fl\Dr$, dimension after dimension. On $0$-cells, $\rep{F}$ takes the same values as $F$. Since $\Cr$ is cofibrant, its underlying $1$-category is free: on generating $1$-cells, $\rep{F}$ is equal to $F$ and, then, it is extended by functoriality on every $1$-cell. Hence, if $u=a_1\cdots a_n$ is a $1$-cell of $\Cr$, where the $a_i$s are generating $1$-cells, we have
\[
\rep{F}(u) \:=\: F(a_1)\cdots F(a_n).
\]
From the monoidal coherence relations satisfied by $F$, there is a unique invertible $2$-cell in $\Dr$ 
\[
\xymatrix@C=4em{
{\rep{F}(u) \:=\: F(a_1)\cdots F(a_n) \quad}
	\ar@2[r] ^-{F_{a_1,\dots,a_n}} 
& {\quad F(a_1\cdots a_n) \:=\: F(u)}
}
\]
from $\rep{F}(u)$ to $F(u)$, built from the coherence $2$-cells of $F$. Since the decomposition of $u$ in generators is unique, we simply denote this $2$-cell by $F_u$. Let $f:u\dfl v:x\fl y$ be a $2$-cell of $\Cr$. We define $\rep{F}(f)$ as the following composite $2$-cell of $\Dr$, where the double arrows, which always go from top to bottom, have been omitted for readability:
\[
\xymatrix{
{\sm F(x)}
	\ar@/^2em/ [rr] ^{\sm\rep{F}(u)} ^-{}="src"
	\ar@/_2em/ [rr] _{\sm \rep{F}(v)} _-{}="tgt"
\ar@{} "src";"tgt" |-{\rep{F}(f)}
&& {\sm F(y)}
}
\qquad=\qquad
\xymatrix{
{\sm F(x)}
	\ar@/^4em/ [rr] ^(0.75){\sm\rep{F}(u)} ^-{}="src1"
	\ar@/^1.5em/ [rr] |-{\sm F(u)} _-{}="tgt1" ^-{}="src2"
	\ar@/_1.5em/ [rr] |-{\sm F(v)} _-{}="tgt2" ^-{}="src3"
	\ar@/_4em/ [rr] _(0.75){\sm \rep{F}(v)} _-{}="tgt3"
\ar@{} "src1";"tgt1" |-{F_u}
\ar@{} "src2";"tgt2" |-{F(f)}
\ar@{} "src3";"tgt3" |-{F_v^-}
&& {\sm F(y)}
}
\]
As a direct consequence, we get that $\rep{F}$ is compatible with vertical composition and identities of $1$-cells. Hence, we have defined a $2$-functor $\rep{F}$ from $\Cr$ to $\Dr$. We note that the monoidal coherence relations satisfied by $F$ imply that, if $u:x\fl y$ and $v:y\fl z$ are $1$-cells of $\Cr$, we have
\[
\xymatrix{
{\sm F(x)}
	\ar@/^2em/ [rr] ^-{\sm \rep{F}(uv)} ^-{}="src"
	\ar@/_2em/ [rr] _-{\sm F(uv)} _-{}="tgt"
\ar@{} "src";"tgt" |-{F_{uv}}
&& {\sm F(z)}
}
\qquad=\qquad
\xymatrix@C=5em{
{\sm F(x)}
	\ar@/^2.5em/ [r] ^(0.2){\sm \rep{F}(u)} ^-{}="src1"
	\ar [r] |-{\sm F(u)} _-{}="tgt1"
\ar@{} "src1";"tgt1" |-{F_u}
	\ar@/_2.5em/ [rr] _-{\sm F(uv)} _-{}="tgt3"
&	{\sm F(y)}
	\ar@/^2.5em/ [r] ^(0.8){\sm \rep{F}(v)} ^-{}="src2"
	\ar [r] |-{\sm F(v)} _-{}="tgt2"
\ar@{} "src2";"tgt2" |-{F_v}
\ar@{} [];"tgt3" |-{F_{u,v}}
& {\sm F(z)}
}
\]
and, if $x$ is a $0$-cell of $\Cr$, we have $F_{1_x} = F_x$.

\subsubsection{Strictification of morphisms of \pdf{2}-representations}

Let $F,G:\Cr\fl\Dr$ be $2$-representations and let $\alpha:F\dfl G$ be a morphism between them. Let us define a pseudonatural transformation $\rep{\alpha}:\rep{F}\dfl\rep{G}$. For a $0$-cell $x$ of $\Cr$, we take $\rep{\alpha}_x=\alpha_x$. If $u:x\fl y$ is a $1$-cell of $\Cr$, we define $\rep{\alpha}_u$ as the following invertible $2$-cell of $\Dr$:
\[
\vcenter{\xymatrix{
& {\sm F(y) }
	\ar@/^/ [dr] ^-{\sm \alpha_y}
\\
{\sm F(x)}
	\ar@/^/ [ur] ^-{\sm \rep{F}(u)}
	\ar@/_/ [dr] _-{\sm \alpha_x}
&& {\sm G(y)}
\\
& {\sm G(x)}
	\ar@/_/ [ur] _-{\sm \rep{G}(u)}
\ar@{} "1,2";"3,2" |-{\rep{\alpha}_u}	
}}
\qquad = \qquad
\vcenter{\xymatrix{
& {\sm F(y)}
	\ar@/^/ [dr] ^-{\sm \alpha_y}
\\
{\sm F(x)}
	\ar@/^5ex/ [ur] ^(0.4){\sm \rep{F}(u)} ^-{}="src1"
	\ar@/_/ [ur] |(0.4){\sm F(u)} _-{}="tgt1"
\ar@{} "src1"!<7.5pt,-7.5pt>;"tgt1"!<-7.5pt,7.5pt> |-{F_u}
	\ar@/_/ [dr] _-{\sm \alpha_x}
&& {\sm G(y)}
\\
& {\sm G(x)}
	\ar@/^/ [ur] |(0.6){\sm G(u)} ^-{}="src2"
	\ar@/_5ex/ [ur] _(0.6){\sm \rep{G}(u)} _-{}="tgt2"
\ar@{} "src2";"tgt2" |-{G_u^-}
\ar@{} "1,2";"3,2" |-{\alpha_u}	
}} 
\]
This defines a pseudonatural transformation $\rep{\alpha}:\rep{F}\dfl\rep{G}$. Indeed, if~$x$ is a $0$-cell of~$\Cr$, we have:
\[
\vcenter{\xymatrix{
& {\sm F(x) }
	\ar@/^/ [dr] ^-{\sm \alpha_x}
\\
{\sm F(x)}
	\ar@/^/ [ur] ^-{\sm 1_{F(x)}}
	\ar@/_/ [dr] _-{\sm \alpha_x}
&& {\sm G(x)}
\\
& {\sm G(x)}
	\ar@/_/ [ur] _-{\sm 1_{G(x)}}
\ar@{} "1,2";"3,2" |-{\rep{\alpha}_{1_x}}	
}}
\quad = \quad
\vcenter{\xymatrix{
& {\sm F(x)}
	\ar@/^/ [dr] ^-{\sm \alpha_x}
\\
{\sm F(x)}
	\ar@/^5ex/ [ur] ^(0.33){\sm 1_{F(x)}} ^-{}="src1"
	\ar@/_/ [ur] |(0.4){\sm F(1_x)} _-{}="tgt1"
\ar@{} "src1";"tgt1" |-{F_x}
	\ar@/_/ [dr] _-{\sm \alpha_x}
&& {\sm G(x)}
\\
& {\sm G(x)}
	\ar@/^/ [ur] |(0.6){\sm G(1_x)} ^-{}="src2"
	\ar@/_5ex/ [ur] _(0.67){\sm 1_{G(x)}} _-{}="tgt2"
\ar@{} "src2";"tgt2" |-{G_x^-}
\ar@{} "1,2";"3,2" |-{\alpha_{1_x}}	
}}
\quad = \quad
\vcenter{\xymatrix @C=3em{
{\sm F(x)}
	\ar@/^4ex/ [r] ^{\sm \alpha_x}
	\ar@/_4ex/ [r] _{\sm \alpha_x}
	\ar@{} [r] |{1_{\alpha_x}}
& {\sm G(x)}
}}
\]
Then, if $u:x\fl y$ and $v:y\fl z$ are $1$-cells of $\Cr$, we get:
\[
\begin{array}{c r c l}
&\vcenter{\xymatrix{
& {\sm F(z) }
	\ar@/^/ [dr] ^-{\sm \alpha_z}
\\
{\sm F(x)}
	\ar@/^/ [ur] ^-{\sm \rep{F}(uv)}
	\ar@/_/ [dr] _-{\sm \alpha_x}
&& {\sm G(z)}
\\
& {\sm G(x)}
	\ar@/_/ [ur] _-{\sm \rep{G}(uv)}
\ar@{} "1,2";"3,2" |-{\rep{\alpha}_{uv}}	
}}
\quad
& = &
\quad 
\vcenter{\xymatrix{
{\sm F(y)}
	\ar@/^6ex/ [r] ^(0.8){\sm \rep{F}(v)} ^-{}="src2"
	\ar [r] |-{\sm F(v)} _-{}="tgt2"
\ar@{} "src2";"tgt2" |-{F_v}
& {\sm F(z) }
	\ar@/^/ [dr] ^-{\sm \alpha_z}
\\
{\sm F(x)}
	\ar@/^6ex/ [u] ^(0.25){\sm \rep{F}(u)} ^-{}="src1"
	\ar [u] |(0.4){\sm F(u)} _-{}="tgt1"
\ar@{} "src1";"tgt1" |-{F_u}
	\ar@/_/ [ur] |(0.4){\sm F(uv)} _-{}="tgt3"
\ar@{} "1,1";"tgt3" |-{F_{u,v}}
	\ar@/_/ [dr] _-{\sm \alpha_x}
&& {\sm G(z)}
\\
& {\sm G(x)}
	\ar@/^/ [ur] |(0.6){\sm G(uv)} ^-{}="src6"
\ar@{} "src6";"3,3" |-{G_{u,v}^-}
	\ar [r] |-{\sm G(u)} ^-{}="src4"
	\ar@/_6ex/ [r] _(0.2){\sm \rep{G}(u)} _-{}="tgt4"
\ar@{} "src4";"tgt4" |-{G_u^-}
& {\sm G(y)}
	\ar [u] |(0.6){\sm G(v)} ^-{}="src5"
	\ar@/_6ex/ [u] _(0.75){\sm \rep{G}(v)} _-{}="tgt5"
\ar@{} "src5";"tgt5" |-{G_v^-}
\ar@{} "1,2";"3,2" |-{\alpha_{uv}}	
}}
\\
= & \quad
\vcenter{\xymatrix{
{\sm F(y)}
	\ar@/^6ex/ [r] ^(0.8){\sm \rep{F}(v)} ^-{}="src2"
	\ar@/_/ [r] |-{\sm F(v)} _-{}="tgt2"
\ar@{} "src2";"tgt2" |-{F_v}
	\ar [ddrr] |-{\sm \alpha_y} _-{}="tgt6" ^-{}="src3"
& {\sm F(z) }
	\ar@/^/ [dr] ^-{\sm \alpha_z} ^-{}="src6"
\\
{\sm F(x)}
	\ar@/^6ex/ [u] ^(0.25){\sm \rep{F}(u)} ^-{}="src1"
	\ar@/_/ [u] |(0.4){\sm F(u)} _-{}="tgt1"
\ar@{} "src1";"tgt1" |-{F_u}
	\ar@/_/ [dr] _-{\sm \alpha_x} _-{}="tgt3"
&& {\sm G(z)}
\\
& {\sm G(x)}
	\ar@/^/ [r] |-{\sm G(u)} ^-{}="src4"
	\ar@/_6ex/ [r] _(0.2){\sm \rep{G}(u)} _-{}="tgt4"
\ar@{} "src4";"tgt4" |-{G_u^-}
& {\sm G(y)}
	\ar@/^/ [u] |(0.6){\sm G(v)} ^-{}="src5"
	\ar@/_6ex/ [u] _(0.75){\sm \rep{G}(v)} _-{}="tgt5"
\ar@{} "src5";"tgt5" |-{G_v^-}
\ar@{} "src3";"tgt3" |-{\alpha_u}	
\ar@{} "src6";"tgt6" |-{\alpha_v}	
}}
\quad
& = &
\quad
\vcenter{\xymatrix{
{\sm F(y)}
	\ar [r] ^-{\sm \rep{F}(v)} 
	\ar [ddrr] |-{\sm \alpha_y} _-{}="tgt2" ^-{}="src1"
& {\sm F(z) }
	\ar@/^/ [dr] ^-{\sm \alpha_z} ^-{}="src2"
\\
{\sm F(x)}
	\ar [u] ^-{\sm \rep{F}(u)}
	\ar@/_/ [dr] _-{\sm \alpha_x} _-{}="tgt1"
&& {\sm G(z)}
\\
& {\sm G(x)}
	\ar [r] _-{\sm \rep{G}(u)} 
& {\sm G(y)}
	\ar [u] _-{\sm \rep{G}(v)} 
\ar@{} "src1";"tgt1" |-{\rep{\alpha}_u}	
\ar@{} "src2";"tgt2" |-{\rep{\alpha}_v}	
}}
\end{array}
\]
Finally, if $f:u\dfl v:x\fl y$ is a $2$-cell of~$\Cr$:
\[
\begin{array}{c c c c c}
\vcenter{\xymatrix{
& {\sm F(y)}
	\ar@/^/ [dr] ^-{\sm \alpha_y}
\\
{\sm F(x)}
	\ar@/^5ex/ [ur] ^(0.4){\sm \rep{F}(u)} ^-{}="src"
	\ar@/_/ [ur] |(0.4){\sm \rep{F}(v)} _-{}="tgt"
\ar@{} "src";"tgt" |-{\rep{F}(f)}
	\ar@/_/ [dr] _-{\sm \alpha_x}
&& {\sm G(y)}
\\
& {\sm G(x)}
	\ar@/_/ [ur] _-{\sm \rep{G}(v)}
\ar@{} "1,2";"3,2" |-{\rep{\alpha}_v}	
}}
\quad
& = & 
\vcenter{\xymatrix{
& {\sm F(y)}
	\ar@/^/ [dr] ^-{\sm \alpha_y}
\\
{\sm F(x)}
	\ar@/^9ex/ [ur] ^(0.33){\sm \rep{F}(u)} ^-{}="src1"
	\ar@/^3ex/ [ur] |-{\sm F(u)} _-{}="tgt1" ^-{}="src2"
	\ar@/_3ex/ [ur] |(0.4){\sm F(v)} _-{}="tgt2"
\ar@{} "src1";"tgt1" |-{F_u}
\ar@{} "src2";"tgt2" |-{F(f)}
	\ar@/_/ [dr] _-{\sm \alpha_x}
&& {\sm G(y)}
\\
& {\sm G(x)}
	\ar@/^/ [ur] |(0.6){\sm G(v)} ^-{}="src3"
	\ar@/_5ex/ [ur] _(0.7){\sm \rep{G}(v)} _-{}="tgt3"
\ar@{} "src3";"tgt3" |-{G_v^-}
\ar@{} "1,2";"3,2" |(0.6){\alpha_v}	
}} 
\\
& = &
\vcenter{\xymatrix{
& {\sm F(y)}
	\ar@/^/ [dr] ^-{\sm \alpha_y}
\\
{\sm F(x)}
	\ar@/^5ex/ [ur] ^(0.3){\sm \rep{F}(u)} ^-{}="src1"
	\ar@/_/ [ur] |(0.4){\sm F(u)} _-{}="tgt1" 
\ar@{} "src1";"tgt1" |-{F_u}
	\ar@/_/ [dr] _-{\sm \alpha_x}
&& {\sm G(y)}
\\
& {\sm G(x)}
	\ar@/^3ex/ [ur] |(0.6){\sm G(u)} ^-{}="src2"
	\ar@/_3ex/ [ur] |-{\sm G(v)} _-{}="tgt2" ^-{}="src3"
	\ar@/_9ex/ [ur] _(0.67){\sm \rep{G}(v)} _-{}="tgt3"
\ar@{} "src2";"tgt2" |-{G(f)}
\ar@{} "src3";"tgt3" |-{G_v^-}
\ar@{} "1,2";"3,2" |(0.4){\alpha_u}	
}} 
& = &
\quad
\vcenter{\xymatrix{
& {\sm F(y)}
	\ar@/^/ [dr] ^-{\sm \alpha_y}
\\
{\sm F(x)}
	\ar@/^/ [ur] ^-{\sm \rep{F}(u)}
	\ar@/_/ [dr] _-{\sm \alpha_x}
&& {\sm G(y)}
\\
& {\sm G(x)}
	\ar@/^/ [ur] |(0.6){\sm \rep{G}(u)} ^-{}="src"
	\ar@/_5ex/ [ur] _(0.6){\sm \rep{G}(v)} _-{}="tgt"
\ar@{} "src";"tgt" |-{\rep{G}(f)}
\ar@{} "1,2";"3,2" |-{\rep{\alpha}_u}	
}}
\end{array}
\]
\noindent
With similar computations, we check that strictification is compatible with the composition of morphisms of $2$-representations and with identities, so that it is a functor from $2\Rep(\Cr,\Dr)$ to $2\Cat(\Cr,\Dr)$.

\begin{proposition}
\label{PropEquivCofib}
Let $\Cr$ be a cofibrant $2$-category. For every $2$-category $\Dr$, the canonical inclusion 
\[
2\Cat(\Cr,\Dr) \:\longrightarrow\: 2\Rep(\Cr,\Dr) 
\]
is an equivalence of categories that is natural in~$\Dr$, with quasi-inverse given by the strictification functor. 
\end{proposition}

\begin{proof}
It is sufficient to check that, for every $2$-representation $F:\Cr\fl\Dr$, there exists a pseudonatural isomorphism $\phi_F:\rep{F}\dfl F$ that is itself natural in $F$. We define $\phi_F$ as follows:
\begin{itemize}
\item if $x$ is a $0$-cell of $\Cr$, then $\rep{F}(x)=F(x)$ and we take $(\phi_F)_x=1_x$,
\item if $u:x\fl y$ is a $1$-cell of $\Cr$, then $(\phi_F)_u:\rep{F}(u)\dfl F(u)$ is defined as the invertible coherence $2$-cell $F_u:\rep{F}(u)\dfl F(u)$.
\end{itemize}
These data satisfy the required coherence properties: the compatibility with the $2$-cells of~$\Cr$ is exactly the definition of $\rep{F}$ and the compatibility with horizontal composition and identities comes from the monoidal coherence relations of~$F$, as already checked. Moreover, if $\alpha:F\dfl G$ is a morphism of $2$-representations, the naturality condition
\[
\xymatrix @!C @C=3em @R=1em{
& {F}
	\ar@2@/^/ [dr] ^-{\alpha}
\\
{\rep{F}}
	\ar@2@/^/ [ur] ^-{\phi_F}
	\ar@2@/_/ [dr] _-{\rep{\alpha}}
			\ar@{} [rr] |-{=}
&& {G}
\\
& {\rep{G}}
	\ar@2@/_/ [];[ur]!<-0.1pt,-0.1pt> _{\phi_G}
}
\]
corresponds, on each $1$-cell $u$ of $\Cr$, to the definition of $\rep{\alpha}$.
\end{proof}

\subsection{\pdf{2}-representations and cofibrant approximations}

Let us recall that, for a $2$-category $\Cr$, we denote by $\rep{\Cr}$ its standard cofibrant replacement. We note that the definition of a $2$-functor from $\rep{\Cr}$ to a $2$-category $\Dr$ is exactly the same as the one of a pseudofunctor from~$\Cr$ to $\Dr$, yielding the following isomorphism of categories:
\[
2\Rep(\Cr,\Dr) \:\simeq\: 2\Cat(\rep{\Cr},\Dr).
\]
In particular, for every monoid $\M$, we get an isomorphism of categories:
\[
\Act(\M) \:\simeq\: 2\Cat(\rep{\M},\Cat).
\]
In what follows, we prove that weak versions of these isomorphisms exist for all cofibrant approximations. More precisely, the category of $2$-representations of a $2$-category $\Cr$ into a $2$-category $\Dr$ is equivalent to the one of $2$-functors from any cofibrant approximation $\tilde{\Cr}$ of $\Cr$ into $\Dr$. 

\begin{lemma}
\label{LemEquivWE}
Let $\Cr$ and $\Dr$ be $2$-categories. The following assertions are equivalent:
\begin{enumerate}[\bf i)]
\item the $2$-categories $\Cr$ and $\Dr$ are pseudoequivalent, \ie, there exist pseudofunctors $F:\Cr\fl\Dr$ and $G:\Dr\fl\Cr$ such that
\[
GF \:\simeq\: 1_{\Cr}
\qquad\text{and}\qquad
FG \:\simeq\: 1_{\Dr}; 
\]
\item for every $2$-category $\Er$, there is an equivalence of categories
\[
2\Rep(\Cr,\Er) \:\approx\: 2\Rep(\Dr,\Er)
\] 
that is natural in $\Er$.
\end{enumerate}
\end{lemma}

\begin{proof}
Let us assume that $\Cr$ and $\Dr$ are pseudoequivalent. As a consequence, for all pseudofunctors $H:\Cr\fl\Er$ and $K:\Dr\fl\Er$, we have:
\[
HGF \:\simeq\: H
\qquad\text{and}\qquad
KFG \:\simeq\: K.
\]
Thus the functors $2\Rep(F,\Er)$ and $2\Rep(G,\Er)$, respectively sending a pseudofunctor~$K:\Dr\fl\Er$ to $KF$ and a pseudofunctor~$H:\Cr\fl\Er$ to $HG$, form the required equivalence of categories. 

Conversely, let us assume that, for every $2$-category $\Er$, we have $2\Rep(\Cr,\Er)\approx 2\Rep(\Dr,\Er)$ natural in $\Er$. We denote by
\[
\Phi_{\Er} \::\: 2\Rep(\Cr,\Er) \:\fl\: 2\Rep(\Dr,\Er)
\qquad\text{and}\qquad
\Psi_{\Er} \::\: 2\Rep(\Dr,\Er) \:\fl\: 2\Rep(\Cr,\Er)
\]
the functors that constitute the equivalence. This means that, for all pseudofunctors $H:\Cr\fl\Er$ and $K:\Dr\fl\Er$, we have the following isomorphisms:
\[
\Psi_{\Er} \Phi_{\Er} (H) \:\simeq\: H 
\qquad\text{and}\qquad
\Phi_{\Er} \Psi_{\Er} (K) \:\simeq\: K.
\]
The naturality of the equivalence means that, for all $2$-categories $\Er$ and $\Er'$ and every pseudofunctor $H:\Er\fl\Er'$, the following diagrams commute:
\[
\xymatrix @!C @C=3em {
2\Rep(\Cr,\Er)
	\ar [r] ^-{\Phi_{\Er}}
	\ar [d] _-{2\Rep(\Cr,H)}
	\ar@{} [dr] |-{\sm =}
& 2\Rep(\Dr,\Er)
	\ar [d] ^-{2\Rep(\Dr,H)}
\\
2\Rep(\Cr,\Er')
	\ar [r] _-{\Phi_{\Er'}}
& 2\Rep(\Dr,\Er')
}
\quad
\xymatrix @!C @C=3em {
2\Rep(\Dr,\Er)
	\ar [r] ^-{\Psi_{\Er}}
	\ar [d] _-{2\Rep(\Dr,H)}
	\ar@{} [dr] |-{\sm =}
& 2\Rep(\Cr,\Er)
	\ar [d] ^-{2\Rep(\Cr,H)}
\\
2\Rep(\Dr,\Er')
	\ar [r] _-{\Psi_{\Er'}}
& 2\Rep(\Cr,\Er').
}
\]
We define the pseudofunctors $F:\Cr\fl\Dr$ and $G:\Dr\fl\Cr$ as follows:
\[
F \:=\: \Psi_{\Dr}(1_{\Dr})
\qquad\text{and}\qquad
G \:=\: \Phi_{\Cr}(1_{\Cr}).
\]
We consider the naturality condition on $\Phi$ with $\Er=\Cr$, $\Er'=\Dr$ and $H=F$. This gives an equality
\[
F\circ \Phi_{\Cr}(K) \:=\: \Phi_{\Dr}(F\circ K)
\]
for every pseudofunctor $K:\Cr\fl\Cr$. Thus, in the special case $K=1_{\Cr}$, we get 
\[
FG \:=\: \Phi_{\Dr}(F) \:=\: \Phi_{\Dr}\circ\Psi_{\Dr}(1_{\Dr}) \:\simeq\: 1_{\Dr}.
\]
In a symmetric way, the naturality condition on $\Psi$ gives $GF \simeq 1_{\Cr}$, thus concluding the proof.
\end{proof}

\noindent
A combination of Proposition~\ref{PropEquivCofib} and of Lemma~\ref{LemEquivWE} gives the following result.

\begin{proposition}
\label{Proposition:ActionCofibrantApproximation}
Let $\Cr$ and $\tilde{\Cr}$ be $2$-categories, with $\tilde{\Cr}$ cofibrant. The following assertions are equivalent:
\begin{enumerate}[\bf i)]
\item the $2$-category $\tilde{\Cr}$ is a cofibrant approximation of $\Cr$;
\item for every $2$-category $\Dr$, there is an equivalence of categories
\[
2\Rep(\Cr,\Dr) \:\approx\: 2\Cat(\tilde{\Cr},\Dr)
\]
that is natural in $\Dr$.
\end{enumerate}
\end{proposition}

\noindent Finally, an application of Theorem~\ref{Theorem:CoherentPresentationsCofibrantApproximations} concludes the proof of Theorem~\ref{Theorem:CoherentAction}. In the particular case of Artin monoids, we thus get Deligne's Theorem~1.5 of~\cite{Deligne97} for any Artin monoid as a consequence of Theorem~\ref{Theorem:GarsideCoherentPresentation}. Moreover, Theorem~\ref{Theorem:ArtinCoherentPresentation} gives a similar result in terms of Artin's coherent presentation, formalising the paragraph~1.3 of~\cite{Deligne97} on the actions of $\B^+_4$.

\begin{small}
\renewcommand{\refname}{\Large\textsc{References}}
\bibliographystyle{amsplain}
\bibliography{bibliographie}
\end{small}

\vfill

\noindent \textsc{Stéphane Gaussent} \\
\textsf{stephane.gaussent@univ-st-etienne.fr} \\
Université de Lyon, Institut Camille Jordan, CNRS UMR 5208 \\
Université Jean Monnet, 42023 Saint-Étienne Cedex 2, France.

\bigskip
\noindent \textsc{Yves Guiraud} \\
\textsf{yves.guiraud@pps.univ-paris-diderot.fr} \\
INRIA, Laboratoire PPS, CNRS UMR 7126 \\
Université Paris 7, Case 7014, 75205 Paris Cedex 13, France.

\bigskip
\noindent \textsc{Philippe Malbos} \\
\textsf{malbos@math.univ-lyon1.fr} \\
Université de Lyon, Institut Camille Jordan, CNRS UMR 5208 \\
Université Claude Bernard Lyon 1, 43 boulevard du 11 novembre 1918, 69622 Villeurbanne Cedex, France.

\end{document}

%% file: macros.tex
\titleformat{\section}[block]{\scshape\filcenter\Large}{\thesection.}{.5em}{}
\titleformat{\subsection}[block]{\bfseries\filcenter\large}{\thesubsection.}{.5em}{\medskip}
\titleformat{\subsubsection}[runin]{\bfseries}{\thesubsubsection.}{.5em}{}[.]

\titlespacing{\subsubsection}{0pt}{10pt}{.5em}

\newtheoremstyle{ntheorem}%
	{\topsep}{\topsep}{\itshape}{0pt}{\bfseries}{.}{.5em}%
	{\thmnumber{#2.\hspace{.5em}}\thmname{#1}\thmnote{ (#3)}}
	
\newtheoremstyle{ndefinition}%
	{\topsep}{\topsep}{\normalfont}{0pt}{\bfseries}{.}{.5em}%
	{\thmnumber{#2.\hspace{.5em}}\thmname{#1}\thmnote{ (#3)}}
	
\newtheoremstyle{nremark}%
	{\topsep}{\topsep}{\normalfont}{0pt}{\itshape}{.}{.5em}%
	{\thmnumber{#2.\hspace{.5em}}\thmname{#1}\thmnote{ (#3)}}

\theoremstyle{ntheorem}
  	\newtheorem{theorem}[subsubsection]{Theorem}
  	\newtheorem{proposition}[subsubsection]{Proposition}
	\newtheorem{lemma}[subsubsection]{Lemma}

\theoremstyle{ndefinition}

	\newtheorem{example}[subsubsection]{Example}

\theoremstyle{nremark}
	\newtheorem{remark}[subsubsection]{Remark}

\fancypagestyle{plain}{
\fancyhf{}\fancyfoot[LC,RC]{}
\fancyfoot[LE,RO]{$\thepage$}

}

\newcommand{\pdf}[1]{\texorpdfstring{$#1$}{1}}

\renewcommand{\objectstyle}{\displaystyle}
\renewcommand{\labelstyle}{\displaystyle}
\UseTips
\SelectTips{eu}{11}

\newdir{ >}{{}*!/-10pt/@{>}}
\newdir{ -}{{}*!/-10pt/@{}}
\newdir{> }{{}*!/+10pt/@{>}}

\makeatletter

\xyletcsnamecsname@{dir4{}}{dir{}}				
\xydefcsname@{dir4{-}}{\line@ \quadruple@\xydashh@}
\xydefcsname@{dir4{.}}{\point@ \quadruple@\xydashh@}
\xydefcsname@{dir4{~}}{\squiggle@ \quadruple@\xybsqlh@}
\xydefcsname@{dir4{>}}{\Tttip@}
\xydefcsname@{dir4{<}}{\reverseDirection@\Tttip@}

\xydef@\quadruple@#1{%
	\edef\Drop@@{%
		\dimen@=#1\relax
		\dimen@=.5\dimen@
		\A@=-\sinDirection\dimen@
		\B@=\cosDirection\dimen@
		\setboxz@h{%
			\setbox2=\hbox{\kern3\A@\raise3\B@\copy\z@}%
			\dp2=\z@ \ht2=\z@ \wd2=\z@ \box2
			\setbox2=\hbox{\kern\A@\raise\B@\copy\z@}%
			\dp2=\z@ \ht2=\z@ \wd2=\z@ \box2
			\setbox2=\hbox{\kern-\A@\raise-\B@\copy\z@}%
			\dp2=\z@ \ht2=\z@ \wd2=\z@ \box2
			\setbox2=\hbox{\kern-3\A@\raise-3\B@ \noexpand\boxz@}%
			\dp2=\z@ \ht2=\z@ \wd2=\z@ \box2
		}%
		\ht\z@=\z@ \dp\z@=\z@ \wd\z@=\z@ \noexpand\styledboxz@
	}%
}

\xydef@\Tttip@{\kern2pt \vrule height2pt depth2pt width\z@
	\Tttip@@ \kern2pt \egroup
	\U@c=0pt \D@c=0pt \L@c=0pt \R@c=0pt \Edge@c={\circleEdge}%
	\def\Leftness@{.5}\def\Upness@{.5}%
	\def\Drop@@{\styledboxz@}\def\Connect@@{\straight@{\dottedSpread@\jot}}}
	
\xydef@\Tttip@@{%
	\dimen@=.25\dimen@
 	\B@=\cosDirection\dimen@
	\setboxz@h\bgroup\reverseDirection@\line@ \wdz@=\z@ \ht\z@=\z@ \dp\z@=\z@
	{\vDirection@(1,-1)\xydashl@ \xyatipfont\char\DirectionChar}%
	{\vDirection@(1,+1)\xydashl@ \xybtipfont\char\DirectionChar}%
}

\xydef@\ar@form{
	\ifx \space@\next \expandafter\DN@\space{\xyFN@\ar@form}%
	\else\ifx ^\next \DN@ ^{\xyFN@\ar@style}\edef\arvariant@@{\string^}%
	\else\ifx _\next \DN@ _{\xyFN@\ar@style}\edef\arvariant@@{\string_}%
	\else\ifx 0\next \DN@ 0{\xyFN@\ar@style}\def\arvariant@@{0}%
	\else\ifx 1\next \DN@ 1{\xyFN@\ar@style}\def\arvariant@@{1}%
	\else\ifx 2\next \DN@ 2{\xyFN@\ar@style}\def\arvariant@@{2}%
	\else\ifx 3\next \DN@ 3{\xyFN@\ar@style}\def\arvariant@@{3}%
	\else\ifx 4\next \DN@ 4{\xyFN@\ar@style}\def\arvariant@@{4}%
	\else\ifx \bgroup\next \let\next@=\ar@style
	\else\ifx [\next \DN@[##1]{\ar@modifiers{[##1]}}
	\else\ifx *\next \DN@ *{\ar@modifiers}%
	\else\addLT@\ifx\next \let\next@=\ar@slide
	\else\ifx /\next \let\next@=\ar@curveslash
	\else\ifx (\next \let\next@=\ar@curveinout 
	\else\addRQ@\ifx\next \addRQ@\DN@{\ar@curve@}%
	\else\addLQ@\ifx\next \addLQ@\DN@{\xyFN@\ar@curve}%
	\else\addDASH@\ifx\next \addDASH@\DN@{\defarstem@-\xyFN@\ar@}%
	\else\addEQ@\ifx\next \addEQ@\DN@{\def\arvariant@@{2}\defarstem@-\xyFN@\ar@}%
	\else\addDOT@\ifx\next \addDOT@\DN@{\defarstem@.\xyFN@\ar@}%
	\else\ifx :\next \DN@:{\def\arvariant@@{2}\defarstem@.\xyFN@\ar@}%
	\else\ifx ~\next \DN@~{\defarstem@~\xyFN@\ar@}%
	\else\ifx !\next \DN@!{\dasharstem@\xyFN@\ar@}%
	\else\ifx ?\next \DN@?{\ar@upsidedown\xyFN@\ar@}%
	\else \let\next@=\ar@error
	\fi\fi\fi\fi\fi\fi\fi\fi\fi\fi\fi\fi\fi\fi\fi\fi\fi\fi\fi\fi\fi\fi\fi \next@}

\makeatother

\newcommand{\fl}{\to}
\newcommand{\dfl}{\Rightarrow}
\newcommand{\tfl}{\Rrightarrow}
\newcommand{\qfl}{\xymatrix@1@C=10pt{\ar@4 [r] &}}
\newcommand{\ifl}{\rightarrowtail}
\newcommand{\pfl}{\twoheadrightarrow}
\newcommand{\ofl}[1]{\xymatrix @C=1.5em {\strut \ar [r] ^-{#1} & \strut}}
\newcommand{\odfl}[1]{\xymatrix @C=1.5em {\strut \ar@2 [r] ^-{#1} & \strut}}

\newcommand{\ens}[1]{\left\{ #1 \right\}}
\renewcommand{\angle}[1]{\langle #1 \rangle}

\newcommand{\cl}[1]{\overline{#1}}
\newcommand{\rep}[1]{\widehat{#1}}
\renewcommand{\tilde}[1]{\widetilde{#1}}
\newcommand{\tck}[1]{#1^{\top}}

\newcommand{\sm}{\scriptstyle}

\newcommand{\ie}{\emph{i.e.}}
\renewcommand{\phi}{\varphi}
\renewcommand{\epsilon}{\varepsilon}
\newcommand{\dr}{\partial}
\newcommand{\Nb}{\mathbb{N}}
\newcommand{\Cr}{\EuScript{C}}
\newcommand{\Dr}{\EuScript{D}}
\newcommand{\Er}{\EuScript{E}}
\newcommand{\Rr}{\EuScript{R}}
\newcommand{\Sr}{\EuScript{S}}
\newcommand{\B}{\mathbf{B}}
\newcommand{\C}{\mathbf{C}}
\newcommand{\D}{\mathbf{D}}

\newcommand{\M}{\mathbf{M}}
\renewcommand{\S}{\mathbf{S}}
\newcommand{\W}{\mathbf{W}}

\def\catego#1{{\bf{\sf #1}}}
\newcommand{\Cat}{\catego{Cat}}
\newcommand{\Act}{\catego{Act}}

\newcommand{\Rep}{\catego{Rep}}
\DeclareMathOperator{\Art}{Art}
\DeclareMathOperator{\Gar}{Gar}

\DeclareMathOperator{\Std}{Std}
\DeclareMathOperator{\Inv}{Inv}

\newcommand{\ifthen}[2]{\ifthenelse{#1}{#2}{}}

\newcommand{\typedeuxbase}[3]{
\!\begin{tikzpicture}[
	scale = 0.5, every node/.style = { inner sep = 0mm, outer sep = 0mm },
	onecell/.style = { text height = 1.5ex, text depth = 0ex },
	]
	
\node [onecell] (u) at (0,0) {#1} ;
\node [onecell] (v) at (1,0) {#2} ;

\ifthen{\equal{#3}{1}}
	{
		\node (uv) at ($ (u.north) + (0.5,0.2) $) {} ;
		\draw ($ (u.north) + (0.1,-0.1) $) -- (uv.south) -- ($ (v.north) + (-0.1,-0.1) $) ;
	}
\ifthen{\equal{#3}{0}}
	{
		\node (uv) at ($ (u.north) + (0.5,0) $) {$\scriptstyle\times$} ;
	}

\end{tikzpicture}\!
}

\newcommand{\typedeux}[1]{\typedeuxbase{$u$}{$v$}{#1}}

\newcommand{\typetroisbase}[4]{
\!\begin{tikzpicture}[
	scale = 0.5, every node/.style = { inner sep = 0mm, outer sep = 0mm },
	onecell/.style = { text height = 1.5ex, text depth = 0ex },
	]

\node [onecell] (u) at (0,0) {#1} ;
\node [onecell] (v) at (1,0) {#2} ;
\node [onecell] (w) at (2,0) {#3} ;

\node (uv) at ($ (u.north) + (0.5,0.2) $) {} ;
	\draw ($ (u.north) + (0.1,-0.1) $) -- (uv.south) -- ($ (v.north) + (-0.1,-0.1) $) ;
\node (vw) at ($ (v.north) + (0.5,0.2) $) {} ;
	\draw ($ (v.north) + (0.1,-0.1) $) -- (vw.south) -- ($ (w.north) + (-0.1,-0.1) $) ;

\ifthen{\equal{#4}{1}}
	{
		\node (uvw) at ($ (uv.north) + (0.5,0.3) $) {} ;
		\draw ($ (uv.north) + (0.1,0) $) -- (uvw.south) -- ($ (vw.north) + (-0.1,0) $) ;
	}
\ifthen{\equal{#4}{0}}
	{\node (uvw) at ($ (uv.north) + (0.5,0) $) {$\scriptstyle\times$} ;}

\end{tikzpicture}\!
}

\newcommand{\typetrois}[1]{\typetroisbase{$u$}{$v$}{$w$}{#1}}

\newcommand{\typequatrebase}[7]{
\!\begin{tikzpicture}[
	scale = 0.5, every node/.style = { inner sep = 0mm, outer sep = 0mm },
	onecell/.style = { text height = 1.5ex, text depth = 0ex },
	]

\node [onecell] (u) at (0,0) {#1} ;
\node [onecell] (v) at (1,0) {#2} ;
\node [onecell] (w) at (2,0) {#3} ;
\node [onecell] (x) at (3,0) {#4} ;

\node (uv) at ($ (u.north) + (0.5,0.2) $) {} ;
	\draw ($ (u.north) + (0.1,-0.1) $) -- (uv.south) -- ($ (v.north) + (-0.1,-0.1) $) ;
\node (vw) at ($ (v.north) + (0.5,0.2) $) {} ;
	\draw ($ (v.north) + (0.1,-0.1) $) -- (vw.south) -- ($ (w.north) + (-0.1,-0.1) $) ;
\node (wx) at ($ (w.north) + (0.5,0.2) $) {} ;
	\draw ($ (w.north) + (0.1,-0.1) $) -- (wx.south) -- ($ (x.north) + (-0.1,-0.1) $) ;

\ifthen{\equal{#5}{1}}
	{
		\node (uvw) at ($ (uv.north) + (0.5,0.3) $) {} ;
		\draw ($ (uv.north) + (0.1,0) $) -- (uvw.south) -- ($ (vw.north) + (-0.1,0) $) ;
	}
\ifthen{\equal{#5}{0}}
	{\node (uvw) at ($ (uv.north) + (0.5,0) $) {$\scriptstyle\times$} ;}
	
\ifthen{\equal{#6}{1}}
	{
		\node (vwx) at ($ (vw.north) + (0.5,0.3) $) {} ;
		\draw ($ (vw.north) + (0.1,0) $) -- (vwx.south) -- ($ (wx.north) + (-0.1,0) $) ;
	}
\ifthen{\equal{#6}{0}}
	{\node (vwx) at ($ (vw.north) + (0.5,0) $) {$\scriptstyle\times$} ;}

\ifthen{\equal{#5}{1}}
	{\ifthen{\equal{#6}{1}}
		{
			\ifthen{\equal{#7}{1}}
				{
					\node (uvwx) at ($ (uvw.north) + (0.5,0.3) $) {} ;
					\draw ($ (uvw.north) + (0.1,0) $) -- (uvwx.south) -- ($ (vwx.north) + (-0.1,0) $) ;
				}
			\ifthen{\equal{#7}{0}}
				{\node (uvwx) at ($ (uvw.north) + (0.5,0) $) {$\scriptstyle\times$} ;}
		}
	}
\end{tikzpicture}\!
}

\newcommand{\typequatre}[3]{\typequatrebase{$u$}{$v$}{$w$}{$x$}{#1}{#2}{#3}}

\newcommand{\typecinq}[6]{
\!\begin{tikzpicture}[
	scale = 0.5, every node/.style = { inner sep = 0mm, outer sep = 0mm },
	onecell/.style = { text height = 1.5ex, text depth = 0ex },
	]
	
\node [onecell] (u) at (0,0) {$u$} ;
\node [onecell] (v) at (1,0) {$v$} ;
\node [onecell] (w) at (2,0) {$w$} ;
\node [onecell] (x) at (3,0) {$x$} ;
\node [onecell] (y) at (4,0) {$y$} ;

\node (uv) at ($ (u.north) + (0.5,0.2) $) {} ;
	\draw ($ (u.north) + (0.1,-0.1) $) -- (uv.south) -- ($ (v.north) + (-0.1,-0.1) $) ;
\node (vw) at ($ (v.north) + (0.5,0.2) $) {} ;
	\draw ($ (v.north) + (0.1,-0.1) $) -- (vw.south) -- ($ (w.north) + (-0.1,-0.1) $) ;
\node (wx) at ($ (w.north) + (0.5,0.2) $) {} ;
	\draw ($ (w.north) + (0.1,-0.1) $) -- (wx.south) -- ($ (x.north) + (-0.1,-0.1) $) ;
\node (xy) at ($ (x.north) + (0.5,0.2) $) {} ;
	\draw ($ (x.north) + (0.1,-0.1) $) -- (xy.south) -- ($ (y.north) + (-0.1,-0.1) $) ;

\ifthen{\equal{#1}{1}}
	{
		\node (uvw) at ($ (uv.north) + (0.5,0.3) $) {} ;
		\draw ($ (uv.north) + (0.1,0) $) -- (uvw.south) -- ($ (vw.north) + (-0.1,0) $) ;
	}
\ifthen{\equal{#1}{0}}
	{\node (uvw) at ($ (uv.north) + (0.5,0) $) {$\scriptstyle\times$} ;}
	
\ifthen{\equal{#2}{1}}
	{
		\node (vwx) at ($ (vw.north) + (0.5,0.3) $) {} ;
		\draw ($ (vw.north) + (0.1,0) $) -- (vwx.south) -- ($ (wx.north) + (-0.1,0) $) ;
	}
\ifthen{\equal{#2}{0}}
	{\node (vwx) at ($ (vw.north) + (0.5,0) $) {$\scriptstyle\times$} ;}

\ifthen{\equal{#3}{1}}
	{
		\node (wxy) at ($ (wx.north) + (0.5,0.3) $) {} ;
		\draw ($ (wx.north) + (0.1,0) $) -- (wxy.south) -- ($ (xy.north) + (-0.1,0) $) ;
	}
\ifthen{\equal{#3}{0}}
	{\node (wxy) at ($ (wx.north) + (0.5,0) $) {$\scriptstyle\times$} ;}
	
\ifthen{\equal{#1}{1}}
	{\ifthen{\equal{#2}{1}}
		{
			\ifthen{\equal{#4}{1}}
				{
					\node (uvwx) at ($ (uvw.north) + (0.5,0.3) $) {} ;
					\draw ($ (uvw.north) + (0.1,0) $) -- (uvwx.south) -- ($ (vwx.north) + (-0.1,0) $) ;
				}
			\ifthen{\equal{#4}{0}}
				{\node (uvwx) at ($ (uvw.north) + (0.5,0) $) {$\scriptstyle\times$} ;}
		}
	}
	
\ifthen{\equal{#2}{1}}
	{\ifthen{\equal{#3}{1}}
		{
			\ifthen{\equal{#5}{1}}
				{
					\node (vwxy) at ($ (vwx.north) + (0.5,0.3) $) {} ;
					\draw ($ (vwx.north) + (0.1,0) $) -- (vwxy.south) -- ($ (wxy.north) + (-0.1,0) $) ;
				}
			\ifthen{\equal{#5}{0}}
				{\node (vwxy) at ($ (vwx.north) + (0.5,0) $) {$\scriptstyle\times$} ;}
		}
	}
	
\ifthen{\equal{#1}{1}}
	{\ifthen{\equal{#2}{1}}
		{\ifthen{\equal{#3}{1}}
			{\ifthen{\equal{#4}{1}}
				{\ifthen{\equal{#5}{1}}
					{
						\ifthen{\equal{#6}{1}}
							{
								\node (uvwxy) at ($ (uvwx.north) + (0.5,0.3) $) {} ;
								\draw ($ (uvwx.north) + (0.1,0) $) -- (uvwxy.south) -- ($ (vwxy.north) + (-0.1,0) $) ;
							}
						\ifthen{\equal{#6}{0}}
								{\node (uvwxy) at ($ (uvwx.north) + (0.5,0) $) {$\scriptstyle\times$} ;}
					}
				}
			}
		}
	}					
\end{tikzpicture}\!
}